\documentclass[12pt,reqno]{amsart}
\usepackage[margin=1.25in]{geometry}

\usepackage[utf8]{inputenc}
\usepackage[T1]{fontenc}
\usepackage{lmodern}
\usepackage{times}
\usepackage{amsmath, amssymb, amsthm}
\usepackage{mathtools}
\usepackage{xcolor}

\newcommand{\VST}{\mathrm{VST}}
\usepackage{bm}
\usepackage[shortlabels]{enumitem}
\usepackage{hyperref}

\hypersetup{
	colorlinks,
	linkcolor={red!50!black},
	citecolor={blue},
	urlcolor={blue!80!black}
}
\usepackage[round]{natbib}
\usepackage[noabbrev,capitalize]{cleveref}
\usepackage{setspace}
\usepackage{microtype}
\newtheorem{thm}{Theorem}[section]
\newtheorem{cor}{Corollary}[section]
\newtheorem{prop}{Proposition}[section]
\newtheorem{lemma}{Lemma}[section]
\theoremstyle{definition}
\newtheorem{defi}{Definition}[section]
\theoremstyle{plain}
\newtheorem{assume}{Assumption}[section]

\theoremstyle{remark}
\newtheorem{remark}{Remark}[section]
\theoremstyle{plain}

\numberwithin{equation}{section}

\AddToHook{cmd/appendix/after}{%
  \setcounter{figure}{0}%
}

\crefname{assume}{Assumption}{Assumptions}
\crefname{remark}{Remark}{Remarks}
\crefname{thm}{Theorem}{Theorems}
\crefname{cor}{Corollary}{Corollaries}
\crefname{prop}{Proposition}{Propositions}
\crefname{lemma}{Lemma}{Lemmas}
\crefname{defi}{Definition}{Definitions}
\crefname{condi}{Condition}{Conditions}
\crefname{ex}{Example}{Examples}
\crefname{property}{Property}{Properties}
\crefformat{aidisclosure}{#2AI Disclosure#3}
\Crefformat{aidisclosure}{#2AI Disclosure#3}

\DeclareMathOperator{\polylog}{polylog}
\DeclareMathOperator{\clip}{clip}
\newcommand{\data}{\mathrm{data}}
\newcommand{\vwh}{\mathrm{vwh}}
\newcommand{\Hel}{\mathrm{Hel}}
\newcommand{\Helsq}{\Hel^2}
\providecommand{\Helbarstar}[2]{%
  \overline{\Hel}^{\,2}\!\left(#1,#2\right)}

\newcommand{\pihat}{\hat{\pi}_n}
\newcommand{\pistar}{\pi_\star}
\newcommand{\cP}{\mathcal{P}}
\newcommand{\cN}{\mathcal{N}}
\newcommand{\cF}{\mathcal{F}}
\newcommand{\R}{\mathbb{R}}
\newcommand{\E}{\mathbb{E}}
\newcommand{\Prob}{\mathbb{P}}
\newcommand{\Var}{\mathrm{Var}}

\newcommand{\iid}{\overset{\text{i.i.d.}}{\sim}}
\newcommand{\avgn}{\frac{1}{n}\sum_{i=1}^{n}}
\newcommand{\rbar}[1]{\bar{r}_{#1}}

\newcommand{\KL}{\mathrm{KL}}

\newcommand{\Wass}{W_2}
\newcommand{\Wtwo}{W_2}

\newcommand{\TV}{\mathrm{TV}}
\newcommand{\rmis}{R^{\star}}

\DeclareMathOperator{\tv}{d_{\mathrm{TV}}}

\DeclareMathOperator{\PP}{\mathbb{P}}
\DeclareMathOperator{\EE}{\mathbb{E}}

\DeclareMathOperator{\pst}{\pi_{\star}}

\newcommand{\fmix}[2]{f_{#1,#2}}

\usepackage{macros}

\makeatletter
\patchcmd{\@settitle}{\uppercasenonmath\@title}{\large}{}{}
\patchcmd{\@setauthors}{\MakeUppercase}{\vspace{-1em}\normalsize}{}{}
\patchcmd{\section}{\scshape}{\bfseries}{}{}
\renewcommand{\paragraph}{\@startsection{paragraph}{4}%
  \z@\z@{-\fontdimen2\font}%
  {\normalfont\bfseries}}
\makeatother

\title{Normal approximations in nonparametric empirical Bayes}
\author{Jiafeng Chen \and Nabarun Deb \and Nikolaos Ignatiadis}
\date{\footnotesize \today. Chen: Department of Economics, Stanford University,
 jiafeng@stanford.edu; Deb: Chicago Booth School of Business,
 nabarun.deb@chicagobooth.edu; Ignatiadis: Department of Statistics and Data Science
 Institute, University of Chicago, ignat@uchicago.edu. We thank Isaiah Andrews, Guido
 Imbens, Sid Kankanala, Patrick Kline, Lihua Lei, Tengyuan Liang, David Ritzwoller, Davide Viviano, and
 seminar participants at Brown Economics and the International Seminar in Selective
 Inference (ISSI) for insightful discussions. We thank Qiyang Han for discussing this
 paper at ISSI.
N.I. gratefully acknowledges support from NSF (DMS 2443410).
This paper uses the Veneto Workers Histories
dataset developed by the Economics Department in
Universit\`{a} Ca' Foscari Venezia
 under the supervision of
Giuseppe Tattara.
An \cref{sec:ai_disclosure} can be found in the appendix.}

\begin{document}
\maketitle
\allowdisplaybreaks
\begin{abstract}

Empirical Bayes (EB) analyses in economics routinely model noisy measurements of latent parameters as
normal. This working model is often justified by an informal appeal to the central limit theorem (CLT), since the noisy measurements are often aggregates of micro-data. This
paper studies this ``CLT heuristic'' both theoretically and empirically. We show that the
denoising regret of nonparametric EB  methods is dominated by the rate
attained under exact normality plus a term reflecting the quality of the CLT
approximation. We prove this result for two leading regimes with potentially non-normal micro-data --- (a)
in the \emph{independent measurements} regime, each noisy measurement is a sample mean of i.i.d. micro-data, and (b) in the \emph{linear measurements} regime, each measurement is a known linear combination of independent micro-data, allowing the measurements to be correlated with each other. Motivated by our theoretical results, we also study sensitivity of EB methods to normal approximation in two empirical contexts. Overall, our results provide theoretical justification and empirical support for the CLT heuristic for nonparametric EB.

\end{abstract}

\pagestyle{plain}

\newpage

\section{Introduction}

Empirical Bayes methods \citep{robbins1956empirical, efron2010largescale} are
increasingly popular in applied work in economics, statistics, and adjacent
fields.\footnote{See~\citet{walters2024empirical,koenker2026empirical, hull2026shrink} for surveys of EB methods in economics.} The typical EB
setup proceeds as follows: An analyst observes noisy estimates $X_1,\ldots, X_n$, alongside
attached standard errors $\hat\sigma_1, \ldots, \hat\sigma_n$, for unobserved
parameters $\theta_1,\ldots,\theta_n$. The analyst hopes to use these estimates to
collectively improve decisions about the unobserved parameters. To list some examples, the parameters
$\theta_1,\ldots,\theta_n$ can represent unknown teacher quality \citep{gilraine2020new},
latent place effects \citep{bergman2024creating}, firm effects \citep{abowd1999high}, gene
expression differences \citep{stephens2017false}, treatment effects from A/B tests \citep
{azevedo2019empirical,abadie2023estimating}, value add of nursing homes \citep{einav2025producing},
biases of various estimators \citep{chetty2026experimentally,hullmetaanalysis}, etc.

A nearly universal modeling convention, at least in EB applications in econometrics and
economics, treats the noisy estimates as exactly normally distributed \[X_i \mid \theta_i,
\hat\sigma_i \sim \cN (\theta_i, \hat\sigma_i^2) \numberthis \label{eq:working}\] and
treats the estimated standard error $\hat\sigma_i$ as if it were the true standard
deviation of $X_i$. This convention is often justified by a heuristic appeal to the
central limit theorem (CLT), applied to whatever \emph{micro-data} produced $X_i$.
Typically, the $X_i$'s are estimates from certain well-behaved statistical procedures (e.g.,
regression coefficients) over the underlying micro-data (e.g., student records within a
teacher). Under standard conditions, $(X_i-\theta_i)/\hat\sigma_i\dto \Norm(0,1)$, as the
size of the micro-data diverges. An analyst operating under \eqref{eq:working} then
proceeds as if this CLT approximation is exact---essentially replacing ``$\dto$'' with
``$\sim$.'' This ignores the fact that the $X_i$'s are only approximately normal and that
$\hat\sigma_i^2$'s are noisy estimates of $\var(X_i \mid \theta_i)$.

This appeal to the CLT is pervasive enough to merit a catalog. In labor economics, the
handbook chapter by \citet{walters2024empirical} states this normality assumption as
equation (1) and treats it as a starting point for EB analyses.  \citet
{kline2022systemic}, who study discriminatory hiring practices, invoke an asymptotic
approximation for estimates at the firm level ``with a growing number of jobs sampled for
each firm.''
\citet{angrist2017leveraging} write, on estimation errors of school quality, ``subject to the usual asymptotic approximations, these errors are
normally distributed with a known covariance structure.'' In e-commerce and A/B testing,
\citet{abadie2023estimating} model the treatment effect estimates as normal, ``motivated by approximate Gaussianity of the large sample distributions of many commonly used
estimators of treatment effects.'' Similar appeals are made in
\citet{azevedo2019empirical,azevedo2020testing} and \citet{wernerfelt2025estimating}. In
public economics, \citet{moon2026optimal} studies the marginal value of public funds
\citep{hendren2020unified} and models the benefits and cost estimates of policies as
 normal, ``motivated by the central limit theorem.'' Finally, the methodological literature on empirical Bayes often starts with normality
\citep{zhang1997empirical,    jiang2009general, efron2011tweedie, Jiang2020,
 ignatiadis2022confidence,deb2022twocomponent, gu2023invidious, Soloff2024,
 hoff2025selective,chen2026empiricalbayesshrinkagemostly, hanzhang,liang2025II,kim2026empirical,you2026using, chang2026personalized} with some explicit appeals to this heuristic (e.g., ``a CLT applied
 to the underlying micro-data'' in \citet{chen2024empirical}). See \cref
 {sec:handwave} for more examples.

With rare exceptions, such statements are not rigorously justified.\footnote{One could
avoid assuming normality with simpler (e.g., linear) decision rules~\citep{kou2017optimal,ignatiadis2019covariatepowered, ghosh2025steins, kwon2026optimal}. In certain simplified settings, the approximation of the central
limit theorem has been explicitly accounted for
\citep{armstrong2022robust,law2023distributional,chen2025compound}.}  A natural question
 is then---as put explicitly by \citet{hirano2023comment}---``In many cases, proceeding
 as if the data are normally distributed may be quite reasonable, but can this appeal to
 approximate normality be put on firmer analytical grounds?''
This paper studies this question both theoretically and empirically. We mainly do so for nonparametric EB methods, in particular
the nonparametric maximum likelihood estimator \citep[NPMLE,][]
{kiefer1956consistency} and related sieve maximum likelihood methods  \citep
{efron2016empirical}. These methods are state-of-the-art for empirical Bayes
estimation and are thus natural objects of study.\footnote{Sieve maximum likelihood is
used in \citet{kline2022systemic} and the NPMLE is used in \citet{gilraine2020new}. The
NPMLE, owing to its excellent theoretical performance, has been studied by a recent
methodology literature \citep[among many others,][]{chen2024empirical,hundrieser2026minimaxoptimalitynpmlegaussian,chen2026sharp,Jiang2020,Soloff2024,fan2023gradient,moon2026optimal,ignatiadis2025empirical,song2026empiricala,adusumilli}.
}

We quantify the effect of approximate normality, in terms of sample sizes of the
underlying micro-data, on the resulting performance of nonparametric EB decisions.
Studying this question is important for the applicability of modern EB
methods, whose theoretical analyses rely heavily on the normality of $X_i$. If
nonparametric EB methods were sensitive to the normal approximation, then their
applicability would be substantially limited.\footnote{Indeed, \citet[e.g., section 4.2]{hull2026shrink} and \citet{kline2025branching} raise such concerns about flexible EB methods and advocate for alternatives that are less efficient but appear more robust.}
Conversely, robustness to the normal approximation would allow researchers to avoid assumptions that the micro-data itself is normally distributed---an assumption made by, among others,
\citet{gu2017empirical,liu2020forecasting, banerjee2020nonparametric,
ignatiadis2025empirical, ho2025large, gaillac2025predicting, lee2026parametric,
song2026empiricala}.

For practitioners, our analysis identifies appropriate notions of micro-data sample sizes to measure the quality of the normal approximation---in particular  for settings where $X_i$ are correlated
\citep{hanzhang}. Our theory therefore connects the credibility of normality-based EB
analyses to features of the underlying data that are, at least in principle, available to
 practitioners and can be reported. When these sample size diagnostics indicate a poor normal approximation, practitioners should explore options that do not rely on the CLT heuristic: e.g., working with an appropriate context-specific likelihood~\citep{kline2021reasonable,chen2025compoundbin,torres2026tweedie}, robustifying the NPMLE procedure~\citep{koenker2026robustifying}, exploiting access to micro-data~\citep{coey2019improving, deng2021postselection, ignatiadis2023empiricala, kline2025branching}, or using universal denoisers~\citep{liang2025}.

 We show that
nonparametric EB methods have squared regret bounded by $\frac{1} {n} + \frac{1}{J}$ (up
to constants and logarithmic factors), for some quantity $J$ that aggregates micro-data
sample sizes for each $X_i$. Relative to standard guarantees for the NPMLE~\citep
{jiang2009general, Jiang2020,Soloff2024}, the additional $1/J$ term reflects the price of
approximate normality. It vanishes either under exact normality or when $J \to \infty$. We
show such results for two regimes of data-generating processes (DGPs), which are specialized from a general analysis of
nonparametric EB methods under misspecification of the working model \eqref{eq:working}.

The first regime we study is termed \emph{independent measurements}, where $X_i$ are sample means of micro-data points $(X_{ij}: j =1,\ldots,J_i)$. These micro-data points are i.i.d.\ across $j$, within a unit $i$, but their distributions may differ across $i$. This regime captures settings where the units $i$ are disjoint experiments, for instance for applications in meta-analysis \citep[e.g.,][]{hullmetaanalysis,azevedo2020testing}. There, the $\hat\sigma_i^2$'s are assumed to estimate the variance of $X_i$ via an appropriately scaled sample variance. We find that the overall sample size $J$---relevant for our convergence rates---can be chosen as the \emph{harmonic mean} of the individual sample sizes: $1/J = \frac{1}{n} \sum_i 1/J_i$.

The second regime we study, which we call \emph{linear measurements}, generalizes the
first. There, for known vectors $a_i$, we assume $X_i = a_i'Y$ where $Y = (Y_1,\ldots,
Y_m)'$ is a vector of i.n.i.d.\ micro-data points. Importantly, this regime nests regression
models where $Y = B\theta + \xi$ such that $X = (B'B)^{-1}B'Y$ is the OLS estimator for
$\theta$. When $\theta$ represents a vector of fixed effects, e.g., in the AKM model in
labor economics \citep{abowd1999high,kline2020leave,jochmans2019fixed}, such settings are
promising applications of EB methods
\citep[e.g.,][]{cheng2025optimal, kline2025branching}. In this regime, because the $X_i$'s
are correlated through the same underlying data $Y$, it is no longer obvious what notions
of \emph{effective} sample sizes $(n, J)$ are appropriate. Building on recent work by \citet{hanzhang}, we identify corresponding notions of $n$ and
$J$ in this regime.\footnote{We caution that this regime is analytically challenging: As a result,
we provide a relatively limited analysis, e.g., under the assumption that $Y_t-\E Y_t$ is
homoskedastic. To deal with more challenging settings, we state user-friendly, high-level general conditions that lead to empirical Bayes regret consistency (without rates) in Appendix~\ref{sec:general}. The results of Appendix~\ref{sec:general} are inspired by~\citet{zhong2022empirical}
}

Motivated by our theoretical results, we investigate the empirical performance of
nonparametric EB methods in economic contexts on firm discrimination in hiring
\citep[][treated as independent measurements]{kline2022systemic} and on firm wage premia
\citep[][treated as linear measurements]{kline2020leave}. These exercises ask how much the quality of normal approximation matters for nonparametric EB methods at
empirically realistic sample sizes. In simulations calibrated to
\citet{kline2022systemic}, we find that increasing $n$ matters much more than increasing
$J$, suggesting that the normal approximation is adequate and does not bottleneck the
performance of EB methods. On the other hand, in matched employer--employee data used to
estimate firm wage premia \citep{kline2020leave}, the normal approximation is more
concerning: Our notion of effective sample size is only $J\approx 2.53$ in the Veneto
Workers Histories dataset. Calibrated simulations indeed show that improving the normal
approximation meaningfully matters.\footnote{Curiously, in numerous simulations we also find
that \emph{estimated $\sigma_i^2$ improves} the performance of EB methods relative to
known $\sigma_i^2$. We conjecture that this phenomenon is more pronounced if the
micro-data distribution has high excess kurtosis. We leave its theoretical analysis to future work. }
Nevertheless, surprisingly, NPMLE predictions for wage changes of held-out workers---leveraging a recent sample-splitting scheme by \citet{kline2025branching}---appear
competitive with alternative methods that seem ex ante more robust to normality.

Developing our theory requires resolving several conceptual ambiguities. First,
a larger micro-data sample size \(J\) ordinarily makes the estimates \(X_i\)  both (i) more ``nearly'' normal and (ii) more precise. Greater precision reduces the value of EB methods, whose gains depend on nontrivial noise in $X_i$ relative to the heterogeneity in \(\theta_i\).
To isolate the effect of improved normal approximation, we therefore consider DGP sequences in which the normal approximation improves, without making \(X_i\) more precise (akin to the local parametrization in \citet{hirano2016panel} for panel data settings, but allowing for \(n\to\infty\)).

 Second, it is also not obvious what a reasonable notion of \emph{regret} is when the analyst
 proceeds under the normal model \eqref{eq:working} but $X_i$ follows some non-normal
 distribution in finite samples. We consider an analyst who estimates the distribution of \(\theta_i\) and computes the posterior mean \(\hat{\mathbb E}[\theta_i\mid X_i,\hat\sigma_i]\) under their estimated prior and the working model \eqref{eq:working}. To benchmark this decision, we consider a \emph{normal oracle} who knows the
 distribution of $\theta_i$ and the true variance $\sigma_i^2$ and receives \emph{exactly normal data}
 $Z_i \mid \theta_i, \sigma_i^2 \sim \Norm(\theta_i, \sigma_i^2)$. This oracle then
 reports the posterior mean $\E[\theta_i \mid Z_i, \sigma_i]$. Our notion of regret is the performance gap
 between the analyst and oracle at estimating $\theta_i$, in root mean-squared error (RMSE).

Lastly, since EB methods improve estimates by pooling strength across the $n$ noisy
estimates, standard EB results often consider large $n$. It is then not
obvious whether the normal approximation needs to be \emph{uniform} over the units $i$, thus
requiring tools from high-dimensional central limit theory
\citep{chernozhukov2017central}. Our guarantees depend on the \emph{average} accuracy of
each coordinate-wise normal approximation, so a high-dimensional CLT is unnecessary for
our results. For our analysis, the key technical tool is Stein's method for
approximating the marginal distribution of each estimate \citep{stein1972bound,
barbour1986asymptotic}.

This paper is organized as follows. \Cref{sec:indep_measurements} introduces the setup
and presents our main results under the independent measurements regime.
\Cref{sec:linear_measurements} extends the analysis to the linear measurements regime, allowing
for dependence among the estimates. \Cref{sec:main-res} outlines the proof strategy
and develops general results for density estimation, deconvolution, and denoising
under approximate normality. \Cref{sec:empirical} investigates the quality of normal approximation in various empirical settings.

\subsection*{Notation} $\phi(\cdot ; \sigma)$ denotes the density of $\Norm(0, \sigma^2)$. $\mathcal P(S)$ denotes the set of all probability distributions supported on $S \subset \R^d$. The symbol $C$ denotes universal constants that do not depend on $n$ or $J$; different occurrences of $C$ may denote different constants.
For two sequences $(A_n, B_n)$, we write $A_n \lesssim B_n$ if there exists a constant $C > 0$ such that $A_n \le C B_n$ for all sufficiently large $n$.
$W_p$ denotes the $p$-Wasserstein distance \citep[Equation 12,][]{villani2021topics} between two distributions. For two Lebesgue densities $p, q$,
\begin{equation}
\Hel(p,q) := \pr{\frac{1}{2} \int \pr{\sqrt{p(x)} - \sqrt{q(x)}}^2 \,dx}^{1/2}
\label{eq:hellinger}
\end{equation}
denotes the Hellinger distance between $p, q$. The notation $\polylog(n)$ denotes {polylogarithmic} factors bounded by $C \log^D n$ for some constant $D > 0$. We assume $ n \geq 2$ throughout. For numbers $s,a,b \in \R$, $\clip(s,[a,b]) = \min(\max(s, a), b)$.

\section{Setup and main results}
\label{sec:indep_measurements}

We consider a class of data-generating processes that we call the
\emph{independent measurements} setting. For unit $i$ we observe a sample average $X_i$ of $J_i \geq 2$ independent measurements $X_{ij}$ and seek to estimate the expectation $\theta_i$ thereof. Formally:
\begin{equation}
X_i = \frac{1} {J_i} \sum_
{j=1}^{J_i} X_{ij} = \theta_i + \frac{1}{J_i} \sum_{j=1}^{J_i} \varepsilon_{ij} ,\;\,\text{ with }\;\, X_{ij} \mid (\theta_i, \alpha_i,  J_i)  \iid P(\cdot \mid \theta_i, \alpha_i, J_i),\numberthis\label{eq:local}
\end{equation}
where $\varepsilon_{ij} := X_{ij}-\theta_i$ conditional on $(\theta_i,\alpha_i, J_i)$ are mean zero ($\E[\varepsilon_{ij} \mid \theta_i,\alpha_i, J_i] = 0$) and independent across $i$. Here, the ``fixed effect'' $\alpha_i$ is a technical device that permits unrestricted heterogeneity across units, e.g., allowing for different variances and distribution shapes~\citep{okui2020kernel, ignatiadis2023empiricala}. In our analysis we treat both $\alpha_1,\ldots,\alpha_n$ and $J_1,\ldots, J_n$ as fixed or conditioned on. We let $\sigma_i^2 := \var(X_i \mid \theta_i, \alpha_i, J_i)$ denote the variance of $X_i$. Following the EB literature \citep{Jiang2020,Soloff2024}, we assume that the variances $\sigma_i^2$ are bounded above and below by constants.

 \begin{assume}[Uniform variance bounds]\label
  {assump:lowerbound} There exist constants $0 < k < K < \infty$ such that $ k < \min_
  {1 \leq i \leq n} \sigma_i \leq \max_{1 \leq i \leq n} \sigma_i < K \mbox{~~for all
  $n \geq 2$. } $
  \end{assume}
  \noindent The variances $\sigma_i^2$ are unknown but are estimated by the (clipped) sample
   variances
  \[\hat\sigma_i^2
    :=\clip\pr{
    \frac{1}{J_i(J_i-1)}
    \sum_{j=1}^{J_i}(X_{ij}-X_i)^2,
  [k^2/2, 2K^2]}.\]
    We clip $\hat\sigma_i^2$
    to $[k^2/2, 2K^2]$ so that
    \cref{assump:lowerbound} is respected by the estimated
    variances.\footnote{This clipping can be dropped in our results if we assume $\min_i J_i
  > C \log n$, for a sufficiently large $C$, since the probability of it binding becomes
    negligible.}

As is standard in EB analyses, we study a random-effects model for $\theta_i$. Here, we
assume \[\theta_i \mid \sigma_i^2 \iid \pistar \in \mathcal P([-M,M]), \numberthis \label{eq:true_prior} \]
where $\pistar$ is often referred to as the ``true prior.''\footnote{As a reminder, $\alpha_i, J_i$ are treated as fixed. If they are random objects that we condition on, we should interpret \eqref{eq:true_prior} as $\theta_i \mid (\sigma_i^2,J_i,\alpha_i)\sim \pistar$, requiring that $J_i, \alpha_i$ also do not predict $\theta_i$. Our proof applies as-is under the weaker assumption $\theta_i \mid (\sigma_i^2, J_i) \sim \pistar$, allowing $\alpha_i$ to predict $\theta_i$ given $(\sigma_i^2, J_i)$. We expect that the conditioning on $J_i$ can also be dropped if the bounds are averaged over $J_i$.   } In the context of teacher
value-added \citep{gilraine2020new}, this represents a thought experiment where teachers are sampled from some
superpopulation, and $\pistar$ is the population distribution of latent teacher effects.
An important restriction of \eqref{eq:true_prior} is \emph{precision independence}
\citep{chen2024empirical}:  $(\sigma_i^2, J_i)$ are assumed to not predict the distribution of $\theta_i$. This
is imposed for convenience of the theoretical analysis, and we anticipate our results
extend to models that allow for precision dependence.\footnote {For robustness to precision-dependence, the \textsc{close-npmle} method in
\citet{chen2024empirical} removes $\E[\theta \mid \sigma]$ and $\var(\theta \mid \sigma)$
via studentization and then proceeds under precision independence for the studentized
estimates. Extending our theory to this method would then require accounting for the
estimation of these conditional moments. \Cref{sec:empirical} also presents empirical results that involve methods that account for precision dependence.} Also for simplicity, throughout, we assume that
$\pistar$ is compactly supported within some interval $ [-M, M]$, where our subsequent
results depend on $M$.\footnote{Much of the nonparametric EB literature
analyzes compactly supported priors as a leading case
\citep{jiang2009general,chen2026sharp,Jiang2020,Soloff2024,hundrieser2026minimaxoptimalitynpmlegaussian}
and the same techniques extend to sub-Gaussian priors typically via truncation arguments,
incurring additional polylogarithmic factors. We expect the same to be true here
and present the compact-support results for simplicity.}

\subsection{Increasing-variance asymptotics for \cref{assump:lowerbound}} Although
\cref{assump:lowerbound} is common in the EB literature, imposing it under \eqref{eq:local} implicitly specifies
a kind of \emph{increasing-variance} sequence of DGPs. Since the
$\sigma_i^2$'s are bounded above and below, the micro-data variance, $J_i \sigma_i^2$,
must then grow as $J_i$ grows: \[ \sigma_i^2 = \var(X_i \mid \theta_i) =
\frac{1}{J_i}\var(\varepsilon_{ij} \mid \theta_i) \implies \var(\varepsilon_{ij} \mid
\theta_i) = \sigma_i^2 J_i. \] At first blush, having the micro-data variances grow in
$J_i$ seems unusual. We argue that it is an appropriate asymptotic approximation here.
Without such scaling---supposing that $\var (\varepsilon_ {ij} \mid \theta_i) =
\tau_i^2$---we would then have $\theta_i \mid X_i \approx \Norm(X_i, \tau_i^2/J_i)$ under
conditions of the Bernstein-von Mises theorem. This renders empirical Bayes unhelpful in
this limit (since the posterior no longer depends on $\pistar$).\footnote{That said, pooling strength may still yield gains even in this vanishing-$\sigma_i$ regime. Recent work by \citet{chetverikov2026} proposes procedures that achieve such gains under the assumption that $\theta_{1},\ldots,\theta_n$ take a small number of distinct values.}

Asymptotic approximations aim to capture relevant
finite-sample features. For EB with a normal likelihood to be both valid and useful, we need
an asymptotic setup where the normal approximation is adequate, but the prior continues
to provide meaningful information. As a concrete example of this increasing-variance
regime, a technology firm may enroll more users in A/B tests, but run each test for
a shorter duration, creating a scenario where both sample size $(J_i)$ and variability $
(\tau_i^2)$ increase.

We could also view the increasing-variance regime \eqref{eq:local} as a shorthand for
considering a \emph{local} asymptotic experiment \citep[applied to EB settings by, e.g.,][]{sen2000hajek,hansen2016efficient}. In a nutshell, local asymptotics  rescales observations with $\bar J^{1/2} \rateeq J_i^{1/2}$ \citep{cam1986asymptotic, van2000asymptotic, hirano2009asymptotics}, which then dilates the variance by the appropriate amount, as detailed in the next remark. Relative to existing work  on shrinkage of local parameters, we  consider asymptotic regimes where both $n$ and $J$ diverge. The local asymptotics regime is especially useful for analyzing nonlinear settings (e.g., binomial, deferred to \cref{asec:ral}).

\begin{remark}[Local asymptotics]

\label{rmk:local_asymp}
	Let $\tau_i^2 = \var_{P_i}(\varepsilon_{ij} \mid \theta_i)$. Suppose $J_i \rateeq \bar J$ diverge at the same rate. Fixing some $\theta_0$ to localize around, we can consider the local parameters \[
h_i = \bar J^{1/2} (\theta_i - \theta_0),
	\]
	which we assume are i.i.d., $h_i \iid G$. Then we may reparametrize (cf. section 2 in \citet{hirano2016panel}):
\[
\theta_i = \theta_0 + \bar J^{-1/2} h_i \quad X_i^h = \bar J^{1/2}(X_i - \theta_0) = h_i + \frac{1}{J_i} \sum_{j=1}^{J_i} \bar J^{1/2} \varepsilon_{ij},
\]
with $\var (X_i^h \mid \theta_i) = \frac{\bar J}{J_i}\tau_i^2 =: \bar \tau_i ^2$.
Under this reparametrization, we can view $X_i^h$ as our estimates and $\bar\tau_i^2$ as their variances. The targets of these estimates are the local parameters $h_i$. Thus, this setup with $(X_i^h, h_i, G, \bar \tau_i^2, \bar J^{1/2} \varepsilon_{ij})$ is analogous to \eqref{eq:local} with $ (X_i, \theta_i, \pistar, \sigma_i^2, \varepsilon_{ij})$ up to
	relabeling of variables.

Note, however, modeling $h_i \iid G$ posits a sequence of data-generating processes where the true effects $\theta_i$
\emph{concentrate} around $\theta_0$ at rate $\bar J^{-1/2}$. Under such sequences, the heterogeneity in the parameters is on the same order as the statistical uncertainty in $X_i$
\citep
	{yang2012bayesian, reimherr2021prior}.
	This is often a reasonable regime for asymptotic approximation.  Indeed, if the estimates were much more precise
	than the heterogeneity in the underlying effects, we would also not expect EB methods to provide much value anyway---the estimates would shrink very
	little. For instance,  in the context of meta-analyses of experiments,  EB is often used in settings where there is meaningful uncertainty about whether treatment effects are positive; in this case, we can set $\theta_0 := 0$ and model the treatment effects as on the same order as standard errors, $\theta_i = \bar J^{-1/2} h_i$.
\end{remark}

\subsection{NPMLE, density estimation, and regret}

Given estimates $(X_1,\ldots, X_n, \hat\sigma_1, \ldots,
\hat\sigma_n)$, we consider an
analyst who estimates the NPMLE under the working model \eqref{eq:working}.  For $\phi
(x;\sigma)$ the density of $\Norm(0, \sigma^2)$,
\begin{equation}
		\pihat \text{ approximately solves }\max_{\pi \in \mathcal{P}[-M,M]} \; \avgn \log \int
		\phi\!\left(X_i - \theta;\, \hat\sigma_i\right) d\pi(\theta), \label{eq:npmle}
\end{equation}
for $\mathcal{P}[-M,M]$ the set of probability distributions on $[-M, M]$. In
practice, such an estimator may be computed by approximating $[-M, M]$ with a grid
\citep{koenker2014convex} or by approximating $\mathcal P[-M, M]$ with a sieve $\mathcal
 G$
 \citep[e.g.,][]{efron2016empirical}. Our theory allows for \emph{approximate} maximizers
 of \eqref{eq:npmle} and thus accommodates such approximations.
 After computing $\pihat$, the analyst estimates $\theta_i$ through the posterior mean \[
 h_{\pihat}(X_i; \hat\sigma_i) := \hat \E_{\pihat}[\theta_i \mid X_i, \hat\sigma_i] :=
 \frac{ \int \theta \phi(X_i-\theta; \hat \sigma_i) d \hat{\pi}_n(\theta)}{\int \phi
 (X_i-\theta; \hat\sigma_i) d \hat{\pi}_n(\theta)} \numberthis \label{eq:posterior_mean}
 \]
 acting as if \eqref{eq:working} holds and the distribution of $\theta_i$ is $\pihat$.

When \eqref{eq:working} holds, EB theory shows that posterior decisions such as
\eqref{eq:posterior_mean} perform well because $\pihat$ is a good estimator of $\pistar$.
A key notion is the  \emph{average Hellinger
distance} between the normal convolutions induced under $\pihat$ and under $\pistar$.
For $f_{\pi, \sigma} = \pi \star \Norm(0,\sigma^2)$ the normal mixture with mixing
distribution $\pi$, define   \citep{Jiang2020,Soloff2024} \[
\Helbarstar{f_\pi}{f_{\pistar}} := \avgn \Helsq\!\left(f_{\pi,\sigma_i},f_
{\pistar,\sigma_i}\right). \numberthis \label{eq:avg_hellinger}
\]  When \eqref{eq:working} holds, \eqref{eq:avg_hellinger} is a measure of how well $f_{\hat\pi_n}$ estimates
 the marginal density of $X_i$.
Without \eqref{eq:working}, \eqref{eq:avg_hellinger} is nonetheless a smoothed distance between $\pi$ and $\pistar$, useful for bounding the Wasserstein distance $W_2
 (\pihat, \pistar)$
 and other performance guarantees \citep{Nguyen2013,chen2026sharp}.

The power of EB methods for decision-making lies in the fact that, for normally
distributed $X_i$,
\eqref{eq:posterior_mean} is an estimate of the optimal decision rule for estimating
 $\theta_i$ in squared error, namely $h_{\pistar}(X_i; \sigma_i)$. Classical EB
 theory shows that the suboptimality of $h_{\pihat}$ in mean-squared error---called the
 \emph{denoising regret}---vanishes at near-parametric rates \citep{jiang2009general}.
 This notion of regret,
 however, needs revision if the data are no longer normally distributed.

 Our results mainly analyze the following regret notion instead: For $Z_i \mid \theta_i, \sigma_i^2 \sim
 \Norm (\theta_i,
 \sigma_i^2)$ \begin{equation} \label{eq:regret-z}
\mathrm{Regret}_n := \underbrace{\sqrt{\avgn \E\bigl|\theta_i - h_{\pihat}(X_i; \hat\sigma_i)\bigr|^2}}_{\text{Feasible RMSE for non-Gaussian data}}
- \underbrace{\sqrt{\avgn \E\bigl|\theta_i - h_{\pistar}(Z_i;\sigma_i)\bigr|^2}}_{\text{Oracle RMSE for Gaussian data}}.
\end{equation}
The definition \eqref{eq:regret-z} compares the performance of the feasible estimator $h_
{\hat\pi_n}
(X_i;\hat\sigma_i)$ against a \emph{normal oracle}, who uses the optimal estimator $h_
{\pistar}(Z_i;\sigma_i)$ under the correctly specified model $Z_i \sim f_{\pistar,
\sigma_i}$. With exactly normal $X_i$ and $\hat\sigma_i = \sigma_i$, this regret notion takes the same form as \citet[eq. (4.8)]{jiang2009general}.\footnote{Under
 \eqref{eq:working}, \citet{Jiang2020, Soloff2024, chen2024empirical} work with the
  difference in MSE rather than in RMSE as the regret notion; their argument uses an
  orthogonality condition that relies on the well-specification of the EB model.} Our
  general results also include a stronger
  regret notion  (\cref{thm:denoising-strong}), comparing to an oracle that knows the law of $X_i$.

\subsection{Results under independent measurements}

Returning to \eqref{eq:npmle},
let $\pihat \in \mathcal P([-M,M])$ be an estimator where for some constant $C > 0$, \[
	\avgn \log \int
		\phi\!\left(X_i - \theta;\, \hat\sigma_i\right) d\pihat(\theta) \ge \avgn \log \int
		\phi\!\left(X_i - \theta;\, \hat\sigma_i\right) d\pistar(\theta) \underbrace{- \frac{C\log n}{n}}_{\text{optimization error}}
		\numberthis \label{eq:approximate_mle}
\]

For this approximate NPMLE, our main result is the following, where we additionally assume $\varepsilon_{ij}$ are sub-Gaussian.

\begin{thm}\label{thm:main_rates} Under model \eqref{eq:local}, suppose
 $\sigma_1^2,\ldots, \sigma_n^2$ are deterministic and satisfy \cref
 {assump:lowerbound}. Further assume that $\varepsilon_{ij}/\sqrt{J_i}$ conditioned on
 $\theta_i$ are sub-Gaussian with variance proxies bounded by a constant $C < \infty$.
 Then the following bounds hold with $ J^{-1} := \frac{1}{n}\sum_{i=1}^n J_i^{-1}$:
 For
 \eqref{eq:avg_hellinger}, we have
  \begin{equation*}
    \E\bk{
      \Helbarstar{f_{\pihat}}{f_{\pistar}} }
    \lesssim
    \log^4 n \cdot \left(\frac{1}{n} + \frac{1}{J}\right).
  \end{equation*}
  For \eqref{eq:regret-z}, we have
  \begin{equation*}
    \bigl(\mathrm{Regret}_n^+\bigr)^2
    \lesssim
   \log^5 n\cdot
    \left(\frac{1}{n} + \frac{1}{J}\right), \text{ where $(\cdot)^+ = \max(\cdot, 0)$. }
  \end{equation*}
  The omitted constants depend on those in \cref{assump:lowerbound},
  \eqref{eq:approximate_mle}, and the sub-Gaussian variance proxy bound $C$.
\end{thm}

\Cref{thm:main_rates} bounds both squared Hellinger distance and squared regret by $1/n +
1/J$, up to log factors. The $1/n$ term recovers the seminal regret bounds of~\citet
{jiang2009general} when normality holds exactly. The $1/J$ term is the new contribution
due to approximate normality, quantifying the approximation quality of the CLT. Here,
\(J\) is the harmonic mean of $J_i$. Our bound requires accurate marginal normal
approximations only on average, without requiring a high-dimensional CLT. Additionally, \cref{thm:main_rates} is tight as a regret analysis for the NPMLE. We provide a matching lower bound for the worst-case regret in \Cref{asub:npmle-lower-bound}. The proof of \cref{thm:main_rates} follows from more general results on misspecification of \eqref{eq:working}, which we outline in \cref{sec:main-res}.

Beyond regret, \cref{thm:main_rates} has applications for deconvolution and inference. By
\citet{Nguyen2013}, we can show that the Hellinger rate immediately implies a convergence
rate for the Wasserstein distance for the true prior $\pistar$, i.e., $W_2(\pihat,
\pistar)$. However, such rates are necessarily logarithmic (in $(1/n+1/J)^{-1}$), owing to the inherent ill-posedness of
normal deconvolution \citep{Dedecker2013}. Building on the arguments in \citet[Proposition 2.4]{hanzhang}, \cref{thm:main_rates} also implies a similar
logarithmic rate for the coverage of estimated credible sets when $\pistar$ has a density. We collect these results in the following corollary.

\begin{cor}[Estimating $\pistar$ and marginal coverage of estimated credible intervals]
\label{cor:inference}
Under the
 setup of \cref{thm:main_rates}, suppose that
 $(1/n+1/J)^{-1} \gtrsim (\log n)^{C_0}$ for a sufficiently large fixed $C_0$.
 For estimating $\pistar$ in the Wasserstein-2
 distance, we additionally have \[ \E W_2(\pihat, \pistar) \lesssim \pr{1 + \log\frac{1}
 {1/n+1/J}}^{-1/2}.
  \]
  If, in addition, $\pistar$ has a Lebesgue density on $[-M, M]$ that is bounded below and above by $0 < 1/C < 1$ and $1 < C<\infty$, then, for every fixed $\alpha\in(0,1)$,
  \begin{align*}
  \left|\Prob\!\left\{\theta_I \in
  \hat \Theta_{I}(X_1,\ldots, X_n; \alpha)\right\}-(1-\alpha)\right|
  &\lesssim\left\{1+\log\frac{1}{1/n+1/J}\right\}^{-1/8} \numberthis \label{eq:coverage}
  \end{align*} for $I \sim \Unif[n]$ an independently randomized index and $\hat \Theta_i$ the
   estimated equal-tailed $(1-\alpha)$-posterior
   credible interval for $\theta_i$ under \eqref{eq:working}, formed using
   {$(X_i,\hat\sigma_i,\pihat)$}.
\end{cor}

Equation~\eqref{eq:coverage} shows that EB credible intervals (EBCIs) are
\emph{consistent} in terms of average coverage. The particular convergence rate is quite
slow, and likely not improvable beyond logarithmic in $(1/n+1/J)^{-1}$, again due to the
fundamental difficulty of deconvolution \citep{Dedecker2013}. Indeed, empirical
studies of NPMLE-estimated credible interval coverage
 tend to show poor coverage in
practice \citep{koenker2020empirical,ignatiadis2026compound}. Nevertheless, coverage of nonparametric EBCIs is of recent interest
\citep{armstrong2022robust, xie2026nonparametric}, and \cref{cor:inference} is
informative of the properties of NPMLE-derived EBCIs.

These results are for the independent-measurement regime \eqref{eq:local}, which restricts
$X_i$ to be sample means of i.i.d.\ micro-data. We view this restriction as mainly for
analytical convenience and expect analogous results to hold for regular parametric maximum
likelihood estimators, smooth M-estimators, or well-behaved U-statistics, computed
separately within each unit. Indeed, \cref{asec:ral} sketches an extension of these
results---using our high-level conditions developed later in \cref{sec:main-res}---to
\emph{asymptotically linear estimators}, which would cover many of the aforementioned
cases. Likewise, we also expect the analysis to extend to smooth functions of sample means
(or other regular estimators) via the delta method.

\section{Results under linear measurements}
\label{sec:linear_measurements}

As an important extension, we partially generalize the setting of
\cref{sec:indep_measurements} to allow for dependence among $X_i$s and to accommodate
settings where $X_i$s are estimated regression coefficients. Consider the following
setting, which we term \emph{linear measurements}: \[X_i = a_i'Y = \underbrace{a_i'\mu}_
{\theta_i} + a_i'\xi, \quad \theta:= (\theta_1,\ldots,\theta_n) \indep
(\xi_1,\ldots,\xi_m)=:\xi, \quad \theta_i \iid \pistar\] for deterministic and known $a_i$.
The variables $(\xi_t:  t=1,\ldots, m)$ are i.i.d.\ with mean zero and variance $\tau^2$, but
may not be exactly normal. This setting arises in general with panel data, spatial data, and network
data. In this setting, compound decisions and empirical Bayes have only recently been
studied by \citet{hanzhang}, \citet{fan2023gradient}, and \citet {barnhard2026assumption},
under exact normality.\footnote{Theorem 4.2 in \citet{hanzhang} does consider
a nonlinear regression model with non-normal residuals and derive convergence rates under
approximate normality of key quantities. This result relies on assumptions that the
covariates in this regression model are exactly normal (whereas we treat
$a_i$ as fixed) and does not cover our results. }

This setting is challenging to analyze; as a result, we make a few simplifying
assumptions.\footnote{\Cref{sec:general} gives consistency results under substantially
weaker assumptions (high-level assumptions about approximate normality of $X_i \mid
\theta_i$, estimated standard errors, and weak dependence), though without explicit
convergence rates.} First, we continue to operate under the simple assumption that
$\theta_i \iid \pistar$. In this setting, this not only assumes away dependence with
$\var(X_i)$, but also assumes away potential dependence with $a_i$. Second, throughout
this section, we will treat $\tau^2$ as known. Allowing $\tau^2$ to be unknown but still
restricting the micro-data errors $\xi_t$ to be homoskedastic requires straightforward
modifications to the proof. Allowing for heteroskedastic errors is
potentially feasible but would require lengthier modifications. We include some simulation
results in \cref{sec:empirical} using estimated heteroskedasticity-robust variances.

This linear measurements setup nests least-squares regression. To see this, note that we can stack $X_i, a_i, Y_t$ and write $X =
AY$. Consider a regression specification with fixed covariates $(b_t: t=1,\ldots,m)$,
where we define $\theta$ as the coefficient vector in the best linear approximation to $\E[Y]$:\[
Y_t = b_t'\theta + u_t \quad \theta = (B'B)^{-1} B'\E[Y] \quad B = [b_1'; \ldots; b_m']
\in \R^{m \times n}. \] We can express the corresponding least-squares estimator $X$ as \[
X = \underbrace{(B'B)^{-1}B'}_AY = \theta + (B'B)^{-1}B'\xi, \text{ for } \xi = Y - \E
[Y]. \] For example, the AKM model \citep{abowd1999high} can be cast in this way, where $\theta$
collects firm and worker fixed effects \citep[see, e.g.,][and \cref{sec:empirical}]{jochmans2019fixed}. Similarly, we can nest the independent measurement setting in
\cref{sec:indep_measurements} by choosing an appropriate $B$.\footnote{That is, $b_t = e_ {i (t)} \in \R^n$, where $t$ is an
observation in the $i(t)$\th{} experiment and $e_k$ is the $k$\th{} standard basis vector.}
This formulation does not require that $\E[Y]$ be linear in $B$, but assuming $Y - \E[Y]$
contains i.i.d.\ entries is certainly restrictive.

We can write $X_i$ in a standardized form:
\begin{equation}
  X_i=\theta_i+\sigma_i\sum_{t=1}^m c_{it}\eta_t,
  \quad \sum_{t=1}^m c_{it}^2=1,
  \label{eq:dependence-weighted-representation}
\end{equation}
by setting $c_{it}^2 = a_{it}^2 / (AA')_{ii}$,
$\sigma_i^2 = \tau^2 (AA')_{ii}$, and
$\eta_{t} = \xi_t / \tau$. Similar to \cref{assump:lowerbound}, we consider DGP sequences in which the variance of each $X_i$ neither vanishes nor diverges: \[\min_i\sigma_i
\rateeq \max_i \sigma_i \rateeq 1.\] One way to rationalize this, similar to the increasing-variance parametrization in \cref{sec:indep_measurements}, is to assume that $\tau^2$ grows at a rate proportional to  $1/(AA')_{ii}$. Specifically,
for fixed positive constants,\footnote{A
sufficient condition is a uniformly bounded condition number for $AA'$. }
\begin{equation}
  c u_n\leq (AA')_{ii}\leq C u_n,
  \qquad
  \frac{c}{u_n}\leq\tau^2\leq\frac{C}{u_n},
  \qquad 1\leq i\leq n.
  \label{eq:dependence-scale-normalization}
\end{equation}
Throughout this section, we
 assume that $A$ is deterministic and full-rank, $m>n+2$, and $(m,n)$ are sufficiently
 large.

Because of the dependence, it is not obvious what the analogues of $n$ and $J$ in
\cref{thm:main_rates} are in this setting. It turns out the following notions
of
\emph{effective sample size} suffice:
\begin{equation}
  n_{*,A}
  :=\frac{n}{\|\operatorname{Cor}(AA')\|_{\mathrm{op}}},
  \quad
  J_A^{-1}
  :=\frac1n\sum_{i=1}^n\sum_{t=1}^m
       \frac{a_{it}^4}{(AA')_{ii}^2}, \quad \tau^2 AA' = \var(X \mid \theta)
  \label{eq:dependence-design-quantities}
\end{equation} where $\operatorname{Cor}(\Sigma)$ returns the correlation matrix of
 $\Sigma$. The effective sample size in $n$, $n_{*, A}$, is simply the raw sample size
 normalized by the operator norm of the correlation matrix of $X$,
 a notion introduced by \citet{hanzhang}. When entries of $X$ are uncorrelated, this operator norm
 is $1$. On the other hand, when all pairs are correlated with correlation $\rho \neq 0$, this operator norm
 is proportional to $n$---meaning that the effective sample size does not diverge. Under exact normality,
 \citet[Proposition 2.3]{hanzhang}  provide evidence that $n_{*,A}$ is the right notion of effective sample size by deriving a complementary minimax lower bound.

More novelly, the effective sample size \(J_A\) summarizes how many micro-samples effectively contribute to each
measurement $X_i$. If $A$ has a block structure that nests the independent measurements setup,
then $J_A$ is equal to the harmonic mean $J$ in \cref{thm:main_rates}. In general, we can
write $J_A$ also as a particular harmonic mean \[ J_A^ {-1} = \frac1n \sum_{i=1}^n
\frac{1}{J_{i,A}}, \text{ for } J_{i,A} = \frac{\pr{\sum_t c_ {it}^2}^2} {\sum_{t}
c_{it}^4} = \frac{1}{\sum_t c_{it}^4}\]where $J_{i,A}^{-1}$ is the \emph{Herfindahl index}
of the weights $\{c_{it}^2 : t=1,\ldots,m\}$. This has the following interpretation --- $c_
{it}^2$ is the contribution of micro-datum $t$ to the variance of the $i$\th{} measurement
$X_i$. The Herfindahl index is a measure of concentration---in antitrust, a rule of thumb
is that the \emph{inverse Herfindahl index} of market shares is roughly the effective
number of equal-sized oligopolists in a market. By analogy, one could view $J_{i,A}$ as
the effective number of micro-data points that contribute to the variance of $X_i$. The
effective sample size $J_A$ then takes their harmonic mean.

These effective sample sizes are important statistics that practitioners should report
when entertaining an empirical Bayes analysis.\footnote{These effective sample sizes complement the diagnostics that \citet[pp.~1871--72]{kline2020leave}  use to assess normal approximations for inference on quadratic forms in \(\theta\). } As an example, our empirical analyses in
\cref{sec:empirical} compute these for the Veneto Workers Histories dataset, a matched
employee--employer dataset used to study firm wage premia. For $X$ the AKM estimates of
firm effects, we find in \cref{tab:vhw_summary_stats} that, of the $n \approx 74000$ firms
in the dataset, the effective number of firms is only about $n_{*, A} \approx 4100$. The
effective sample size for samples-per-firm is only $J_A \approx 2.53$, suggesting
that the normal approximation is likely poor unless the micro-data were approximately
normal to begin with.

A key regularity condition for our results here is that the law of $\eta_{t}$ satisfies a \emph{logarithmic Sobolev inequality (LSI)}.\footnote{This is
a
standard notion in concentration inequalities. See Chapters 5--6 in \citet{sgp2013} and
Chapter 3 in \citet{wainwright2019highdimensional}.} This condition strengthens the
uniform sub-Gaussianity condition in \cref{thm:main_rates}. It naturally accommodates dependent data, yielding sub-Gaussian concentration for Lipschitz functions of $(\eta_1,\ldots,\eta_m)$ \citep[][Theorem~5.3]{ledoux2001concentration} and Bernstein-type concentration for smooth functions with bounded Hessians \citep[][Theorem~3.19]{adamczak2017moment}.

\begin{assume}[Log-Sobolev inequality]
\label{assump:dependence-lsi}
The variables $\eta_1,\ldots,\eta_m$ in \eqref{eq:dependence-weighted-representation} are i.i.d. from a \emph{continuous distribution}, which has mean zero and variance one and
satisfies LSI, i.e., there is a fixed $C_{\rm LS}<\infty$ such that every continuously differentiable $g:\R\to\R$
satisfies
\begin{equation}
  \operatorname{Ent}\{g(\eta_1)^2\}
  \leq 2C_{\rm LS}\E|g'(\eta_1)|^2.
  \label{eq:dependence-one-dimensional-lsi}
\end{equation}
Here, for a nonnegative random variable $U$,
$\operatorname{Ent}(U)=\E(U\log U)-(\E U)\log(\E U)$, with $0\log0=0$.
\end{assume}
The LSI condition \eqref{eq:dependence-one-dimensional-lsi} is automatically satisfied
if $\eta_1,\ldots, \eta_m \iid \Norm(0,1)$. It
allows for some non-normal
random variables, such as random variables with strongly log-concave densities ($p(\eta) =
e^{-V(\eta)}, V''\ge \kappa > 0$), uniform distributions, and normal
location mixtures with compactly supported mixing distributions.  It excludes non-continuous distributions,
though discrete versions of
\eqref{eq:dependence-one-dimensional-lsi} do exist and we leave those for future work.

We now discuss the nonparametric EB estimator and regret criterion in this setting. In the
presence of dependence, even if $X_i \mid \theta$ were exactly normal, the joint
likelihood of $X$ no longer factors into individual terms. The likelihood in
\eqref{eq:npmle} is then only an \emph{independence likelihood}~\citep{chandler2007inference}, a special case of \emph{composite likelihood}~\citep{varin2011overview,graham2020sparse}, ignoring the dependence
between $X_i$. We consider maximum independence likelihood estimators:\footnote{As in
\cref{thm:main_rates}, we could continue to allow for approximate maximization, but
suppress it for ease of presentation.} \begin{equation}\label{eq:dependence-npmle} \pihat
\in \operatorname*{argmax}_{\pi \in \mathcal{P}[-M,M]} \; \avgn \log \int \phi\!\left(X_i
- \theta;\, \sigma_i\right) d\pi(\theta). \end{equation} Resorting to the independence
likelihood is often unavoidable. Many datasets where EB methods are applicable, like
the Opportunity Atlas \citep{chetty2026opportunity}, only publish the marginal standard
errors for each $X_i$. Fortunately, in the
setting where $\eta_t$ are exactly normally distributed, maximizing the independence
likelihood does not lose much in terms of convergence rates \citep[Proposition
2.3]{hanzhang}.

We also continue to consider the \emph{marginal}
regret notion in \eqref{eq:regret-z}, setting $\hat\sigma_i = \sigma_i$ since $\sigma_i$
is known here. Similar to \citet{hanzhang}, we caution that the oracle comparison
restricts the oracle to the decision rule $h_{\pistar} (Z_i; \sigma_i)$, which is merely
the best \emph{separable} decision rule (i.e., only a function of $(Z_i, \sigma_i)$) under
the normal model. It is suboptimal relative to the oracle Bayes decision rule $\E_ {\pistar}
[\theta_i \mid Z_1,\ldots, Z_n, \sigma_1,\ldots, \sigma_n]$.\footnote{An analogous
consideration arises in multiple testing, where the marginal local false discovery rate
$\mathbb{P}_{\pistar}[\theta_i = 0 \mid Z_i, \sigma_i]$ is a commonly used target under
dependence, even though it is not equal to $\mathbb{P}_{\pistar}[\theta_i = 0 \mid
Z_1,\ldots,Z_n,\sigma_1,\ldots,\sigma_n]$~\citep{heller2021optimal,
karmakar2025inference}.}

With these preparations, we now state our result in this setting---which is entirely analogous to \cref{thm:main_rates}, with effective notions of sample size $(n_{*,A}, J_A)$ replacing $(n,J) $.

\begin{thm}
\label{thm:quantitative-dependence-maintext}
Under the preceding setup, we have
\[
  \E\Helbarstar{f_{\pihat}}{f_{\pistar}}
  \lesssim (\log m)^4
       \left(\frac1{n_{*,A}}+\frac1{J_A}\right) \text{ and  } \bigl(\mathrm{Regret}_n^+\bigr)^2
    \lesssim
   (\log m)^6
    \left(\frac1{n_{*,A}}+\frac1{J_A}\right).
\]
If the $\eta_t$ are standard normal, $J_A^{-1}$ may be omitted from both
bounds.
\end{thm}

As in \cref{thm:main_rates}, the term $n_{*,A}^{-1}$ captures the difficulty of
learning $\pistar$ from dependent measurements, while $J_A^{-1}$ captures the additional
error from approximating the marginal noise distributions of the $X_i$ by normal distributions.
Consequently, vanishing Hellinger distance and vanishing marginal regret are possible
when both effective sample sizes grow sufficiently quickly. As in the independent-measurement regime, this result only requires
control of the \emph{marginal} normal approximations \emph{on average}, without requiring a normal
approximation to the entire vector $X$.

When $\eta_t$ are standard normal, \cref{thm:quantitative-dependence-maintext} recovers
the Hellinger and regret results for compactly supported priors of \citet[Theorem 2.1 and
Proposition 2.5]{hanzhang}. It provides an alternative proof as well. Theorem 2.1 in
\citet{hanzhang} uses the geometric Brascamp--Lieb inequality~\citep{chen2015improved} to handle the multivariate normal dependence, in analyzing empirical processes of the form $\frac{1}{n} \sum_i (f(X_i) - \E[f(X_i)])$. Meanwhile, the proof of
\cref{thm:quantitative-dependence-maintext} handles this difficulty with the following
decomposition: \[ \underbrace{\frac{1}{n}\sum_i
f(X_i)}_{F_n(\eta_1,\ldots,\eta_m)} - \E\bk{\frac{1}{n}\sum_i f(X_i)} = \underbrace{F_n -
\E[F_n \mid \theta_1,\ldots,\theta_n]}_{\text{Concentrates by LSI}} +
\underbrace{\frac{1}{n}\sum_i \E[f(X_i) \mid \theta_i] - \E[F_n]}_{\text{Concentrates as
averages of independent terms}}. \] Normal random variables already
satisfy LSI. Since all this argument requires is the LSI, it then readily generalizes to
non-normal $\eta_t$.

\section{Proof outline and generalized notions of approximate normality}
\label{sec:main-res}

We now present an outline of the proof of \cref{thm:main_rates} in the independent
measurements regime. In doing so, we identify natural high-level notions (i) for the
distribution of $X_i$ to be \emph{approximately normal} and (ii) for the NPMLE
\eqref{eq:approximate_mle} using \emph{estimated variances} to approximately maximize the
\emph{oracle} likelihood with $\sigma_i$. We view these high-level conditions themselves
 as a contribution; they also lead to results not covered in either of the two leading
 regimes discussed. Meanwhile, the proofs of
\cref{thm:main_rates,thm:quantitative-dependence-maintext} follow this broad template,
 verifying these high-level conditions using Stein's method \citep{stein1972bound,barbour1986asymptotic}.

We start by considering the average Hellinger distance, which leads to
regret convergence by results in \citet{chen2026sharp} and \citet{jiang2009general}. Let \[L_n (\pi) :=
\frac{1}{n} \sum_{i=1}^n \log \frac{f_{\pi, \sigma_i}(X_i)}{f_{\pistar, \sigma_i} (X_i)}\]
be the sample log-likelihood ratio
 evaluated at the non-normal $X_i$, but with respect to the oracle variance $\sigma_i^2$.
 Let $L(\pi) = \E_X L_n(\pi)$ be its population expectation, with respect to the law of
 the $X$'s.

Our argument modifies a version of the \emph{localization} argument for convergence rates of
M-estimators in \citet [Theorem 3.2.5]{vaart2023empirical}, \citet
{ghosal2001entropies}, \citet{geer2000empirical}, and \citet{wong1995probability}. Let us
recall this argument. Suppose an estimator $\pihat$ satisfies the \emph
{basic inequality} \[ 0 \le L_n(\pihat) = (L_n(\pihat) - L(\pihat)) + L
(\pihat). \numberthis \label{eq:basic_ineq}
 \]
 If $\sigma_i$ were known and the likelihood were exactly maximized, then \eqref{eq:basic_ineq} holds automatically.

 Furthermore, if the working model \eqref{eq:working} were correct and
 $\hat\sigma_i=\sigma_i$, then \[ L(\pi) = \frac{1}{n} \sum_{i=1}^n
 \underbrace{\E_{f_{\pistar, \sigma_i}} \log \frac{f_{\pi, \sigma_i}} {f_{\pistar,
 \sigma_i}}}_{-\operatorname{KL}(f_{\pistar, \sigma_i}\,\Vert\,f_{\pi, \sigma_i})} \le -2
 \Helbarstar {f_\pi}{f_{\pistar}}.\numberthis\label{eq:quadratic_term} \] Thus, an
 inaccurate $\pihat$ in terms of the average Hellinger distance would contribute a
 negative quadratic term to the right-hand side of \eqref{eq:basic_ineq}.

 Suppose,
 additionally, that we have a \emph{local maximal inequality} for the first term in
 \eqref{eq:basic_ineq}: With high probability, \[ \sup_{\pi:
 \Helbarstar{f_\pi}{f_{\pistar}} \le r^2} |L_n(\pi) - L(\pi)| \lesssim
 \frac{\polylog(n)} {\sqrt{n}} r \log^C(1/r). \numberthis
 \label{eq:local_maximal_ineq_maintext} \] When the working model \eqref{eq:working}
 holds, \eqref{eq:local_maximal_ineq_maintext} is possible because $L_n(\pi)$ is a sample
 mean of independent terms, with variance bounded by \[ \var(L_n(\pi)) \le \frac{1}{n^2}
 \sum_{i=1}^n \E_{f_{\pistar, \sigma_i}} \log^2 \frac{f_{\pi, \sigma_i}} {f_{\pistar,
 \sigma_i}} \overset{(1)}{\lesssim} \frac{1}{n} \Helbarstar {f_\pi}{f_{\pistar}}
 \log^C\pr{\frac{1}{\Helbarstar{f_\pi}{f_{\pistar}}}} \numberthis
 \label{eq:variance_implication}.\]Here, the inequality (1) comes from \citet[Lemma 4.1]
 {ghosal2001entropies}; see also \citet{wong1995probability} and \citet{kaji2026hellinger}. Bernstein's inequality applied to $L_n(\pi) - L(\pi)$ then shows
 that \eqref{eq:local_maximal_ineq_maintext} holds pointwise for a fixed $\pi$. Familiar
 entropy calculations in, e.g., \citet{Jiang2020} upgrade
 this result to be uniform over $\br{\pi: \Helbarstar{f_\pi}{f_{\pistar}} \le r^2}$.

 With these preparations, the basic inequality immediately gives the Hellinger rate. To see this, let $\delta_n^2 \rateeq \Helbarstar{f_{\pihat}}{f_{\pistar}}$ be the average Hellinger rate. If \eqref{eq:quadratic_term} and \eqref{eq:local_maximal_ineq_maintext} were to hold, then plugging $\delta_n^2$ into the basic
 inequality \eqref{eq:basic_ineq} would give that \[
	0 \le C\frac{\polylog(n)}{\sqrt{n}}\delta_n \log^C(1/\delta_n) - c\delta_n^2
	\implies \delta_n^2
	\lesssim \frac{\polylog(n)}{n}.
\]
We thus have an upper bound on the Hellinger rate $\delta_n$, which matches results under exact normality \citep{Jiang2020}.

In our setting, the non-normality of $X_i$ breaks the preceding argument in at least three places:
\begin{itemize}
	\item The basic inequality \eqref{eq:basic_ineq}---due to noise in
$\hat\sigma_i$ and approximate optimization;
\item  The  relation between $L(\pi)$ and Hellinger
distance in
\eqref{eq:quadratic_term}---due to replacing $\E_{f_{\pistar, \sigma_i}}$ with $\E_
{X_i}$; and
\item The variance bound in \eqref{eq:variance_implication}---again, due to replacing
$\E_{f_{\pistar, \sigma_i}}$ with $\E_
{X_i}$.
\end{itemize}
We next introduce high-level conditions to quantify these departures. These notions are verified for independent measurements in
\cref{sub:verify-for-indep-meas}.

\subsection{Hellinger convergence under approximate normality} First, we state a notion of
 approximate oracle NPMLE that quantifies the violation of
\eqref{eq:basic_ineq}.
\begin{defi}
\label{def:approximate_npmle}
	For a fixed sequence $b_n \ge 0$, we say that $\pihat$ is a \emph{$b_n$-approximate
	oracle
        NPMLE} if for all sufficiently small $\eta > 0$, there exists a constant $C_\eta >
        0$ such that with probability at least
\(1-n^{-2}\),
\begin{equation}
 L_n(\pihat)= \avgn
  \log\frac{f_{\pihat,\sigma_i}(X_i)}
           {f_{\pistar,\sigma_i}(X_i)}
  \geq
  -\eta\Helbarstar{f_{\pihat}}{f_{\pistar}}
  - C_\eta b_n.
  \label{eq:likelihood-bound}
\end{equation}
\end{defi}
\noindent
If $\pihat$ is a $b_n$-approximate NPMLE, then \eqref{eq:basic_ineq} holds with $-\eta\Helbarstar{f_{\pihat}}{f_{\pistar}}
  - C_\eta b_n$ on the left-hand side. Importantly, this bound depends on $\Helbarstar{f_{\pihat}}{f_{\pistar}}$, allowing one to choose a sequence $b_n$ that vanishes quickly to zero \citep[cf.][Corollary SM6.1]{chen2024empirical}.
  
    Second, we define a notion of the marginal law of $X_i$ being an \emph{approximate normal
	convolution} of $\pistar$, which quantifies the gaps in \eqref
	{eq:quadratic_term} and \eqref{eq:variance_implication}. For $T \geq 1$, $y \ge 0$,
	define the following function
\begin{align}\label{eq:gT}
		g_{T}(y) := \begin{cases} y \log^{2}(T/y) & \mbox{if } y>0 \\ 0 & \mbox{if } y=0 \end{cases} \approx y\log^2(1/y).
\end{align}

\begin{defi}\label{asn:assumeindep}
Fix $M > 0$, $T \geq 1$, and constants $c_1, c_2 > 0$.\footnote{	In the
sequel we suppress the dependence on $M$, $T$, $c_1$, and $c_2$, treating them as fixed throughout our analysis.} Let $\mu$ be a probability measure on $\R$, let $\sigma > 0$, and let $\pistar \in \cP([-M,M])$. We say that $\mu$ is \emph{an $(r_1, r_2)$-approximate normal convolution} of $\pistar$ at scale $\sigma$ if $r_1, r_2 \geq 0$ and, for every $\pi \in \cP([-M,M])$,
	\begin{align}\label{eq:condition-a}
		\E_{X \sim \mu} \log \frac{f_{\pi, \sigma}(X)}{f_{\pistar, \sigma}(X)} &\leq c_1\left
		(-\Helsq(f_{\pi, \sigma}, f_{\pistar, \sigma}) + r_1\right), \text{ and  }
\\
\label{eq:condition-b}
		\E_{X \sim \mu} \log^2 \frac{f_{\pi, \sigma}(X)}{f_{\pistar, \sigma}(X)} &\leq c_2\left( g_{T}\!\left(\Helsq(f_{\pi, \sigma}, f_{\pistar, \sigma})\right) + r_2\right).
	\end{align}

	\end{defi}
\noindent
\Cref{asn:assumeindep} is stated as a high-level condition. If $\mu$ is exactly a normal
convolution of $\pistar$, then \eqref{eq:condition-a} and \eqref{eq:condition-b} hold
with $r_1 = r_2 = 0$ \citep[Lemma 4.1]
{ghosal2001entropies}.  The quantities $r_1$
and $r_2$ thus quantify the price of misspecifying \eqref{eq:working} in equations \eqref{eq:quadratic_term} and
 \eqref{eq:variance_implication}, respectively.

We should expect $r_1, r_2$ to behave like some approximation error of the CLT. To give some intuition, \cref{prop:hellinger-bound} shows that if $X_i \sim \mu_i$ is distributed continuously, then we can choose \begin{align} r_{1,i} = r_{2,i} =
 g_{T}\!\left(\Helsq(\mu_i, f_{\pistar,\sigma_i})\right),
 \label{eq:r-hel}
 \end{align}
and so we can apply our results alongside any $X_i$ for which existing CLTs in the Hellinger metric exist. For instance,~\citet{austern2024bounding} study $X_i$ that are non-degenerate U-statistics computed based on $J_i$ samples and show, under certain technical conditions,
that $\Helsq(\mu_i, f_{\pistar,\sigma_i}) \lesssim J_i^{-1}$ up to logarithmic factors. Thus, for such $X_i$ we can apply our result with $r_{1,i}, r_{2,i}  \lesssim J_i^{-1}$ (up to logarithmic factors).

We think of~\eqref{eq:r-hel} as proof of principle that our analysis is applicable for broad classes of $X_i$ that are approximately normally distributed. Below, e.g., in the proof of~\cref{thm:main_rates},
we develop a more specialized analysis that avoids some of the stronger regularity conditions needed for a Hellinger CLT (e.g., our results also hold for discrete $X_i$), and shows that Definition~\ref{asn:assumeindep} is fully compatible with CLTs in weaker metrics.

Under these high-level conditions, we can complete the localization argument outlined
above and obtain a general result for independent $X_i$---without needing $X_i$ to follow  \eqref{eq:local}. Like in \cref{thm:main_rates}, we require $X_i$ to be uniformly sub-Gaussian.

\begin{assume}[Uniform sub-Gaussianity]\label{assump:sub-Gauss}
		There exist $C_1, C_2 > 0$ such that for all $t \geq 1$,\footnote{Traditionally sub-Gaussianity (see \citet[Chapter 2.6]{vershynin2018high}) is an assumption on centered random variables. In that light, the above assumption can be viewed as a combination of sub-Gaussianity and a uniform bound on expectations, i.e., $\limsup_{n\to\infty}\sup_{1\le i\le n} |\E[X_i]|<\infty$.} we have
		$
		\max_{1 \leq i \leq n} \Prob\!\left(|X_i| > t\right) \leq C_1 \exp(-C_2 t^2).
		$
	\end{assume}

\begin{thm}[High-level result for Hellinger rates of NPMLE]\label{thm:main}
Assume $X_i$ are independent and have marginal distributions $X_i \overset{\mathrm{ind}}
{\sim} \mu_i$.
Let $\hat \pi_n$ be a $b_n$-approximate oracle NPMLE in the sense of
\eqref{eq:likelihood-bound}. Suppose \cref{assump:sub-Gauss} holds and
suppose  every $\mu_i$ is an $(r_{1i}, r_{2i})$-approximate normal convolution in
    	the sense of \cref{asn:assumeindep} with respect to scales $\sigma_1,\ldots,
    	\sigma_n$, which satisfy \cref{assump:lowerbound}.
    	Define $\rbar{j} := \avgn r_{j,i},
    	 j \in \br{1,2}$.
    	Then there exists some constant $C$ such that
    	for \[\delta_n^2 = C \pr{\frac{\log^4 n}{n} + b_n + \bar r_1 + \bar r_2},\] we
        have
\[
			\Prob\bigl(\Helbarstar{f_{\hat\pi_n}}{f_{\pistar}} \ge
			 \delta_n^2\bigr) \lesssim n^{-2}
    		\text{ and  }\;\; \E\bk{
    			\Helbarstar{f_{\hat\pi_n}}{f_{\pistar}}
    		} \lesssim \delta_n^2.
    	\]
\end{thm}

\Cref{thm:main} recovers the standard parametric rate in the correctly
specified case as in \cite{zhang2009generalized} and
\cite{Jiang2020} by setting $r_{1i} = r_{2i} = 0$. More generally, \cref {thm:main}
 makes explicit how likelihood misspecification and approximate oracle likelihood maximization contribute to the convergence bound through $\bar r_1, \bar r_2, b_n$.\footnote{
Let us briefly contrast our analysis
with other results on nonparametric maximum likelihood under misspecification. \citet{patilea2001convex} and~\citet{geer2000empirical} establish convergence results under a modified squared Hellinger distance, which in general is different from the usual Hellinger distance in~\eqref{eq:hellinger}.
In an empirical Bayes setting, such a strategy is pursued in~\citet{kim2026empirical}. However, it is unclear how one would translate rates on this modified Hellinger distance to rates on denoising regret, the main objective of this work.
}
Notably,
$ \rbar{1}$ and $ \rbar{2}$ aim to capture the gap between $\mu_i$ and $\cN(0, \sigma_i^2)
* \pistar$ \emph{averaged} across $1\le i\le n$---obviating
high-dimensional central limit theory. The rate $\delta_n^2$ is not overly influenced if a
vanishing fraction of units are not approximately normal convolutions. When specialized to the independent measurements regime (\cref{sub:verify-for-indep-meas}), we can bound $
\max(\bar r_1, \bar r_2) \lesssim \frac{\polylog(J)}{J},
$
implying the rate in \cref{thm:main_rates}.

The generality of \cref{thm:main} generates additional results not covered by \cref{thm:main_rates}.  First, it allows
for showing that the NPMLE \emph{adapts} to more accurate normal
approximations, at least when $\sigma_i$ is known, by deriving tighter bounds on $r_1,r_2$. In particular, if the
$\varepsilon_{ij}$'s in \eqref{eq:local} are \emph{symmetric}, \cref{asub:symmetric} shows
that $\rbar{1}+ \rbar{2}\lesssim \frac{1}{n}\sum_ {i=1}^n J_i^{-2}$ up to logarithmic
factors. Plugging in these faster rates for $(r_1,r_2)$ improves on \cref{thm:main_rates}. Second, it naturally enables analyses of more complex regimes, such as for asymptotically linear estimators. In particular, combined with the local asymptotics setup in \cref{rmk:local_asymp}, \cref{thm:main} shows Hellinger results in settings with binomial data \citep[recently studied by][]{chen2025compoundbin}. We sketch these results in \cref{asec:ral}.

\subsection{Verifying \cref{asn:assumeindep,def:approximate_npmle} under independent
measurements}
\label{sub:verify-for-indep-meas}

To specialize \cref{thm:main} to the independent measurements regime, we prove
\cref{thm:main_rates}
by showing that
\[
\rbar{1}+\rbar{2}\lesssim {\avgn\frac{\log^2 J_i}{J_i}\lesssim\frac{\log^2 J}{J}}
\;\;\text{ and }\;\;
b_n\lesssim \log^4 n\left(\frac1n+\frac1J\right),
\]
where $J^{-1}=\avgn J_i^{-1}$ as in \cref{thm:main_rates}.

For verifying \eqref{eq:condition-a} and
\eqref{eq:condition-b} with $(r_1,r_2)$ at these rates, note that either can be written as a bound on
\begin{equation}
	\E_{\pistar}{\bk{\br{\E_{X}[t(X_i) \mid \theta_i] - \E_{f_{\pistar, \sigma}}[t(Z_i) \mid
	\theta_i]}}
	} + \E_ {f_ {\pistar, \sigma}} [t
	(Z_i)]
\label{eq:stein_plus_minus}
\end{equation}
for  $t(\cdot)$ equal to either the log-likelihood ratio or its square.
For the first term in \eqref{eq:stein_plus_minus},
 differences like $\E_ {X} [t(X_i) \mid \theta_i] - \E_{f_{\pistar, \sigma}}[t
 (Z_i) \mid \theta_i]$ are precisely controlled by Stein's method~\citep{stein1972bound}. Specifically, a first-order Edgeworth expansion~\citep{barbour1986asymptotic} yields, loosely speaking, the following bound for smooth $t$:
 \[ \abs{ \E_{X}[t(X_i) \mid \theta_i] - \E_{f_{\pistar, \sigma}}[t
  (Z_i) \mid \theta_i] } \lesssim \frac{1}{\sqrt{J_i}} \E_{Z \sim \Norm
  (\theta_i, \sigma_i^2)}\br{(1+|Z|^3) |t(Z)|} + \frac{1} {J_i}.  \numberthis \label
  {eq:steinexpansion}\]
  Furthermore, the second term $\E_ {f_{\pistar, \sigma}} [t(Z_i)]$ in \eqref{eq:stein_plus_minus}  and $\E \bk{
  (1+|Z|^3) |t(Z)|}$ in~\eqref{eq:steinexpansion}  can then be related to  $\Helsq(f_{\pihat, \sigma_i}, f_{\pistar,\sigma_i})$, along the lines of \citet[Lemma 4.1]{ghosal2001entropies}. The magnitudes of these terms then vanish as $\Helsq(f_{\pihat, \sigma_i}, f_{\pistar,\sigma_i}) \to 0$. Thus
the leading term in the normal approximation error exhibits the same localization as in \eqref{eq:quadratic_term} and
\eqref{eq:variance_implication}.
   The precise bounds on $r_{1i},r_{2i}$, including
the logarithmic factors, are given in \cref{prop:moment-normal-convolution}.

Finally, to verify that \(\hat\pi_n\) is an approximate oracle NPMLE, we follow  \citet[Theorem SM6.1]{chen2024empirical}. We linearize the likelihood around \(\sigma_i\) and control the resulting errors using arguments like \eqref{eq:steinexpansion} repeatedly.

	\subsection{Consequences of \cref{thm:main} for deconvolution and denoising regret}
	\label{sec:deconvolution}

  Finally, we state consequences of \cref{thm:main} for deconvolution and regret.
  Deconvolution results follow directly from \citet{Nguyen2013}; see also \citet[Theorem 10]{Soloff2024}.

	\begin{thm}[Deconvolution]\label{thm:deconvolution}
		Suppose the assumptions required for \cref{thm:main} hold.
		Let $\delta_n$ be as in \cref{thm:main}. Then
		$
		\E\!\left[\Wass(\pihat, \pistar)\right] \lesssim
		\big(1+\log(1+\delta_n^{-1})\big)^{-1/2}.
		$
	\end{thm}

We next turn to the denoising regret \eqref{eq:regret-z}. In the general case, we bound
regret in terms of the rate $\delta_n^2$ in \cref{thm:main} as well as the Wasserstein-2
discrepancy between the true conditional law of $X_i$ given $\theta_i$ and the normal working model,
\[
\mathcal{W}_n^2 := \avgn \E_{\theta_i \sim \pistar}\!\left[\Wtwo^2\!\left(\mathrm{Law}(X_i
\mid \theta_i),\, \cN(\theta_i, \sigma_i^2)\right)\right].
\]

\begin{thm}[Denoising regret against the normal oracle]\label{thm:denoising-normal}
Suppose the conditions of \cref{thm:main} hold for some $\pistar\in \cP([-M,M])$.
Let $\delta_n$ be as in \cref{thm:main}. Suppose also that $\hat\sigma_i$ are uniformly bounded above and away
from zero almost surely, and that
\begin{equation}
\avgn \E\!\left(\hat\sigma_i^2-\sigma_i^2\right)^2
\lesssim \delta_n^2.
\label{eq:estimated-variance-regret-condition}
\end{equation}
Then \[
\left(\mathrm{Regret}_n\right)_+^2 \lesssim
\delta_n^2\max\{1,\log(\delta_n^{-1})\}
+\mathcal{W}_n^2.\numberthis \label{eq:general_regret}
\]
\end{thm}

\noindent In the well-specified case, the first term of \eqref{eq:general_regret} quantifies how estimation of $f_
{\pistar}$ affects regret, using recent results from~\citet[Theorem~1]{chen2026sharp}. The
term $\mathcal W_n^2$, which vanishes under exact normality, is the cost of evaluating a
normal posterior mean on approximately normal data.

We conclude by considering a stronger regret notion---against a Bayes-optimal denoiser given the
actual estimate $X_i$
\[
\delta_i^{\star}(x) := \E[\theta_i \mid X_i = x],
\]
which can use any remaining non-normal structure in the law of $X_i\mid \theta_i$. The
corresponding regret is
\[
\mathrm{Regret}_n^{\star} := \sqrt{\avgn \E\bigl|\theta_i - h_{\pihat}
(X_i;\hat \sigma_i)\bigr|^2} - \sqrt{\avgn \E\bigl|\theta_i - \delta_i^{\star}
(X_i)\bigr|^2}.
\]
We assume that $X_i \mid \theta_i$ admits a conditional Lebesgue density,
say $\nu_ {\theta_i,i}(\cdot)$. We define an additional discrepancy between the true
conditional density of $X_i\mid \theta_i$ and the normal working model, the Hellinger
discrepancy,
\[
\mathcal{H}_n^2 := \avgn \E_{\theta_i \sim \pistar}\!\left[\Helsq\!\left(\nu_{\theta_i,i}(\cdot),\, \cN(\theta_i, \sigma_i^2)\right)\right].
\]
We then have that squared $\mathrm{Regret}_n^{\star}$ is controlled with this additional term $\mathcal{H}_n^2$:

\begin{thm}[Denoising regret against the true posterior]\label{thm:denoising-strong}
Under the conditions of Theorem~\ref{thm:denoising-normal}, including \eqref{eq:estimated-variance-regret-condition},
and assuming $\mathcal H_n$ is well-defined,
\[
\left(\mathrm{Regret}_n^{\star}\right)_+^2 \lesssim
\delta_n^2\max\{1,\log(\delta_n^{-1})\}
+\mathcal{W}_n^2 + \mathcal{H}_n^2.
\]
\end{thm}

\noindent The additional $\mathcal{H}_n^2$ term reflects the gap between using the normal likelihood and the exact, possibly non-normal, likelihood of $X_i$. In the well-specified case $X_i\mid \theta_i \sim \cN(\theta_i,\sigma_i^2)$, both $\mathcal{W}_n$ and $\mathcal{H}_n$ vanish, the two oracles coincide, and $\mathrm{Regret}_n = \mathrm{Regret}_n^{\star}$.

\section{Empirical investigations of normal approximation quality}
\label{sec:empirical}

We now study how normal approximation quality affects EB performance in economically
realistic settings, using simulations, simulations calibrated to real data, and
sample-splitting exercises on real data.

\subsection{Calibrated simulations with independent measurements}

\begin{figure}[tbh]
	\centering
	\includegraphics[width=\linewidth]{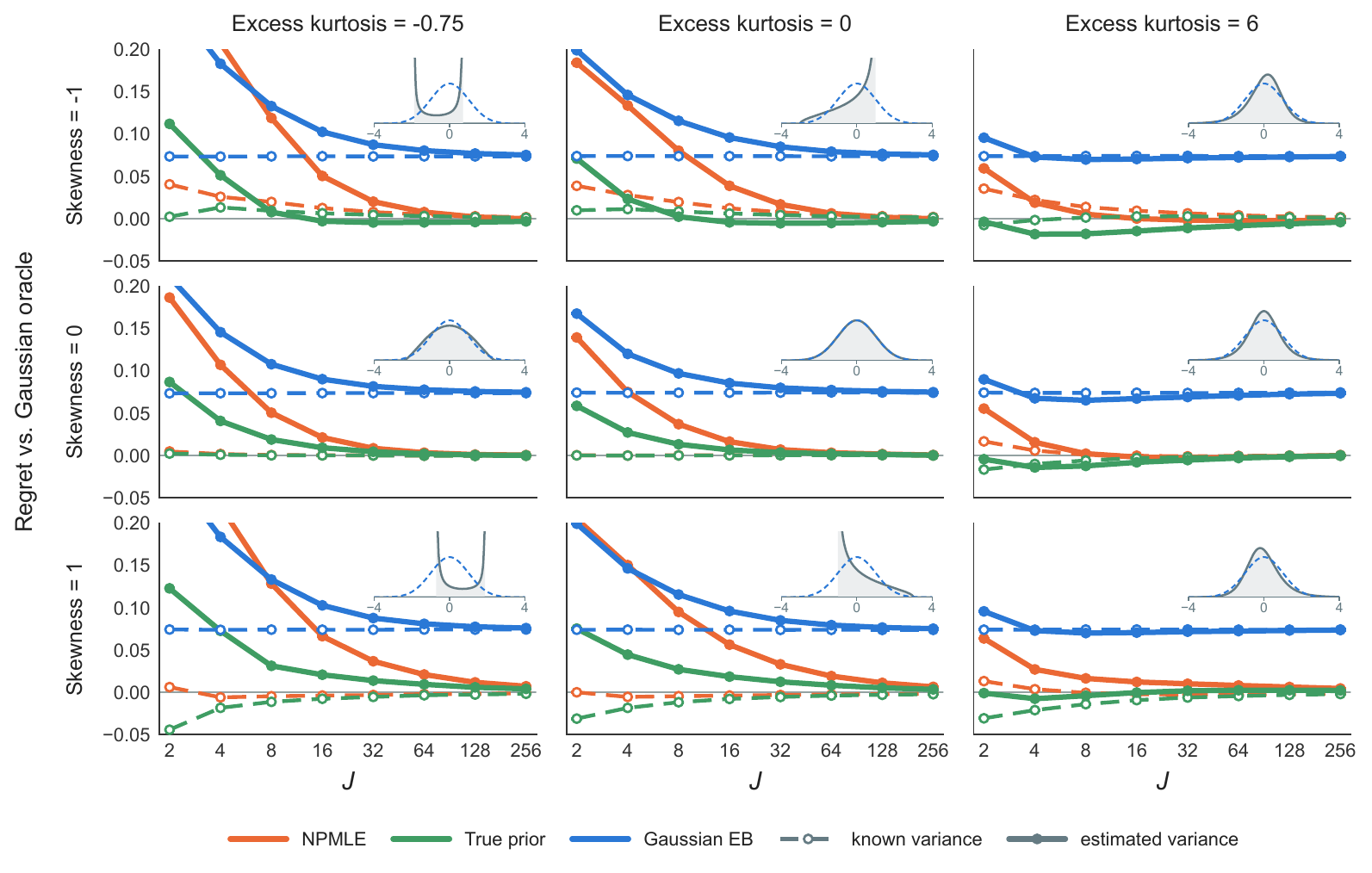}

	\begin{proof}[Notes]
	The graph includes $\pm 2$ Monte Carlo standard errors over $M=100$ draws (though barely visible). For each panel, an inset compares the density of $\varepsilon_{ij} / \sqrt{J}$ to that of a standard normal distribution. Here all $J_i = J$ and $n = 10^5$.
	\end{proof}
	\caption{Regret of EB methods as a function of $J$, by skewness and excess kurtosis of the law of $\varepsilon_{ij}/\sqrt{J}$.}
	\label{fig:pure_simulation}
\end{figure}

\subsubsection{Simulation with Pearson distributions} We begin with simulating from \eqref{eq:local} and \eqref{eq:true_prior} via specifying different $\pistar$ and different laws of $\varepsilon_{ij}/\sqrt{J}$. For the law of $\varepsilon_{ij}/\sqrt{J_i}$, we consider the family of
\emph{Pearson distributions}, which parametrizes a family of mean-zero, variance-one
continuous distributions by skewness and kurtosis \citep{pearson1894contributions}. For
$\pistar$, we choose the empirical distribution of $\theta_i$ in the ``two towers''
example in \citet[Figure 1]{efron2019bayes}, which is approximately a mixture of
$\Unif[-1.7,-0.7]$ and $\Unif[0.7, 2.7]$.

We vary the skewness and kurtosis of $\varepsilon_{ij}$ over a \(3\times3\) grid and take \(J=2,4,8,\ldots,256\).
\Cref{fig:pure_simulation} compares normal posterior mean predictions \(h_\pi(X_i;\sigma_i)\) using three estimated priors $\pi$:
\textcolor[HTML]{EB6834}{(i) NPMLE},
\textcolor[HTML]{2A78D6}{(ii) an estimated Gaussian prior}, or
\textcolor[HTML]{3F9D63}{(iii) the true prior $\pistar$}. We also vary whether $\sigma_i$ is the true
or estimated standard deviation. In these simulations, we consider a relatively large $n = 10^5$ to isolate the scaling behavior in $J$. We make the following observations:

\begin{enumerate}[wide]
    \item
    Across all nine setups, the scaling of NPMLE regret (with estimated or oracle $\sigma_i$)
    as a function of $J$ is below the upper bound $1/\sqrt{J}$ from \cref{thm:main} (see
    \cref{fig:pure_loglog} for a log--log plot). In particular, the regret of NPMLE-with-oracle-$\sigma_i$ scales approximately as $1/\sqrt{J}$ for skewness $-1$. This is
    consistent with \cref{thm:main}. Many empirical regret curves appear closer to $1/J$ instead, which is suggestive of additional adaptation properties of the NPMLE not captured by \cref{thm:main},
    along the lines of \cref{thm:moment_adapt}.

    \item

    In settings with small $J$, one might conjecture that simpler decision rules---e.g., modeling $\pistar \sim \Norm(m_0, s_0^2)$---may be more robust than the NPMLE. In \cref{fig:pure_simulation}, this is true for some configurations, but the NPMLE with estimated variance actually outperforms its parametric counterpart in all nine configurations when $J\ge 8$.

    \item \label{item:kurtosis} Generally, procedures with estimated $\sigma_i$ perform worse than their oracle-$\sigma_i$ counterparts. Curiously, this effect is less pronounced when excess kurtosis is high---with the oracle prior $h_{\pistar}$, procedures with estimated variance even sometimes outperform those with oracle variance and nearly always outperform the Gaussian oracle. This is potentially due to the fact that $\cov((X_i - \theta_i)^2, \hat\sigma_i^2/\sigma_i^2) > 0$ if and only if the excess kurtosis is positive: Heuristically,  a large $\hat\sigma_i^2$ leads to more shrinkage of $X_i$, which is a favorable decision if $(X_i - \theta_i)^2$ is large.
\end{enumerate}

\subsubsection{Simulation calibrated to \citet{kline2022systemic}} We next turn to the setting of \citet{kline2022systemic}, who study systemic discrimination of large U.S.
firms. \citet{kline2022systemic} estimate the white--black callback gap for job posts of these firms through correspondence experiments. Through a
calibrated simulation, we study how potential non-normality of micro-data affects
nonparametric empirical Bayes methods in this setting. In short, we find that the analogue
of $J$ in \citet{kline2022systemic} is likely sufficiently large; increasing it does not
meaningfully affect the performance of nonparametric EB methods in our simulations.

To describe our simulation, let $X_{ij}^{\data}$ be the empirical white--black callback
gap at job $j$ offered by firm $i$ and let $X_i^\data = \frac{1}{J_i^{\data}}
\sum_{j=1}^{J_i^\data} X_{ij}^{\data}$.\footnote{Typically, for each job,
\citet{kline2022systemic} send four fictitious white-coded resumes and four black-coded
resumes, and $X_{ij}^{\data}$ is the empirical difference in callback rates.} For each
firm $i$, we again simulate from \eqref{eq:local}, where we set the law of $\varepsilon_{ij}$ to be the empirical distribution of job-level callback gaps, properly rescaled \[\br{\frac{\sqrt{J_i}}{\sqrt{J_i^\data - 1}} (X_{ij}^{\data} -
X_i^\data) : j=1,\ldots, J_i^\data}.\]
We then let
$\pistar$ be the empirical distribution of $X_i^{\data}$, rescaled around $X_i^{\data}$'s
empirical mean to match the estimated signal variance $\var_n(X_i^\data) - (1-1/n)\E_n [\hat
\sigma_i^2]$. In \citet{kline2022systemic}, over $n=108$ firms, the harmonic mean of
$J_i^\data$ is about 87. In our simulations, we consider $J_i := \max(2, \lceil r
J_i^\data \rceil )$ for $r = 2^{-5}, 2^{-4},\ldots, 2^2$. \Cref{fig:kline_re} studies varying $J$ and $n$ in this setting.

\begin{figure}[tbh]
	\centering
	\includegraphics[width=\linewidth]{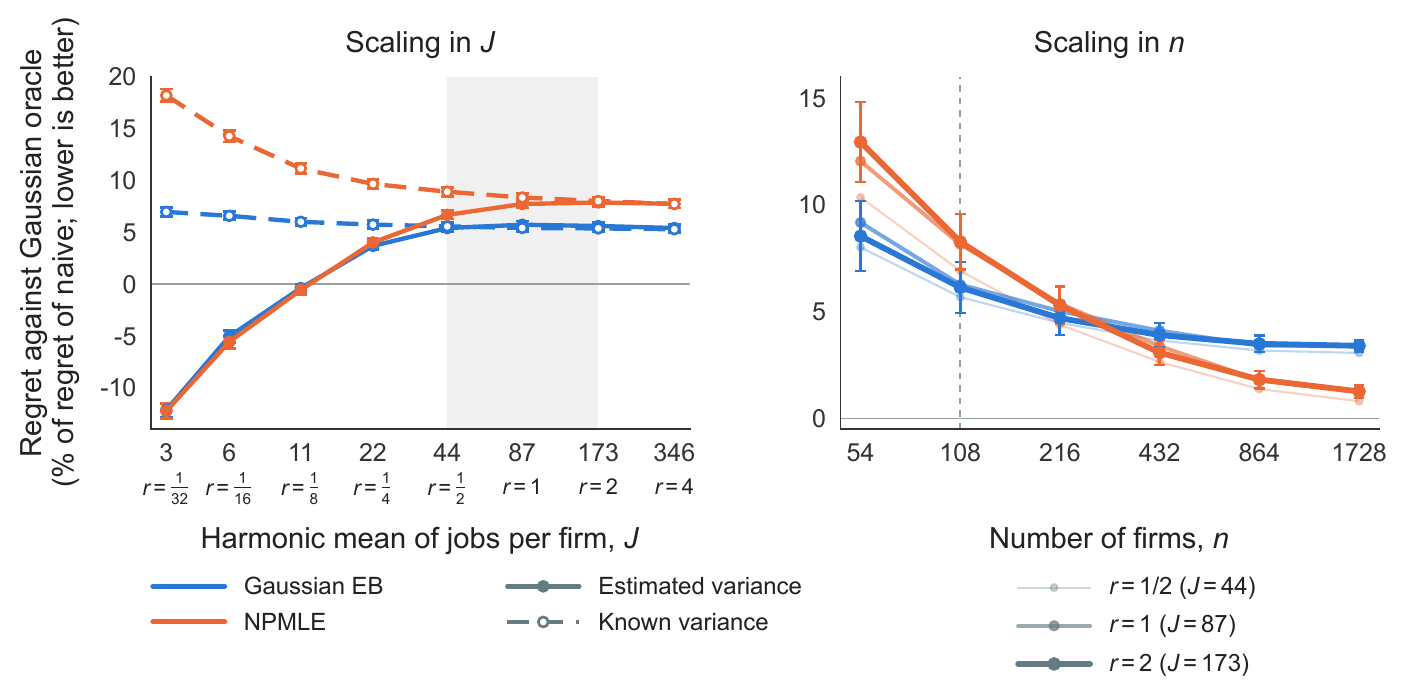}

	\begin{proof}[Notes]
  There are $10^4$ Monte Carlo draws for the left panel and $10^3$ for the right panel, with $\pm 2$ Monte Carlo standard errors shown (though barely visible on the left; Monte Carlo standard errors are only shown with $r=2$ in the right panel).
	\end{proof}

	\caption{Simulations calibrated to \citet{kline2022systemic}}
	\label{fig:kline_re}
\end{figure}

We find in \cref{fig:kline_re} (left panel) that doubling or halving $J$ (i.e.,
$r\in\br{1/2, 1, 2}$) around the observed $J=87$ barely changes
regret.\footnote{Curiously, potentially due to high excess kurtosis of the simulated law
of $\varepsilon_{ij}/\sqrt{J_i}$ \`a la \cref{item:kurtosis} above, NPMLE with
estimated variance generally performs \emph{better} than its oracle-$\sigma_i$
counterpart.} Here, due to small $n=108$, it seems that NPMLE does not perform as well as
more parametric EB methods. To assess the effect of scaling $n$, \cref{fig:kline_re}
(right panel) plots the behavior of regret in $n$ for $r = \br{1/2, 1, 2}$. Regret
decreases much more as we scale $n$, while the three regret curves across $r$ are similar.
These results suggest that, empirically, the $1/n$ term for squared regret dominates at
empirically relevant values of $(n, J)$ in \citet{kline2022systemic}.

To be sure, these simulations impose precision independence \citep{chen2024empirical}. To
examine sensitivity to this assumption, \cref{fig:kline_fe} considers a ``fixed-effect''
setting in which $\theta_i$ are fixed at a re-scaled version of $X_i^\data$ and only
$\varepsilon_{ij}$ are resampled. We compare EB methods, including some that adjust for
precision dependence, and obtain similar qualitative conclusions. A version of the
nonparametric EB procedure in \citet{kline2024discrimination} performs best.

\subsection{Empirical results with linear measurements}

We next turn to the Veneto Workers Histories (VWH) dataset, a matched employer-employee
dataset in Italy spanning 1984--2001. Following \citet{kline2020leave} in sample construction and specification, we use a two-period AKM model for 1999 and 2001 \citep{abowd1999high}.   Let \(Y_{wt}\) denote
worker \(w\)'s log wages and let \(i_t(w)\) denote the worker's employer
in period $t \in \br{1,2}$. We consider the following model with a time fixed effect: \[
Y_{wt} =\alpha_w+\theta_{i_t(w)} +\gamma\one\{t=2\}+\epsilon_{wt}. \numberthis \label{eq:akm}\]
We denote the number of firms by $n$ and the number of workers by $W$.
The OLS estimator for $\theta_{i_t(w)}$ (throughout, we impose the convention that the estimated firm effects sum to zero) can be written as\footnote{This arises from observing that estimating $\theta_{i}$ is equivalent to estimating the regression \[
Y_{w2} - Y_{w1} = \gamma + \theta_{i_2(w)} - \theta_{i_1(w)} + \epsilon_{w2} - \epsilon_{w1}.
\]
This way of writing $X^\vwh$ implicitly considers $Y_w$ as independently distributed. If we are furthermore willing to assume $Y_{wt}$ are independent across workers and time, we could represent $X^\vwh$ differently, which would lead to doubling $J_A$.
} \[
X^{\vwh} = \underbrace{(B_\vwh'M_\vwh B_\vwh)^+ B_\vwh'M_\vwh}_{A^\vwh \in \R^{n \times W}} Y^{\vwh}, \quad Y_w^{\vwh} = Y_{w2} - Y_{w1}, \quad
\]
for a matrix $B_\vwh \in \R^{W \times n}$ with entries $B_{wi} = \one(i_2(w) = i) - \one(i_1(w) = i)$ and a projection matrix $M_\vwh = I_W- 1_W1_W'/W \in \R^{W\times W}$.

\subsubsection{Key summary statistics}

\begin{table}[tbh]
	\caption{Statistics on the VWH dataset (real data)}
	\label{tab:vhw_summary_stats}
	\centering

	\begin{tabularx}{\linewidth}{@{}>{\raggedright\arraybackslash}Xcc@{}}\toprule
    & Main sample & Leave-one-worker-out \\\midrule
Number of firms $n$ & 73,928 & 42,488 \\
Total number of workers $W$ & 747,190 & 659,983 \\
\quad Number of \emph{movers} ($i_2(w) \neq i_1(w)$) & 197,587 & 164,216 \\
\quad Number of \emph{stayers}  ($i_2(w) = i_1(w)$) & 549,603 & 495,767 \\
Average number of movers per firm & 2.67 & 3.86 \\
Mean (standard deviation) of $Y_{w2} - Y_{w1}$ & 0.032 (0.22) & 0.033 (0.21) \\
\quad Among movers & 0.030 (0.31) & 0.034 (0.29) \\
\quad Among stayers & 0.033 (0.18) & 0.033 (0.18) \\
\midrule
\textbf{Effective number of firms $n_{*, A}$} & \textbf{4,127.3} & \textbf{2,662.4} \\
\textbf{Effective number of workers $J_A$} & \textbf{2.53} & \textbf{4.59} \\
\bottomrule
\end{tabularx}

\vspace{0.5em}
	\begin{proof}[Notes]
	Our main sample is the largest set of firms connected by worker mobility, as in \citet{kline2020leave}. We also show results for the largest set of firms that remains connected after removing any single worker.
  \citet[Table I]{kline2020leave} report 73,933 firms and 197,572 movers in the largest connected set (our main sample) and 42,489 firms and 164,203 movers in the leave-one-worker-out sample. The small discrepancies appear to reflect how ties are broken when selecting each worker's highest-earning job in a given year.
	\end{proof}
\end{table}

\Cref{tab:vhw_summary_stats} shows some key statistics for these OLS estimators. In our main sample, there are
$n\approx 74000$ firms, but due to the high correlation implied by $A^\vwh$, the
effective sample size $n_{*, A}$ is only about $4100$, a reduction by a factor of 18. More
concerningly, were one to consider normality-based EB methods, the average
effective sample size within firm, $J_A$, is only $2.53$.\footnote{Here, each
observation is a time-difference of wages within worker. This differencing may contribute
to better approximate normality of the micro-data.} This is consistent with other
measures and discussions \citep[e.g.,][ Section 7.3]{kline2025branching}: The average number of movers per firm---of which the firm effect estimates are statistical
summaries---is similarly small at $2.67$. Thus, any optimism about the performance of normality-based EB methods here seems to rest on the hope that $Y_{w2}-Y_{w1}$ is approximately normal to begin with.\footnote{Restricting to a more connected component of firms---such as a component that remains connected upon removing any single worker---results in higher $J$ but lower $n$.}

\subsubsection{Calibrated simulation} To partially assess whether such optimism is justified, we again consider a calibrated
simulation. Since firm effects are essentially estimated from movers---those with $i_2(w)
\neq i_1(w)$---we calibrate the time effect $\gamma$ and the distribution of
$(Y_{w2}-Y_{w1}) - (\theta_{i_2(w)} - \theta_{i_1(w)})$ from the stayers. Under the AKM
model, for a stayer $w$, $Y_{w2}-Y_{w1} =\gamma+\epsilon_{w2}-\epsilon_{w1}. $ We thus:
\begin{itemize}
    \item
    Set
    $\gamma$ to be the average $Y_{w2}-Y_{w1}$ among stayers,
   \item Set $P_\epsilon$ to be the
    empirical distribution of $Y_{w2}-Y_{w1} - \gamma$ among stayers,
    \item Set $\tau^2$ to be the variance of $P_\epsilon$, and
    \item Set $P_{\epsilon}^{0, h}$ to be a smoothed version of $P_\epsilon$ by convolving $P_\epsilon$ with $\Norm(0, \tau^2 h^2)$ and renormalizing so that its mean is zero and its variance is $1$. We choose $h = 0.05$ throughout.
\end{itemize}
We consider the regression \[
Y_{w2} - Y_{w1} - \gamma = \theta_{i_2(w)} - \theta_{i_1(w)} + \epsilon_{w}
\]
among the movers, treating $\gamma$ as known. This corresponds to a version of \eqref{eq:akm} where we estimate $\gamma$ via the stayers and separately estimate $\theta$ via the movers. Let the OLS estimator in this regression be \[
X^\data = A \cdot (Y_{w2} - Y_{w1} - \gamma)_{w \text{ a mover}}, \quad \sigma_i^2 = \tau^2 (AA')_{ii},
\]
for some matrix $A$ that depends on the graph of movers and firms. We calibrate $\pistar$ to the NPMLE on $(X^\data, \br{\sigma_i^2}_{i=1}^n)$, centered to have mean zero.

\begin{figure}[tbh]
	\centering
	\includegraphics[width=\linewidth]{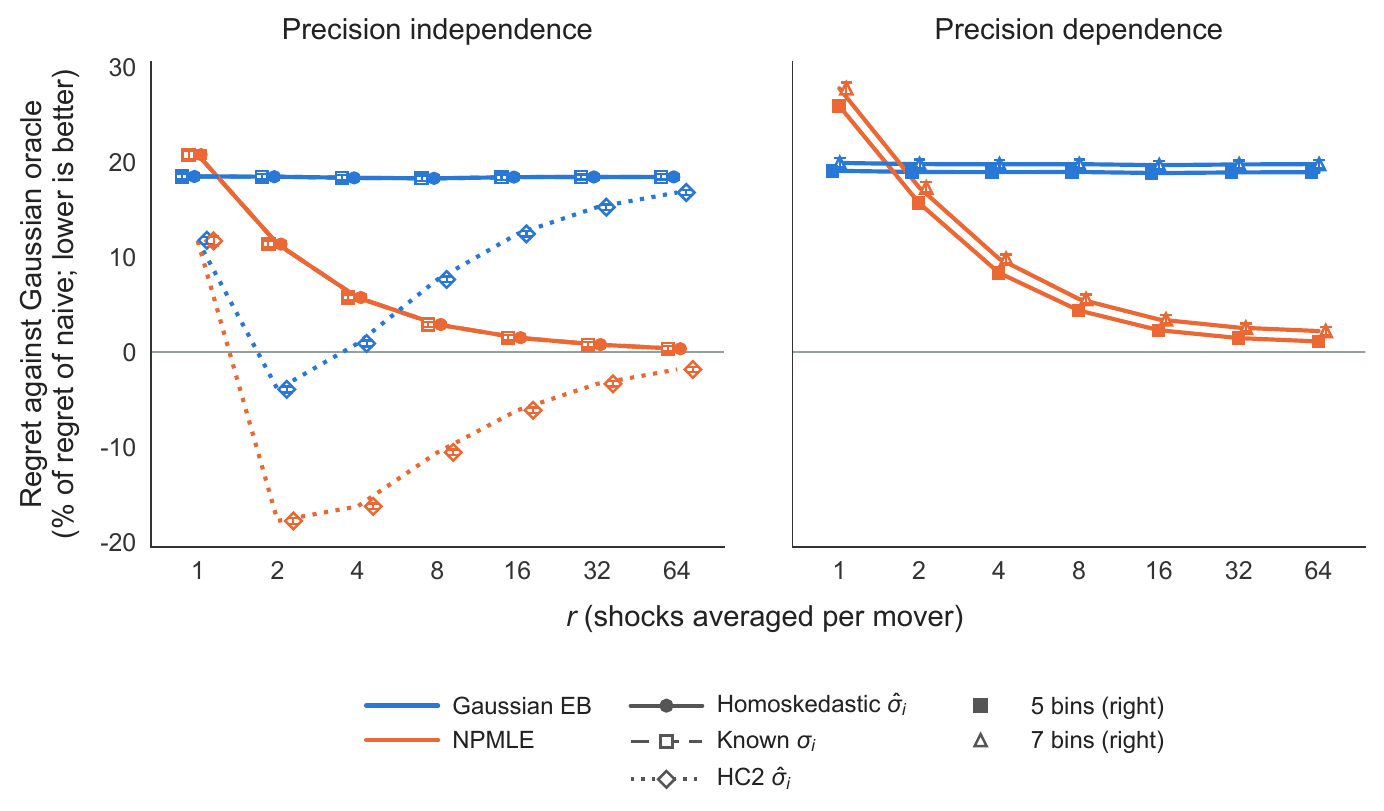}
	\caption{Simulations calibrated to the VWH dataset}

	\begin{proof}[Notes]
  Simulations use 100 Monte Carlo replications. Error bars show $\pm 2$ Monte Carlo standard errors.
  The right panel compares estimators fitted within five or seven bins on $\hat\sigma_i^2$, using estimated homoskedastic variances for $\hat\sigma_i^2$.
	\end{proof}
	\label{fig:veneto}
\end{figure}

For each simulation draw, we fix some value of $r \in \N$, where increasing $r$ increases the degree of normality. We draw $\theta_i \iid \pistar$
and $\eta_{wk} \iid P_{\epsilon}^{0,h}$ for $k = 1,\ldots, r$, and form \[ \xi_w^{(r)}
=\frac{\tau}{\sqrt r}\sum_{k=1}^{r}\eta_{wk}, \quad Y^{(r)}_w=\theta_{i_2(w)} -
\theta_{i_1(w)} +\xi^{(r)}, \quad X^{(r)}=AY^{(r)}. \] For every \(r\), $
\cov(X^{(r)}\mid\theta)=\tau^2AA', $ while the effective number of primitive shocks
$\eta_{wk}$ increases from \(J_A\) to \(rJ_A\). Thus, varying \(r\) improves the normal
approximation without changing $\cov(X^{(r)} \mid \theta)$. We then treat $X^{(r)}, A,
Y_w^{(r)}$ as data and feed them to various EB methods.\footnote{To caveat, $X^{(r)}$ sums to zero
by construction due to the normalization chosen for $A$. We do not constrain the simulated
true firm effects to sum to zero in sample (only requiring $\pistar$ to be mean zero).
Thus, technically, $X_i^{(r)}$ is a slightly biased estimator for $\theta_i$, with
conditional bias $-\frac{1}{n}\sum_i \theta_i \approx 0$, which we ignore throughout.}

\Cref{fig:veneto} (left panel) compares NPMLE and parametric Gaussian EB in this setting,
primarily using estimated homoskedastic variance (that is, we estimate $\tau^2$ via
the degrees-of-freedom-corrected sample variance of the regression residuals $\hat \xi_w^{(r)}$). At the original
sample size $r=1$, Gaussian EB performs slightly better than NPMLE. As \(r\) increases,
however, NPMLE regret declines rapidly (empirically on the order of $1/r \ll 1/\sqrt{r}$),
while regret of Gaussian EB changes little. Doubling $J$ ($r=2$) would lead to NPMLE
outperforming Gaussian EB.

We also consider a few robustness checks and variants. First, we compare the effect of
variance estimation on performance. Methods with estimated homoskedastic variance perform
virtually identically to their oracle-$\tau^2$ counterparts, indicating that $\hat\tau^2$ is
 precise. Heteroskedasticity-robust estimators for $\sigma_i^2$, however, appear
to substantially \emph{help} the methods when $J$ is small, potentially because of the
mechanism we speculate about in \cref{item:kurtosis}.\footnote{See \cref{sub:VWH_details} for details on how these HC2 standard errors are constructed.} Second, \cref{fig:veneto} (right panel)
considers a version of this simulation with built-in precision dependence: We stratify
$X^\data$ into 5 bins by $\sigma_i^2$ and calibrate a prior by NPMLE within each bin.
To accommodate precision dependence, different EB methods then similarly bin on $\hat\sigma_i^2$ and shrink within each bin.
The qualitative conclusions do not appear to differ in this setting.

Lastly, we also consider \emph{branching}, a sample-splitting scheme proposed by
\citet{kline2025branching}. Branching produces a subset of firms $S \subset [n]$ and two (or more)
disjoint subsets of movers $(M_1, M_2)$ such that firm effects for those in $S$ can be
estimated separately using either set of movers. This produces independent firm effect
estimates $X_1^S$ and $X_2^S$. Estimating $\E[X_1^{S} \mid X_2^S]$ (and the reverse conditional expectation) then generates a
denoised estimate without normal approximations. This method is potentially attractive in
settings with poor normal approximation. However, branching removes movers who move across
$S$ and $S^C$, whose wages could be informative for those firms inside $S$. This appears
costly: \cref{fig:branching} shows that branching performs worse than using the OLS estimates $X$ estimated on the full set of movers.

\subsubsection{Empirical application}
Branching also lets us compare prediction methods on real data, using held-out wage changes. For an $S$ selected by branching, we can create three disjoint sets
of movers $(M_1, M_2, M_3)$ such that all firm effects in $S$ can be estimated separately using
$M_1$ or $M_2$. We then treat $M_3$ as hold-out data and $M_1 \cup M_2$ as training data
(the union ensures that branching can be used to split the training data). Different
procedures output predictions for firm effects $\hat\theta_i$, which we evaluate using the hold-out set $M_3$ by predicting \[
\widehat{\Delta Y}_w = \hat\gamma_{\text{stayer}} + \hat\theta_{i_2(w)} - \hat\theta_{i_1(w)}
\]
for their change in wages $Y_{w2} - Y_{w1}$. Here, $\hat\gamma_{\text{stayer}}$ is the average $Y_{w2}-Y_{w1}$ among stayers, which we treat as fixed and use as an estimate of the time fixed effect throughout. If the AKM model is correctly specified, then better predictions of $\theta_i$ translate into better predictions of wage changes on held-out movers.\footnote{Throughout this exercise we treat $\gamma$ as known and equal to $\hat\gamma_{\text{stayer}}$. Thus on the training set we regress $Y_{w2}-Y_{w1} - \hat\gamma_{\text{stayer}}$ on destination-minus-origin firm indicators among the movers.}

\begin{table}[tbh]
    \centering
    \caption{Prediction of held-out wage changes in the VWH sample (real data)}
    \label{tab:veneto_sample_splitting}
    \begin{tabular}{lrrr}
    \toprule
    Method & \shortstack[r]{Prediction\\MSE} & \shortstack[r]{MSE reduction\\from naive (\%)} & \shortstack[r]{Forecast bias\\coef.} \\
    \midrule
    Naive ($\hat\theta_i = X_i$) & 0.0774 & 0 & 0.570 \\
    Gaussian EB & 0.0672 & 13.1 & 0.805 \\
    NPMLE & 0.0670 & 13.4 & 0.802 \\
    EB-URE \citep{cheng2025optimal} & 0.0651 & 15.9 & 0.845 \\
    Branching & 0.0696 & 9.97 & 1.73 \\
    Branching (balanced pairs) & 0.0671 & 13.2 & 1.03 \\
    \midrule
    \multicolumn{3}{l}{Mean effective number of workers $J_A$ in training sets $M_1\cup M_2$} & 5.9 \\
    \multicolumn{3}{l}{Mean effective number of firms $n_{*,A}$ in training sets $M_1\cup M_2$} & 1,447 \\
    \bottomrule
\end{tabular}

    \begin{proof}[Notes]
      $S$ contains 21,565 firms selected by a procedure following \citet{kline2025branching}. These results average over 5 different choices of $(M_1\cup M_2, M_3).$ The forecast bias coefficient is the slope coefficient in a regression of $Y_{w2} - Y_{w1}$ on $\widehat {\Delta Y}_w$ and a constant, among movers in $M_3$. If $\widehat {\Delta Y}_w$ were the posterior mean of $Y_{w2} - Y_{w1}$, then we should expect this coefficient to be one.

      Both Gaussian EB and NPMLE use estimated homoskedastic variances and are fit separately in 5 $\hat\sigma_i^2$ bins to account for precision dependence. EB-URE follows the procedure in \citet{cheng2025optimal} designed to optimize an estimated MSE for shrinkage decisions.
      Both branching variants account for precision dependence by including controls for precision. Branching (balanced pairs) splits the training set evenly into branches compared to the original procedure described in \citet[p.~15]{kline2025branching}. See \cref{sub:VWH_details} for details on construction of $S$, branching, and \citet{cheng2025optimal}.
    \end{proof}
\end{table}

\Cref{tab:veneto_sample_splitting} compares different methods by the mean-squared error of predicting $Y_{w2} - Y_{w1}$ and the extent to which
$\widehat {\Delta Y}$ is calibrated for $Y_{w2} - Y_{w1}$. We find that the NPMLE performs surprisingly well even though $J$ is small. Its prediction MSE is nearly identical to those of Gaussian EB and balanced-pair branching, and lower than that of the original branching procedure (though balanced-pair branching performs substantially better on calibration). The good performance of NPMLE may be partly explained by the fact that the sample-split procedure chooses a
well-connected subgraph $S$ of firms and movers to begin with, and so the effective sample size here ($J_A
\approx 5.9$) is considerably larger than $J_A \approx 2.53$ in the full sample.
This perhaps makes the normal approximation for $X$ more plausible.
Finally, the results also suggest that the worker-firm graph contains useful information beyond the precision of each \(X_i\). The EB-URE estimator of \citet{cheng2025optimal} additionally uses the structure of the matches to shrink worker and firm effects jointly and achieves the lowest prediction MSE.

Overall, these simulations and empirical exercises in the VWH data suggest that we should be concerned about the quality of the normal approximation in OLS estimates of firm effects. For instance, improving the normal approximation quality substantially reduces the regret of NPMLE in calibrated simulations. Nevertheless, NPMLE delivers competitive predictions of held-out data, when restricted to a more connected subset of the firm graph. Surprisingly, for predicting $\theta_i$, we find that parametric EB methods or sample splitting methods---which are ostensibly more robust to the normal approximation---offer relatively modest improvement over the NPMLE, if at all.

\section{Conclusion}

This paper quantifies how the accuracy of the CLT approximation affects nonparametric EB methods that assume exact normality. In two leading regimes, we find that theoretical guarantees for nonparametric EB methods deteriorate gradually as the noisy estimates depart from exact normality. We view these results as reassuring for applications of nonparametric EB under approximate normality.

As a practical payoff, our results identify sample-size-type quantities that can and
should be reported alongside an empirical Bayes analysis. In the independent measurements
regime, these are the number of parallel settings ($n$) and the harmonic mean of their
micro-data sample sizes ($J$). In the linear measurements regime, these additionally account
for the fact that all estimates $X_i$ depend on a common random vector $Y$. These quantities summarize the information available for learning the prior $\pistar$, as well as the micro-data sample sizes governing the normal approximation quality.

Empirically, our calibrated simulations suggest that the normal approximation is adequate
in the firm discrimination context in \citet{kline2024discrimination}, while its quality
matters substantially more for estimating firm wage premia in the VWH dataset---where
effective $J$ is small. Nevertheless, perhaps surprisingly, on a more connected subset of VWH data, the NPMLE delivers competitive predictions of held-out wage changes relative to
alternatives that appear more robust to departures from normality.

\bibliographystyle{aer}
\bibliography{References}

\clearpage
\appendix

\begin{center}
  \large \bfseries Online appendix for ``Normal approximations in nonparametric empirical Bayes''

  \normalsize \normalfont

\vspace{1em}

  Jiafeng Chen, Stanford University

  Nabarun Deb, University of Chicago Booth School of Business

  Nikolaos Ignatiadis, University of Chicago

  \vspace{1em}

    \today
\end{center}

\startcontents
{\footnotesize\printcontents{}{0}{\textbf{Contents}\vskip1em\hrule\vskip1em}}
\vskip1em\hrule\vskip5pt

\section*{AI Disclosure} \label[aidisclosure]{sec:ai_disclosure}

We began working on this
project in the fall of 2024. By April 2026, we had shown most of our workhorse results that covered \cref{thm:main},  \cref{prop:hellinger-bound}, and the results in \cref{sec:deconvolution,sec:general}. These do not involve AI supplying proof ideas and formal statements. In May 2026,
GPT-5.5 made the observation that differences in log-likelihood ratios under $\mu_i$ and
$f_{\pistar, \sigma}$ can be controlled via Stein's method, which yields fast rates under our localization argument in \cref{thm:main}. This led to a version of
\cref{thm:main_rates} in the independent measurement regime, but with known $\sigma_i^2$,
which we included in a first draft of this paper posted on arXiv (2605.31599v1 on May 29,
2026). Subsequent AI models (we mostly used OpenAI's GPT) have exhibited impressive
and rapidly improving abilities to write proofs. They have played a larger role 
in our subsequent results on the independent
measurement regime with unknown variances and on the linear measurements regime. We also used GPT Astra and Claude Opus 5.5 heavily for implementing the empirical exercises.

The paper itself is entirely written by the authors.

\section{Sketch of extension to asymptotically linear estimators}
\label{asec:ral}

We sketch an approach for extending the independent measurements regime to asymptotically linear estimators. Following \cref{rmk:local_asymp}, consider the local parametrization:\[
h_i \iid \pistar \quad \theta_i = \theta_0 + \frac{h_i}{\sqrt{\bar J}}.
\]
Let $Y_{ij} \mid (h_i, \alpha_i) \iid P_{\alpha_i, \theta_i}$ denote the micro-data. We treat $(\alpha_{1:n}, J_{1:n})$ as fixed throughout this section.

We observe some $\tilde X_i$, an asymptotically linear estimator of $h_i$ using $\br{Y_{ij}}_{j=1}^{J_i}$: \[
\tilde X_i = \bar J^{1/2} (\hat\theta_i - \theta_0) = h_i + \frac{\sqrt{\bar J}}{J_i} \sum_{j=1}^{J_i} \psi_{\alpha_i, \theta_i}(Y_{ij}) + R_i,
\]
where $\psi_{\alpha_i, \theta_i}(\cdot)$ is an influence function and $R_i$ captures all terms that are not first order. We let $\tilde \sigma_i^2$ be the (clipped) feasible estimated variance for $\tilde X_i$. Typically, $\tilde \sigma_i^2$ is the clipped version of the sample variance of some $
\br{\sqrt{\bar J/J_i} \hat\psi_i(Y_{ij}): j=1,\ldots,J_i}
$
for some estimate $\hat\psi_i$ of the influence functions.
Let \[
\sigma_{i,*}^2(h_i) := \var\bk{
\frac{\sqrt{\bar J}}{J_i} \sum_{j=1}^{J_i} \psi_{\alpha_i, \theta_i}(Y_{ij}) \mid h_i
}
\]
be the oracle asymptotic variance. We consider the NPMLE $\tilde \pi_n$ on $(\tilde X_i, \tilde \sigma_i^2)$, under the working model that $\tilde X_i \sim \Norm(h_i, \tilde \sigma_i^2)$.

Before analyzing the NPMLE, there is a conceptual difficulty related to $\sigma_{i,*}^2(h_i)$ being a deterministic function of $h_i$. Because of this, it can be awkward to impose \eqref{eq:true_prior}: $\sigma_{i,*}^2(h_i)$ should mechanically predict $h_i$. As an example, consider $Y_{ij} \mid (h_i, \alpha_i) \iid \Bern(\theta_i)$ and let $\tilde X_i = \bar J^{1/2}(\bar Y_i - \theta_0)$ be the rescaled MLE for $\theta_i$. The influence function expansion for $\tilde X_i$ is \[
\tilde X_i = h_i + \frac{\sqrt{\bar J}}{J_i} \sum_{j=1}^{J_i}  \underbrace{(Y_{ij}-\theta_i)}_{\psi_{\alpha_i, \theta_i}(Y_{ij})} \quad \sigma_{i,*}^2(h_i) = \frac{\bar J}{J_i} \theta_i (1-\theta_i)
\]
with $R_i = 0$. Here, the estimated variance $\tilde \sigma_i^2$ is the clipped version of $\frac{\bar J}{J_i} \bar Y_i (1-\bar Y_i)$, taking $\hat \psi_i(Y):= Y - \bar Y_i$. Note that $\theta_i$ and $\sigma_{i,*}^2(h_i)$ are mechanically correlated, making it incoherent to assume \eqref{eq:true_prior}.

We avoid this conceptual difficulty by following standard practice when analyzing limiting experiments and assuming \eqref{eq:true_prior} with respect to the limiting asymptotic variances. Specifically, we let
\[
\sigma_i^2 := \var_{P_{\alpha_i, \theta_0}}\bk{
\frac{\sqrt{\bar J}}{J_i} \sum_{j=1}^{J_i} \psi_{\alpha_i, \theta_0}(Y_{ij})
}
\]
be the oracle asymptotic variance evaluated at $\theta_i = \theta_0$.  We will invoke \cref{thm:main} with respect to those variances.

In the binomial example above, $\sigma_i^2 = \frac{\bar J}{J_i}\theta_0(1-\theta_0)$.  The assumption \[
h_i \mid \sigma_i^2 \iid  \pistar \numberthis \label{eq:precision_indep_RAL}
\]
then automatically holds since we already assumed $h_i \iid \pistar$ fixing $(J_i, \alpha_i)$.

To analyze $\tilde \pi_n$, it is helpful to define an oracle version of $\tilde X_i$ to which our results apply.
Consider \[ X_i := h_i + \frac{\sigma_i}{\sigma_{i,*}(h_i)} \frac{\sqrt{\bar J}}{J_i}
\sum_{j=1}^{J_i} \psi_{\alpha_i, \theta_i}(Y_{ij}) \numberthis \label{eq:RAL_oracleX} \]
the version of $\tilde X_i$ where we discard the linearization error $R_i$ and rescale so
that $\var(X_i \mid h_i) = \sigma_i^2$. We can then construct $\hat \sigma_i^2$ as the
clipped version of the sample variance of the summands in $X_i$: \[ \br{ \sqrt{\bar J/J_i}
\frac{\sigma_i}{\sigma_{i,*}(h_i)} \psi_{\alpha_i, \theta_i}(Y_{ij}) : j=1,\ldots, J_i}.
\] Note that $(X_i, \hat\sigma_i^2)_{i=1}^n$ fits the independent measurements regime in
\cref{sec:indep_measurements}. Upon imposing conditions on sub-Gaussianity of
$\psi_{\alpha_i, \theta_i}(Y_{ij})$, \cref{thm:main} then shows that any approximate maximum
likelihood $\hat\pi_n$ for $(X_i, \hat\sigma_i^2)$ has $\frac{1}{n} + \frac{1}{J}$
convergence in average squared Hellinger distance, up to logs. In the binomial example,\footnote{Interestingly,
this choice of $X_i$ has a connection to variance-stabilizing transforms. Let $g(p) = 2\arcsin\sqrt{p}$ and $\tilde X^{\VST}_i := \sqrt{\bar J}(g(\bar Y_i) - g(\theta_0))$ be the rescaled estimate for $g(\theta_i)$ under a variance-stabilizing transform. Then the influence function expansion for $\tilde X^\VST$ satisfies (upon linearizing $g(\theta_i)-g(\theta_0)$), \[
\sqrt{\theta_0(1-\theta_0)} \tilde X_i^{\VST} = X_i + R_i^\VST \quad R_i^\VST \lesssim \bar J^{1/2} (\bar Y_i - \theta_0)^2 + \bar J^{-1/2} h_i^2,
\]
assuming $\theta_i, \theta_0 \in [\epsilon, 1-\epsilon]$ for some $\epsilon > 0$. Thus, for the variance stabilizing transform, $X_i$ is directly the asymptotically linear term without rescaling.
} \[
X_i = h_i + \frac{\sqrt{\bar J}}{J_i}
\sum_{j=1}^{J_i} \sqrt{\frac{\theta_0(1-\theta_0)}{\theta_i(1-\theta_i)}}  (Y_{ij} - \theta_i).
\]

It then suffices to show that $\tilde \pi_n$ is such an approximate MLE on infeasible data $(X_i, \hat\sigma_i^2)$:
\[
 \avgn\log\frac{f_{\tilde\pi_n,\hat\sigma_i}(X_i)}{f_{\pistar,\hat\sigma_i}(X_i)}
 \ge -\eta \Helbarstar{f_{\tilde \pi_n}}{f_{\pistar}}-C_\eta \polylog(n)(n^{-1}+\bar J^{-1})
\]
with probability at least $1-O(n^{-2})$ for sufficiently small $\eta > 0$. Using existing arguments to transfer from $\hat\sigma_i$ to $\sigma_i$ (\cref{prop:estimated-variance-oracle-likelihood}), this would then verify \cref{def:approximate_npmle} for $\tilde \pi_n$ and allow us to directly invoke \cref{thm:main}. Thus, we can reduce the Hellinger rate for $\tilde \pi_n$ to proving the following:
For every sufficiently small fixed $\eta>0$,
at $\pi=\tilde\pi_n$ and with probability at least $1-Cn^{-2}$,
\[
 \avgn\bigg[
 \overbrace{ \log\frac{f_{\tilde\pi_n,\tilde\sigma_i}(\tilde X_i)}{f_{\pistar,\tilde\sigma_i}(\tilde X_i)}}^{\ge 0}
  -\log\frac{f_{\tilde\pi_n,\hat\sigma_i}(X_i)}{f_{\pistar,\hat\sigma_i}(X_i)}
 \bigg]
 \le \eta\Helbarstar{f_{\tilde \pi_n}}{f_{\pistar}}+C_\eta \polylog(n)(n^{-1}+\bar J^{-1}) \numberthis \label{eq:sufficient_condition_RAL}.
\]
A sufficient condition for \eqref{eq:sufficient_condition_RAL}, which is satisfied for the binomial model, is the following:
\begin{assume}
Assume $\psi_{\alpha_i, \theta_i}(Y_{ij})$ is uniformly sub-Gaussian. Further assume that with probability at least $1-C n^{-2}$, uniformly for all $i$, \begin{enumerate}
    \item $|\sigma_{i,*}^2(h_i) - \sigma_i^2| \lesssim J_i^{-1/2}.$
    \item  The remainder $|R_i| \lesssim \polylog(n) J_i^{-1/2}$
    \item There is an influence function $\zeta_{\alpha_i, \theta_i}$ with $\E[\zeta_{\alpha_i, \theta_i}(Y_{ij}) \mid h_i] = 0$ for the variance estimator that is uniformly sub-exponential and satisfies
    \[
\abs[\bigg]{\tilde \sigma_i^2 - \sigma_{i,*}^2(h_i) - \frac{1}{J_i} \sum_{j=1}^{J_i} \zeta_{\alpha_i, \theta_i}(Y_{ij})} \lesssim \frac{\polylog(n)}{J_i}.
    \]
\end{enumerate}
\end{assume}
	\section{Proofs of results in Section~\ref{sec:main-res}}\label{sec:pfmain}
    \subsection{Proof of \cref{thm:main}}
\paragraph{Proof outline}    The proof of \cref{thm:main} follows a localization argument
for proving convergence rates of $M$-estimators. We define the following stochastic
process of truncated likelihood ratios. For some $t > 0$ sufficiently large, let
    \[
    S_{n,t}:=[-M-2t\sqrt{\log n},M+2t\sqrt{\log n}],
    \quad
    A_{nt}(\pi):=\frac{1}{n}\sum_{i=1}^n\mathbf{1}(X_i\in S_{n,t})
    \log\frac{f_{\pi,\sigma_i}(X_i)}{f_{\pistar,\sigma_i}(X_i)}.
    \]
    Truncating to $S_{n,t}$ does not bind with high probability, and thus we ignore it for
    the rest of the outline.

    Fix some candidate sequence of priors $\pi$, omitting the $n$ subscript, and let
    $\delta^2 = \Helbarstar
    {f_ {\pi}}{f_ {\pistar}}$. If $\pi$ is an approximate NPMLE in the sense
    of
    \eqref{eq:likelihood-bound}, then with high probability \[ A_{nt}(\pi) \ge -\eta
    \delta^2 -  C_\eta b_n.
    \]
    That is, $\pi$'s suboptimality in terms of the \emph{sample log-likelihood} is
    bounded.
    On the other hand, if $X_i \sim \mu_i$ are approximate normal convolutions, then
    \eqref{eq:condition-a} implies \[
    	\E A_{nt}(\pi) \le c_1 \pr{-\delta^2 + \bar r_1}.
    \]
    In other words, $\pi$'s suboptimality in terms of the \emph{population
    log-likelihood} is large if the Hellinger distance $\delta^2$ is large.
    Combining the previous two displays and choosing some fixed $0 < \eta < c_1/2$
    implies a lower bound on the centered likelihood ratio process \[
    	A_{nt}(\pi) - \E A_{nt}(\pi) \ge c_1 \delta^2 - C b_n - c_1 \bar r_1. \numberthis
    	\label{eq:deviation}
    \]
    The key input is then the following local maximal inequality, which bounds the
    probability of
    \eqref{eq:deviation} for large $\delta^2$. \Cref{prop:local-maximal} is proved by
    upgrading a pointwise Bernstein inequality to be uniform over a local neighborhood,
    via a standard entropy calculation as in \citet{zhang2009generalized}. It makes use of
    the fact that the bound in \eqref{eq:condition-b} of \cref{asn:assumeindep} adapts to
    small Hellinger distance between $f_\pi$ and $f_{\pistar}$.

    \begin{prop}[Local Bernstein-type maximal inequality  for $A_{nt}$]
    \label{prop:local-maximal}
    Assume the conditions of \cref{thm:main}.
    	Suppose every $\mu_i$ is an $(r_{1i}, r_{2i})$-approximate normal convolution in
    	the sense of \cref{asn:assumeindep} and that \cref{assump:lowerbound} holds. For
    	every fixed $c > 0$ and sufficiently large fixed $t > 0$, choose $c^* > 0$
    	sufficiently large depending on $(c,t)$. Then, there exists some $C$ that depends
    	on $c$ such that for every $r$ satisfying \begin{equation}
  c^*\left\{\frac{\log^4 n}{n}+\rbar{2}\right\}
  \leq r\leq1,
  \label{eq:local-radius}
\end{equation}
we have that, for any localized parameter space $\mathcal G(M, r) \subseteq \{\pi \in
\mathcal P[-M,M]:
\Helbarstar{f_\pi}{f_{\pistar}}\leq r\}$, \[
	\Prob \br{
		\sup_{\pi \in \mathcal G(M,r)} A_{nt}(\pi) - \E A_{nt}(\pi) \ge cr
	} \lesssim \exp\pr{-C \frac{nr}{\log^2 n}} \numberthis \label{eq:local-maximal}.
\]
    \end{prop}

    \Cref{prop:local-maximal} rules out large $\delta^2$ by showing that it leads to the
    unlikely event in \eqref{eq:deviation}. Suppose it were the case that $\delta^2 = o(1)$ but
    $\delta^2 \gg
  \frac{\log^4 n} {n} + b_n + \bar r_1 + \bar r_2$; then for some $r$ satisfying
  \eqref{eq:local-radius} and $r > \Helbarstar
    {f_ {\pi}}{f_ {\pistar}}$, \[
  	c_1 \delta^2 - C b_n - c_1 \bar r_1 \ge c r \text{ and } \pi\in \mathcal G(M, r) =
  	\{\pi \in
\mathcal P[-M,M],
\Helbarstar{f_\pi}{f_{\pistar}}\leq r\}
  \]
  Thus, by the local maximal inequality \eqref{eq:local-maximal}, the probability of
  \eqref{eq:deviation} is bounded by $n^{-C}$. This shows that $\delta^2$ is unlikely to
  be larger than the target rate $\frac{\log^4 n} {n} + b_n + \bar r_1 + \bar r_2$.   The
  above outline treats the random Hellinger distance of \(\hat\pi_n\) as though it
  is a fixed sequence $\delta^2$. The proof makes this argument precise through a standard
  \emph{peeling argument} \citep{geer2000empirical}---partitioning the possible squared
  Hellinger distances into deterministic dyadic intervals.

We start by proving \cref{prop:local-maximal}. First we state some standard calculations
of Gaussian mixture densities.

	\begin{lemma}\label{lem:log-mixture-bound}
		For any $\pi \in \cP([-M,M])$, we have
		\[
		\left| \log \left(  \int \phi\!\left(x-\theta;{\sigma}\right) d\pi(\theta) \right)
		\right| \leq \left| \log(\sqrt{2\pi}\sigma) \right| + \frac{1}{2\sigma^2}(|x|+M)^2
		\quad \forall\; x \in \R.
		\]
	\end{lemma}

     \begin{lemma}\label{lem:sup-log-bound}
	For any $\pi \in \cP([-M,M])$ and any $c, T > 0$, we have
	\[
	\sup_{|x| \leq T} \log \left( 1 + \frac{2c}{\int \phi\!\left(x-\theta;{\sigma}\right) d\pi(\theta)} \right) \leq 2\sqrt{2\pi}\, c\sigma \exp\!\left(\tfrac{1}{\sigma^2}(T^2 + M^2)\right).
	\]
\end{lemma}

\begin{proof}[Proof of \cref{prop:local-maximal}]
Fix some $s\ge 1$ to be chosen. Define the following pseudo-metric over $\mathcal G
(M,r)$:\[ d_{\infty, S_{n,t}} (\pi_1, \pi_2) = \max_{1\le i\le n} \sup_{x\in
S_{n,t}} |\fmix{\pi_1}{\sigma_i}(x) - \fmix{\pi_2}{\sigma_i}(x)|.
\]
Consider an $n^{-s^2}$-cover of $\mathcal G(M, r)$ with centers $\pi_1, \ldots, \pi_N$ and
cardinality $N$. Note that $\mathcal G(M ,r) \subseteq \mathcal P[-M, M]$ and \citet[Lemma
4]{Jiang2020} bounds the covering number of $\mathcal P[-M, M]$. Applying
\citet{Jiang2020}'s result and the standard packing-covering inequalities
\citep[][(2.6)]{ghosal2001entropies} shows that there exists a constant $C$ depending on
$M, k, K, s$, and $t$ such that $
	\log N \le C \log^2 n.
$

For a given $\pi \in \mathcal G(M,r)$, if $\pi_j$ is closest to $\pi$ under $d_{\infty,
S_{n,t}} $, then for all $x \in S_{n,t}$, \[
	\abs{
		\log f_{\pi, \sigma_i}(x) - \log f_{\pi_j, \sigma_i}(x)}
		\leq
\log\left\{
  1+
  \frac{2n^{-s^2}}{
    \min\{f_{\pi,\sigma_i}(x),f_{\pi_{j},\sigma_i}(x)\}
  }
\right\}
	 \lesssim n^{-3}
\]
by \cref{lem:sup-log-bound}, \cref{assump:lowerbound}, and choosing a large $s$.
Therefore, for any $\pi \in \mathcal G(M ,r)$, there exists some $\pi_j$ such that $
	A_{nt}(\pi) - \E A_{nt}(\pi) \le A_{nt}(\pi_j) - \E A_{nt}(\pi_j) + O(n^{-3}).
$
Taking the supremum on both sides, \[
	\sup_{\pi \in \mathcal G(M ,r)} A_{nt}(\pi) - \E A_{nt}(\pi) \le \max_{j\in[N]} \br{A_
	{nt} (\pi_j) - \E A_{nt}(\pi_j) }+ O(n^{-3}). \numberthis
	\label{eq:maximal_intermediate}
\]
It suffices to analyze $A_ {nt} (\pi_j) - \E A_{nt}(\pi_j)$ and then apply a union bound
	over $j$. We do so by applying Bernstein's inequality.  Note that by
	\cref{lem:log-mixture-bound} and the definition of $S_{n,t}$, \[
		\max_{i\in [n]} \max_{j \in [N]} \sup_{x\in \R} \abs[\bigg]{\mathbf{1}(x\in S_{n,t}) \log
		\frac{f_ {\pi_j, \sigma_i} (x)} {f_ {\pistar,
		\sigma_i}(x)}} \lesssim t^2 \log n.
	\]
Hence the summands of $A_{nt}(\pi_j)$ are uniformly bounded by $Ct^2 \log n$. The average
second moment admits the following bound by \eqref{eq:condition-b} \[
	\frac{1}{n} \sum_{i=1}^n \E\bk{\mathbf{1}(X_i\in S_{n,t}) \log^2\abs[\bigg]{
		\frac{f_ {\pi_j, \sigma_i} (X_i)} {f_ {\pistar,
		\sigma_i}(X_i)}}
	} \lesssim g_T(r) + \bar r_2.
\]
Since $A_{nt}(\pi_j) - \E A_
{nt}(\pi_j)$ is a sample mean of independent bounded and centered random variables,
Bernstein's inequality gives: \begin{align*}
	\Prob\bk{
		A_{nt}(\pi_j) - \E A_{nt}(\pi_j) \ge cr
	} &\le 2\exp\pr{
		-C \frac{nr^2}{g_T(r) + \bar r_2 + r t^2 \log n}}.\end{align*}
Therefore, since \eqref{eq:local-radius} implies that $r
\log^2 n \gtrsim g_T (r) + \bar r_2
+
rt^2 \log n$, incrementing $c_*$ if necessary,
\begin{align*}
	\Prob\bk{
		\max_{j \in[N] } A_{nt}(\pi_j) - \E A_{nt}(\pi_j) \ge cr
	}&\le N \cdot \max_j \Prob\bk{
		A_{nt}(\pi_j) - \E A_{nt}(\pi_j) \ge cr
	} \\
    &\lesssim \exp\pr{-C \frac{nr}{ \log^2 n }}.
\end{align*}
Since $n^{-3} = o(r)$, applying the above bound to \eqref{eq:maximal_intermediate}
concludes the proof.
\end{proof}

 \begin{proof}[Proof of \cref{thm:main}]

 For
 $\ell \ge 0$, define a \emph{shell} by \begin{align*}
  \cF_\ell
  :=
  \biggl\{
    \fmix{\pi}{\cdot}:\pi\in\cP([-M,M]),\;
    2^\ell \delta_n
    \leq
    \pr{\Helbarstar{f_\pi}{f_{\pistar}}}^{1/2}
    \leq
    2^{\ell+1} \delta_n
  \biggr\}.
\end{align*}
Fix a sufficiently large constant $t>0$ so that \cref{prop:local-maximal}
applies and $\Prob\bk{
		\exists i : X_i \not \in S_{n,t}
	} \lesssim n^{-2}$ by \cref{assump:sub-Gauss}.
Work on the event $\mathcal E$ where (i) \eqref{eq:likelihood-bound} holds with some fixed
$0 < \eta < c_1/16$ and (ii) for all $i$, $X_i \in S_ {n,t}$.
By the choice of $t$,
$\Prob(\mathcal E^c)\lesssim n^{-2}$. On the event $\mathcal E$, if
$\fmix{\hat\pi_n}{\cdot} \in \mathcal F_\ell$, then
\[
A_{nt}(\hat\pi_n) \ge
-\eta
\Helbarstar{f_{\hat\pi_n}}{f_{\pistar}} - C_\eta b_n
. \numberthis
	\label{eq:sample_gap}
\]
On the other hand, for a fixed $\pi$ with
${f_{\pi,\cdot}\in\mathcal F_\ell}$, by applying
\eqref{eq:condition-a}
	\begin{align*}
	\E A_{nt}(\pi)
&{\le \frac{1}{n}\sum_{i=1}^n\left\{
\E\log\frac{f_{\pi,\sigma_i}(X_i)}{f_{\pistar,\sigma_i}(X_i)}
+\E\abs[\bigg]{\mathbf{1}(X_i\notin S_{n,t})
\log\frac{f_{\pi,\sigma_i}(X_i)}{f_{\pistar,\sigma_i}(X_i)}}
\right\}}\\
&\le -c_1\Helbarstar{f_\pi}{f_{\pistar}}
+c_1\rbar{1}+O(n^{-3}) \le -c_1\,2^{2\ell}\delta_n^2
+c_1\rbar{1}+O(n^{-3}). \numberthis
  \label{eq:truncated-population-drift}
\end{align*}
This follows from controlling $\frac{1}{n} \sum_i \E \abs[\big]{\mathbf{1}(X_i\notin S_{n,t})
           	\log\frac{f_{\pi,\sigma_i}(X_i)}
           {f_{\pistar,\sigma_i}(X_i)}
           } \lesssim n^{-3}$ by \cref{lem:log-mixture-bound} and
           \cref{assump:sub-Gauss}.

Equations \eqref{eq:sample_gap} and \eqref{eq:truncated-population-drift} show that, on
the relevant event $\mathcal E$, for $\bar A_ {nt}(\pi) := \E[A_{nt} (\pi)]$, if
$\fmix{\hat\pi_n}{\cdot}\in\mathcal F_\ell$, then
\[
	A_{nt}(\hat\pi_n) - \bar A_{nt}(\hat \pi_n) \ge (c_1 - \eta) \Helbarstar{f_{\hat\pi_n}}
	{f_{\pistar}} - c_1 \bar r_1 - C_\eta b_n - O(n^{-3}) \ge \frac{c_1}{2} 2^{2\ell}
	\delta_n^2 \numberthis \label{eq:key_deviation_event}
\]
by applying $\eta < c_1/16$, the definition of
$\mathcal F_\ell$, and the
definition of $\delta_n^2$
(increasing $C$ depending on the fixed $t$, $\eta$, and $C_\eta$ if
necessary).

For every $\ell$ such that $\mathcal F_\ell$ is nonempty, set
\[
r_\ell={\min(1,2^{2\ell+2}\delta_n^2)}, \quad \mathcal G(M,r_\ell):=
\left\{\pi\in\cP([-M,M]):
\Helbarstar{f_\pi}{f_{\pistar}}\le r_\ell\right\}.
\]
If $\fmix{\hat\pi_n}{\cdot}\in\mathcal F_\ell$, then
$\hat\pi_n\in\mathcal G(M,r_\ell)$. Moreover, note that if $\mathcal F_\ell$ is non-empty,
then \eqref{eq:local-radius} holds for $r_\ell$ by increasing the constant in $\delta_n$
if necessary: Indeed, if $r_\ell<1$, then $r_\ell\ge4\delta_n^2$, whereas if $r_\ell=1$,
nonemptiness of the shell gives $2^{2\ell}\delta_n^2\le1$, and therefore
$\log^4(n)/n+\bar r_2\le\delta_n^2/C\le1/C$.

Finally, since
$r_\ell\le4\cdot2^{2\ell}\delta_n^2$, the last term in
\eqref{eq:key_deviation_event} is at least
$c_1r_\ell/16$.
Thus
\begin{align*}
\Prob\left(\mathcal E\cap
\left\{\fmix{\hat\pi_n}{\cdot}\in\mathcal F_\ell\right\}\right) &\le
{\Prob\left(\eqref{eq:key_deviation_event}\text{ holds},\,
\fmix{\hat\pi_n}{\cdot}\in\mathcal F_\ell\right)} \\
&\le \Prob\bk{
\sup_{\pi\in {\mathcal G(M,r_\ell)}}
\left\{A_{nt}(\pi)-\E A_{nt}(\pi)\right\}
\ge \frac{c_1}{16}r_\ell
}\\
&\lesssim \exp\pr{-C\frac{nr_\ell}{\log^2n}}. \numberthis \label{eq:shell_ineq}
\end{align*}
The last inequality follows from \cref{prop:local-maximal}.

We bound \eqref{eq:shell_ineq} further. If $r_\ell = 2^{2\ell+2}\delta_n^2$, then,
after possibly increasing the constant in $\delta_n^2$, $
\eqref{eq:shell_ineq} \le (n^{-2})^{2^{2\ell}}.$ Otherwise,
if $r_\ell = 1$, $1/4 \le 2^{2\ell}  \delta_n^2 \le 1$, which implies that there are at
most two non-empty shells with $r_\ell =1$, for which $\eqref{eq:shell_ineq} \le e^{-C
n/\log^2n} \lesssim n^{-2}$.
Thus
\begin{align*}
\Prob\bk{
	\Helbarstar{f_{\hat\pi_n}}{f_{\pistar}} \ge
	 \delta_n^2
} &\le  {\Prob(\mathcal E^c)}
+  {
\sum_{\substack{\ell\ge0:\,\mathcal F_\ell\ne\varnothing}}}
\Prob\left(\mathcal E\cap
\left\{\fmix{\hat\pi_n}{\cdot}\in\mathcal F_\ell\right\}\right)
\lesssim n^{-2}.
\end{align*}
This proves the high-probability bound. Since squared Hellinger
distance is at most one,
\[
\E\Helbarstar{f_{\hat\pi_n}}{f_{\pistar}}
\le \delta_n^2+
\Prob\bk{\Helbarstar{f_{\hat\pi_n}}{f_{\pistar}}\ge\delta_n^2}
\lesssim\delta_n^2,
\]
where the last inequality also uses $n^{-2}\lesssim\delta_n^2$.
This concludes the proof.
\end{proof}

\subsection{Proof of \cref{thm:deconvolution}}

\begin{proof}[Proof of \cref{thm:deconvolution}]
	After rescaling by $\sigma_i$, \citet[Theorem~2(2) and Example~2]{Nguyen2013} and
	$\TV\leq\sqrt{2}\,\Hel$ give, uniformly over $1\leq i\leq n$,\footnote{The bound in
	\citet{Nguyen2013} considers sufficiently small Hellinger distance. Since $W^2_2$ is
	bounded for two compactly supported priors, we can allow for any Hellinger distance by
	enlarging the constant.}
	\[
		\Wass^2(\pihat,\pistar)
		\lesssim
		\left\{1+\log\!\left(
		1+\Hel(f_{\pihat,\sigma_i},f_{\pistar,\sigma_i})^{-1}
		\right)\right\}^{-1}.
	\]
	(This is uniform over $i$ because of \cref{assump:lowerbound}.) Since some coordinate
	has squared Hellinger distance no larger than its average,
	\[
		\Wass^2(\pihat,\pistar)
		\lesssim
		\left[
		1+\log\!\left\{
		1+\left(\Helbarstar{f_{\pihat}}{f_{\pistar}}\right)^{-1/2}
		\right\}
		\right]^{-1}.
	\]
	Consequently, \cref{thm:main} and $\Wass(\pihat,\pistar)\leq2M$ imply
	\begin{align*}
		\E\Wass^2(\pihat,\pistar)
		&\lesssim
		\left\{1+\log(1+\delta_n^{-1})\right\}^{-1}
		+\Prob\!\left(
		\Helbarstar{f_{\pihat}}{f_{\pistar}}>\delta_n^2
		\right) \\
		&\lesssim
		\left\{1+\log(1+\delta_n^{-1})\right\}^{-1},
	\end{align*}
	where the last step uses $\delta_n^2\gtrsim(\log^4n )/n$. Jensen's inequality
	completes the proof.
\end{proof}

\subsection{Proof of \cref{thm:denoising-normal}}

    \begin{proof}[Proof of \cref{thm:denoising-normal}]
		For this proof and the proof of \cref{thm:denoising-strong}, write
		\[
		\rmis(\pistar)
		:=\avgn \E|\theta_i-h_{\pistar}(Z_i;\sigma_i)|^2.
		\]
		By the reverse triangle inequality,
		\begin{align}
\left(\mathrm{Regret}_n\right)_+
		& {\leq
		\pr{
		\sqrt{\avgn \E|\theta_i-h_{\pihat}(X_i;\sigma_i)|^2}
		-
		\sqrt{\avgn \E|\theta_i-h_{\pistar}(Z_i;\sigma_i)|^2}
		}_+} \nonumber\\
		&{\quad+
		\sqrt{\avgn \E\bigl|h_{\pihat}(X_i;\hat\sigma_i)
		-h_{\pihat}(X_i;\sigma_i)\bigr|^2}.}
		\label{eq:estimated-variance-regret-split}
		\end{align}
		We first show that the last term is bounded by $\sqrt{C\delta_n^2\max\{1,\log(\delta_n^{-1})\}}$. We do so
		by differentiating in $\sigma_i$. For $v>0$, we have
		\begin{equation*}
		\frac{\partial}{\partial v}h_\pi(x;\sqrt v)
		=
		\frac{1}{2v^2}
		\operatorname{Cov}\!\left(\theta,(x-\theta)^2\mid x;v\right),
		\end{equation*}
		where the covariance is under the prior $\pi$ and likelihood $X \mid \theta, v \sim \Norm(\theta, v)$.  Since
		\begin{equation*}
		\operatorname{Cov}\!\left(\theta,(x-\theta)^2\mid x;v\right)
		=-2x\operatorname{Var}(\theta\mid x;v) +\operatorname{Cov}(\theta,\theta^2\mid x;v),
		\end{equation*}
	 compact support of $\pi$ and the lower bounds on $\sigma_i$ and
		$\hat\sigma_i$ imply, uniformly over $\pi\in\cP([-M,M])$,
		\begin{equation}
		\left|h_\pi(x;\hat\sigma_i)-h_\pi(x;\sigma_i)\right|
		\lesssim (1+|x|)\left|\hat\sigma_i^2-\sigma_i^2\right|.
		\label{eq:posterior-mean-variance-lipschitz}
		\end{equation}
	For any $T\geq1$, \eqref{eq:posterior-mean-variance-lipschitz},
		$|h_\pi|\leq M$, \cref{assump:sub-Gauss}, and
		\eqref{eq:estimated-variance-regret-condition} therefore yield
		\begin{align*}
		\avgn \E\bigl|h_{\pihat}(X_i;\hat\sigma_i)
		-h_{\pihat}(X_i;\sigma_i)\bigr|^2  &\lesssim
		(1+T^2)\avgn \E(\hat\sigma_i^2-\sigma_i^2)^2
		+\avgn\Prob(|X_i|>T) \nonumber\\
		&\lesssim (1+T^2)\delta_n^2+e^{-cT^2} \le C\delta_n^2\max\{1,\log(\delta_n^{-1})\}.
		\end{align*}
		The last inequality follows by taking
		$T^2=C_0\max\{1,\log(\delta_n^{-1})\}$ for a sufficiently large fixed $C_0$.

		The first term on the right side of
		\eqref{eq:estimated-variance-regret-split} satisfies
		\begin{align*}
        &\pr{\sqrt{\avgn \E|\theta_i - h_{\pihat}(X_i; \sigma_i)|^2}
		- \sqrt{\rmis(\pistar)}}_+
        \nonumber \\
		 &\leq \underbrace{\left|\sqrt{\avgn \E|\theta_i - h_{\pihat}(X_i; \sigma_i)|^2} -
		 \sqrt{\avgn
		 \E|\theta_i - h_{\pihat}(Z_i; \sigma_i)|^2}\right|}_{\text{(I)}} \\
		 &\quad + \underbrace{
		 \pr{
		 \sqrt{\avgn \E|\theta_i - h_{\pihat}(Z_i; \sigma_i)|^2}
		 - \sqrt{\rmis(\pistar)}}_+}_{
		 \text{(II)}}.
		\end{align*}

		\medskip\noindent\emph{Bound for Term (I)}. For any $\pi \in \cP([-M,M])$, the Gaussian posterior mean $h_\pi(\cdot; \sigma)$ satisfies
		\begin{align}\label{eq:lipschitzmean}
			\frac{\partial}{\partial x} h_\pi(x; \sigma) = \frac{\Var_\pi(\theta \mid X = x)}{\sigma^2} \leq \frac{M^2}{\sigma^2} \leq \frac{M^2}{k^2},
		\end{align}
		since $\theta \in [-M, M]$ under any posterior derived from a prior supported on $[-M,M]$. In particular, $h_\pi(\cdot; \sigma_i)$ is $(M^2/k^2)$-Lipschitz uniformly over $\pi \in \cP([-M,M])$ and $1 \leq i \leq n$. Therefore,
		\begin{align}\label{eq:lipcoup}
			\text{(I)} &\leq \sqrt{\avgn \E\bigl|h_{\pihat}(X_i; \sigma_i) - h_{\pihat}(Z_i; \sigma_i)\bigr|^2} \leq \frac{M^2}{k^2}\sqrt{\avgn \E|X_i - Z_i|^2}.
		\end{align}
		We couple $(X_i, Z_i)$ conditionally on $\theta_i$ by choosing an optimal $\Wtwo$-coupling of $\mathrm{Law}(X_i \mid \theta_i)$ and $\cN(\theta_i, \sigma_i^2)$, independently across $i$, to obtain
		\begin{equation*}
			\text{(I)} \leq \frac{M^2}{k^2}\,\mathcal{W}_n.
		\end{equation*}

		\medskip\noindent\emph{Bound for Term (II)}. The triangle inequality gives
		\[
			\text{(II)}^2 \leq \avgn \E
			\left|h_{\pihat}(Z_i;\sigma_i)-h_{\pistar}(Z_i;\sigma_i)\right|^2.
		\]
		We apply \citet[Theorem~1]{chen2026sharp} to bound the difference in posterior means
		under $\pihat$ and $\pistar$. To apply it, consider the change of variables $Z_i
		\mapsto Z_i / \sigma_i$. The resulting priors are supported on $[-M/k,M/k]$ by
		\cref{assump:lowerbound}. Hellinger distance is invariant under this transformation,
		and differences in posterior means are invariant up to some constant by
		\cref{assump:lowerbound}. We may then apply Theorem 1 in \citet{chen2026sharp} for
		each coordinate and use Jensen's inequality to aggregate:\footnote{The displayed bound
		in \citet [Theorem~1] {chen2026sharp} is stated for sufficiently small Hellinger
		distance. We apply it when the Hellinger distance is sufficiently small and apply
		Jensen's inequality. For any coordinate with large Hellinger distance, both posterior
		means lie in $[-M,M]$, so the same inequality follows after enlarging the constant.}
		\[
			\avgn \E\left|h_\pi(Z_i;\sigma_i)-h_{\pistar}(Z_i;\sigma_i)\right|^2
			\lesssim
			\Helbarstar{f_\pi}{f_{\pistar}}
			\max\!\left\{1,
			\log\!\left(\frac{1}{\sqrt{\Helbarstar{f_\pi}{f_{\pistar}}}}\right)
			\right\}.
		\]
		By Bernstein's inequality, we have that for a fixed $\pi$,
		\[
			\avgn \left|h_\pi(Z_i;\sigma_i)-h_{\pistar}(Z_i;\sigma_i)\right|^2
			\lesssim
			\Helbarstar{f_\pi}{f_{\pistar}}
			\max\!\left\{1,
			\log\!\left(\frac{1}{\sqrt{\Helbarstar{f_\pi}{f_{\pistar}}}}\right)
			\right\}
			+\frac{\log^2 n}{n} \numberthis \label{eq:bernstein-regret-hellinger}
		\]
		with probability $1-e^{-C \log^2 n}$.

    To apply \eqref{eq:bernstein-regret-hellinger} to the data-dependent \(\hat\pi_n\), we
    make the bound uniform over \(\pi\in\mathcal P([-M,M])\). A covering argument, as in the
    proof of \cref{prop:local-maximal}, gives such a uniform bound with probability
    \(1-O(n^{-2})\). Applying \cref{thm:main} and converting the high-probability bound to an
    expectation bound completes the argument.
    \end{proof}

\subsection{Proof of \cref{thm:denoising-strong}}

     We now move on to the proof of \cref{thm:denoising-strong}. Let us first state a useful technical lemma which controls the expected squared Hellinger distance between two posteriors with the same prior; also see~\citet[Corollary A.2]{norets2017adaptive}.

    \begin{lemma}\label{lem:hellinger-posteriors}
		Let $\pi \in \cP([-M,M])$ and let $f_\theta$, $g_\theta$ be two families of conditional densities on~$\R$ indexed by $\theta \in [-M,M]$, with marginals $m(x) := \int f_\theta(x)\,d\pi(\theta)$ and $m_G(x) := \int g_\theta(x)\,d\pi(\theta)$. Define the corresponding
		posterior densities with respect to $\pi$
		\[
		p_f(\theta \mid x) :=
	\frac{f_\theta(x)}{m(x)}, \quad
		p_g(\theta \mid x) :=
		\frac{g_\theta(x)}{m_G(x)}.
		\]
		Then
		\[
		\int \Helsq\!\left(p_f(\cdot \mid x),\, p_g(\cdot \mid x)\right) m(x)\,dx
		\leq 4 \int \Helsq(f_\theta, g_\theta)\,d\pi(\theta).
		\]
	\end{lemma}

    \begin{proof}[Proof of \cref{lem:hellinger-posteriors}]
		Let us write $U := \sqrt{f_\theta(x)}$, $V := \sqrt{g_\theta(x)}$, $P := \sqrt{m(x)}$, $Q := \sqrt{m_G(x)}$. Then
		\[
		\sqrt{p_f(\theta|x)} - \sqrt{p_g(\theta|x)} = \frac{1}{PQ}\bigl(UQ - VP\bigr).
		\]
		We then decompose $UQ - VP = Q(U - V) + V(Q - P)$ and apply $(a + b)^2 \leq 2a^2 + 2b^2$ to get
		\[
		\Helsq(p_f(\cdot|x), p_g(\cdot|x)) \leq
\int \frac{1}{P^2}\,(U - V)^2\,d\pi(\theta)
		+
		\int \frac{V^2}{P^2 Q^2}\,(Q - P)^2\,d\pi(\theta).
		\]
		Let us multiply both sides of the above display by $m(x) = P^2$ and integrate over $x$.

		\medskip\noindent\emph{Term 1.} We obtain
		$
		\iint(\sqrt{f_\theta} - \sqrt{g_\theta})^2\,d\pi(\theta)\,dx
		= 2\int \Helsq(f_\theta, g_\theta)\,d\pi(\theta)$.

		\medskip\noindent\emph{Term 2.} The factor
		$
		\int g_\theta(x)/m_G(x)\,d\pi(\theta) = 1$ collapses the
		$\theta$-integral, which makes the second term simplify to $\int (\sqrt{m_G} - \sqrt{m})^2\,dx = 2\,\Helsq(m, m_G)$. By the Cauchy--Schwarz inequality, we then have
		\[
		\int \sqrt{m\,m_G}\,dx = \int \sqrt{\textstyle\int f_\theta\,d\pi \cdot \int g_\theta\,d\pi}\,dx \geq \iint \sqrt{f_\theta\,g_\theta}\,d\pi\,dx = 1 - \int \Helsq(f_\theta, g_\theta)\,d\pi(\theta),
		\]
		so $\Helsq(m, m_G) \leq \int \Helsq(f_\theta, g_\theta)\,d\pi(\theta)$. Combining the two terms completes the proof.
	\end{proof}

	\begin{proof}[Proof of \cref{thm:denoising-strong}]
	Observe that we can relate the two regrets:
	\begin{align}
		\mathrm{Regret}_n^\star
		&=
		\mathrm{Regret}_n
		+\sqrt{\avgn \E|\theta_i-h_{\pistar}(Z_i;\sigma_i)|^2}
		-\sqrt{\avgn \E|\theta_i-\delta_i^\star(X_i)|^2}, \notag \\
		(\mathrm{Regret}_n^\star)_+
		&\leq
		(\mathrm{Regret}_n)_+
		+\left|
		\sqrt{\avgn \E|\theta_i-h_{\pistar}(Z_i;\sigma_i)|^2}
		-\sqrt{\avgn \E|\theta_i-\delta_i^\star(X_i)|^2}
		\right|.
		\label{eq:regret-star-comparison}
		\end{align}
		We insert the intermediate quantity $D := n^{-1}\sum_{i=1}^n \E|\theta_i - h_{\pistar}(X_i; \sigma_i)|^2$ and decompose:
		\begin{align}
			&\left|\sqrt{\rmis(\pistar)}
			-\sqrt{\avgn \E|\theta_i-\delta_i^\star(X_i)|^2}\right| \nonumber\\
			&\quad\leq
			\underbrace{\left|\sqrt{\rmis(\pistar)}-\sqrt{D}\right|}_{\text{(I)}}
			+\underbrace{\left|\sqrt{D}
			-\sqrt{\avgn \E|\theta_i-\delta_i^\star(X_i)|^2}\right|}_{\text{(II)}}.
			\label{eq:piece-II-split}
		\end{align}

		\medskip\noindent\emph{Bound for Term (I)}. By the same Lipschitz and coupling argument as in \eqref{eq:lipcoup}, we get:
		\begin{equation}\label{eq:IIa-bound}
			\text{(I)} \leq \frac{M^2}{k^2}\,\mathcal{W}_n.
		\end{equation}

		\medskip\noindent\emph{Bound for Term (II)}. Since
		$\delta_i^\star(x)= \E[\theta_i \mid X_i = x]$ minimizes the MSE, the Pythagorean identity gives
		\[
		D - \avgn \E|\theta_i-\delta_i^\star(X_i)|^2
		= \avgn \E\bigl|h_{\pistar}(X_i;\sigma_i)-\delta_i^\star(X_i)\bigr|^2 \geq 0.
		\]
		Using $\sqrt{a} - \sqrt{b} \leq \sqrt{a - b}$ for $a \geq b \geq 0$, we get
		\[
			\text{(II)} \leq \sqrt{\avgn \E\bigl|h_{\pi_\star}(X_i; \sigma_i) - \delta_i^\star(X_i)\bigr|^2}.
		\]
		Since $\theta_i \in [-M, M]$ under both posteriors, both $h_{\pistar}(x; \sigma_i)$ and $\delta_i^\star(x)$ lie in $[-M, M]$, so
		\[
		\bigl|h_{\pistar}(x; \sigma_i) - \delta_i^\star(x)\bigr|\leq
		2M\,\TV\!\left(p_{G,i}(\cdot \mid x),\,
		p_{\mathrm{true},i}(\cdot \mid x)\right),
		\]
		where $p_{G,i}(\theta \mid x) :=
		\phi\!\left(x - \theta;{\sigma_i}\right)/\fmix{\pistar}{\sigma_i}(x)$
		is the ``normal'' posterior density with respect to $\pistar$, and
		$p_{\mathrm{true},i}(\theta \mid x) :=
		\nu_{\theta,i}(x)/m_i(x)$ is the true posterior density with respect to
		$\pistar$, with $m_i(x) := \int \nu_{\theta,i}(x)\,d\pistar(\theta)$. Here
		$\nu_{\theta,i}$ denotes the conditional
		density of $X_i$ given $\theta$ as in
		\cref{sec:deconvolution}. By $\TV^2 \leq 2\,\Helsq$ and
		\cref{lem:hellinger-posteriors}, we obtain
		\begin{align*}
		&\;\;\;\;\avgn \E_{X_i \sim m_i}\!\left[\TV^2\!\left(p_{G,i}(\cdot \mid X_i),\, p_{\mathrm{true},i}(\cdot \mid X_i)\right)\right] \\
		&\leq 2\avgn \E_{X_i\sim m_i}\!\left[\Helsq\!\left(p_{G,i}(\cdot|X_i),\, p_{\mathrm{true},i}(\cdot|X_i)\right)\right]
		\leq 8\,\mathcal{H}_n^2.
		\end{align*}
		Therefore,
		\begin{align}\label{eq:IIb-bound}
			\text{(II)} \leq 2M
			\sqrt{\avgn \E_{X_i \sim m_i}\!\left[\TV^2\!\left(p_{G,i}(\cdot \mid X_i),\,
			p_{\mathrm{true},i}(\cdot \mid X_i)\right)\right]} \leq 4\sqrt{2}M\,\mathcal{H}_n.
		\end{align}
		Equations \eqref{eq:regret-star-comparison},
		\eqref{eq:piece-II-split}, \eqref{eq:IIa-bound}, and \eqref{eq:IIb-bound} imply
		\begin{equation*}
		(\mathrm{Regret}_n^\star)_+
		\leq
		(\mathrm{Regret}_n)_+
		+C\mathcal W_n+C\mathcal H_n.
		\end{equation*}
		Applying \cref{thm:denoising-normal} and
		$(a+b+c)^2\leq3(a^2+b^2+c^2)$ gives
		\[
		(\mathrm{Regret}_n^\star)_+^2
		\lesssim \delta_n^2\max\{1,\log(\delta_n^{-1})\} +
		\mathcal{W}_n^2 + \mathcal{H}_n^2.
		\qedhere
		\]
	\end{proof}

\subsection{\cref{prop:hellinger-bound} and its proof}

The following proposition bounds $r_1, r_2$ in terms of the Hellinger distance between $\mu$ and $f_{\pistar, \sigma_i}$. This indicates that we can think of $r_1, r_2$ as measuring CLT under a slightly stronger metric.

\begin{prop}\label{prop:hellinger-bound} Suppose Assumptions~\ref{assump:sub-Gauss}
and~\ref{assump:lowerbound} hold and that $\mu_i$ admits a Lebesgue density. Then, for
some constant $T \geq 1$ independent of $i$, the measure $\mu_i$ is an $(r_{1,i},
r_{2,i})$-approximate normal convolution of $\pistar$ at scale $\sigma_i$ in the sense of
Definition~\ref{asn:assumeindep}, with~\eqref{eq:r-hel}: \begin{align*} r_{1,i} = r_{2,i} =
g_{T}\!\left(\Helsq(\mu_i, f_{\pistar,\sigma_i})\right). \end{align*} In
the well-specified case $\mu_i = f_{\pistar, \sigma_i}$, we have $r_{1,i} = r_{2,i} = 0$.
\end{prop}

  The proof of \cref{prop:hellinger-bound} requires a few technical lemmas which we present first. The first couple of lemmas explore some analytic properties of the function $g_T(\cdot)$ defined in \eqref{eq:gT}.

	\begin{lemma}[Subadditivity of $g_T$]\label{lem:gT-subadditivity}
		Let $T > 2e$ and recall the definition of $g_T(\cdot)$ from \eqref{eq:gT}. Then for all $0 < x, y < 1$,
		\[
		g_T(x + y) \leq g_T(x) + g_T(y).
		\]
	\end{lemma}

\begin{lemma}[Mixed entropy-type inequality]\label{lem:mixed-entropy-ineq}
		For all $T > e^2$ and all $0 < a, b < 1$,
		\[
		\sqrt{2}\,a\sqrt{g_T(a^2)+g_T(b^2)} - b^2 \leq  2 g_{T}(a^2) - \frac{1}{2}\,b^2.
		\]
	\end{lemma}

    \noindent For our next set of results, we introduce a reweighted Kullback-Leibler divergence given by
	\[
	\KL(q;\, p_1 | p_2) = \int q \log \frac{p_1}{p_2}.
	\]
	Clearly if $q = p_1$, then $\KL(q;\, p_1 | p_2) = \operatorname{KL}(p_1\,\Vert\,p_2)$ and if $q = p_2$, then $\KL(q;\, p_1 | p_2) = -\operatorname{KL}(p_2\,\Vert\,p_1)$. Note that the quantity $\KL(q;\, p_1 | p_2)$ arises naturally in the context of \cref{asn:assumeindep} for bounding the left-hand side of~\eqref{eq:condition-a}. In particular, the left-hand side of \eqref{eq:condition-a} equals $\KL(q;\, p_1 | p_2)$ with $q=\mu_i$, $p_1=f_{\pi,\sigma_i}$, and $p_2=f_{\pst,\sigma_i}$.

    \noindent In a similar vein, we introduce a reweighted $k$-th order Kullback-Leibler variation given by
	\[
	V_k(q;\, p_1, p_2) := \int q \left| \log \frac{p_1}{p_2} \right|^k.
	\]
    The quantity $V_k(q;\, p_1, p_2)$ arises naturally in the context of \cref{asn:assumeindep} for bounding the left-hand side of~\eqref{eq:condition-b}. In particular, the left-hand side of \eqref{eq:condition-b} equals $V_k(q;\, p_1, p_2)$ with $k=2$, $q=\mu_i$, $p_1=f_{\pi,\sigma_i}$, and $p_2=f_{\pst,\sigma_i}$.

	If $q = p_1$ or $p_2$, then clearly the reweighted $k$-th order Kullback-Leibler variation is equal (up to permuting $p_1$ and $p_2$) to the usual $k$-th order Kullback-Leibler variation given by
    \[
	V_k(p_1 \| p_2) = \int p_1 \left| \log \frac{p_1}{p_2} \right|^k.
	\]
    Bounds for the usual Kullback-Leibler variation have been studied extensively (see \cite{wong1995probability,kaji2026hellinger}). Our subsequent bounds can be viewed as extensions of these existing results to the reweighted setting.

	Our first result studies a bound on $V_k(p_1 \| p_2)$ in terms of the squared Hellinger distance between $p_1$ and $p_2$. Related bounds appear in \citet[Theorem~2 and Proposition~4]{kaji2026hellinger} and \citet[Theorem~5]{wong1995probability}; the proof techniques are similar. We present a version tailored to the setting of our paper.

	\begin{lemma}\label{lem:Vk-hellinger-bound}
		Suppose there exists some $\delta \in (0,1)$ such that \[M_\delta := \int p_1 \left(
		\frac{p_1}{p_2}\right)^\delta < \infty. \numberthis \label{eq:exponential_ratio}\]
		Then we have
		\[
		V_k(p_1 \| p_2) \leq 10\,\delta^{-k}\,\Helsq(p_1, p_2) \left[ C_{k,\delta} + \left|\log\!\left(\frac{M_\delta}{5\,\Helsq(p_1, p_2)}\right)\right|^k \right]
		\]
		for some constant $C_{k,\delta} > 0$, provided $k \geq 2$.
	\end{lemma}

    In our next result, we provide a bound on the reweighted Kullback-Leibler variation, which is new to the best of our knowledge. We show that $V_k(q;p_1,p_2)$ can be bounded in terms of $\Helsq(p_1,p_2)$ and $\Helsq(q,p_2)$.

	\begin{lemma}\label{lem:reweighted-Vk-bound}
		Suppose there exists $\delta \in (0,1)$ such that
		\[
		T_\delta := \int q \left[ \left(\frac{p_1}{p_2}\right)^\delta + \left(\frac{p_2}{p_1}\right)^\delta \right] < \infty.
		\]
		Set $A := \Helsq(p_1, p_2) + \Helsq(q, p_2)$. Then we have:
		\[
		V_k(q;\, p_1, p_2) \leq 96\,\delta^{-k} \left(C_{k,\delta} \vee \left|\log\!\left(\frac{T_\delta}{A}\right)\right|^k\right) A
		\]
		for some constant $C_{k,\delta} > 0$, provided $k \geq 2$.
	\end{lemma}

    Our final technical lemma provides a bound on the reweighted Kullback-Leibler divergence, also in terms of $\Helsq(p_1,p_2)$ and $\Helsq(q,p_2)$.

	\begin{lemma}\label{lem:reweighted-KL-hellinger}
		Suppose there exists some $\delta \in (0,1)$ such that $M_\delta$ and $T_\delta$ from \cref{lem:Vk-hellinger-bound,lem:reweighted-Vk-bound} are finite. Then we have
		\[
		\KL(q;\, p_1 | p_2) \leq \sqrt{2}\,\Hel(q, p_2)\,\sqrt{V_2(q;\, p_1, p_2) + V_2(p_2 \| p_1)} - 2\Helsq(p_1, p_2).
		\]
		Note that $V_2(q;\, p_1, p_2)$ and $V_2(p_2 \| p_1)$ can be further bounded in terms of $\Helsq(q, p_2)$ and $\Helsq(p_1, p_2)$ from \cref{lem:Vk-hellinger-bound,lem:reweighted-Vk-bound}.
	\end{lemma}

    \begin{proof}[Proof of \cref{prop:hellinger-bound}]
		We verify conditions~\eqref{eq:condition-a}--\eqref{eq:condition-b} with $r_{1,i}$ and $r_{2,i}$ as defined in the Proposition. Define
		\begin{equation*}
			M_{\delta,i} := \int f_{\pst,\sigma_i}\left(\frac{f_{\pst,\sigma_i}}{f_{\pi,\sigma_i}}\right)^\delta.
		\end{equation*}
		Since both $f_{\pi,\sigma}$ and $f_{\pistar,\sigma}$ are normal location mixtures with the same variance~$\sigma^2$, we have:
		\[
\frac{f_{\pst,\sigma_i}(x)}{f_{\pi,\sigma_i}(x)} = \frac{\int \exp(\theta x/\sigma_i^2)\exp(-\theta^2/(2\sigma_i^2))\,d\pst(\theta)}{\int \exp(\theta x/\sigma_i^2)\exp(-\theta^2/(2\sigma_i^2))\,d\pi(\theta)}.
		\]
		As both $\pi$ and $\pistar$ are supported on $[-M, M]$, the numerator is at most $\exp(M|x|/\sigma_i^2)$ and the denominator is at least $\exp(-M|x|/\sigma_i^2 - M^2/(2\sigma_i^2))$, so
		\begin{equation}\label{eq:ratio-growth}
			\frac{f_{\pst,\sigma_i}(x)}{f_{\pi,\sigma_i}(x)} \leq \exp\!\left(\frac{2M|x|}{\sigma_i^2} + \frac{M^2}{2\sigma_i^2}\right) \leq \exp\!\left(\frac{2M|x|}{k^2} + \frac{M^2}{2k^2}\right).
		\end{equation}
		As $f_{\pistar,\sigma_i}$ has sub-Gaussian tails with a uniformly upper bounded sub-Gaussian norm, this implies
        \begin{equation}\label{eq:moment-cond-wass}
            M_{\delta}:=\sup_{n\ge 1}\max_{1\le i\le n} M_{\delta,i}<\infty
        \end{equation}
        for any $\delta \in (0, 1)$. Next define
        $$T_{\delta,i}:=\int \left[\left(\frac{f_{\pi,\sigma_i}}{f_{\pst,\sigma_i}}\right)^{\delta}+\left(\frac{f_{\pst,\sigma_i}}{f_{\pi,\sigma_i}}\right)^{\delta}\right]\,d\mu_i.$$
        As the $\mu_i$s are uniformly sub-Gaussian by \cref{assump:sub-Gauss}, by using \eqref{eq:ratio-growth}, we have
        \begin{equation}\label{eq:moment-cond-wass-ver}
            T_{\delta}:=\sup_{n\ge 1}\max_{1\le i\le n} T_{\delta,i}<\infty
        \end{equation}
	        for any $\delta\in (0,1)$. Therefore \eqref{eq:moment-cond-wass} and \eqref{eq:moment-cond-wass-ver} verify the conditions needed to apply Lemmas \ref{lem:Vk-hellinger-bound}--\ref{lem:reweighted-KL-hellinger}. For the rest of the proof, let $q_i:=\mu_i$, $p_{1,i}:=f_{\pi,\sigma_i}$, and $p_{2,i}:=f_{\pst,\sigma_i}$.

		\medskip
		\noindent \emph{Bound on $r_2$}. Applying \cref{lem:reweighted-Vk-bound} with $k = 2$, and setting $A_i := \Helsq(p_{1,i}, p_{2,i}) + \Helsq(q_i, p_{2,i})$, we obtain
		\[
			\E_{X\sim \mu_i}t_{\pi,i}(X)^2 \lesssim A_i\log^2\!\left(\frac{T}{A_i}\right),
		\]
		for some $T>T_{\delta}\vee 2e$ large enough. Recall the definition of $g_T(\cdot)$ from \eqref{eq:gT}. By our choice $T>2e$, note that \cref{lem:gT-subadditivity} gives
		\begin{equation}\label{eq:hel-gT-split}
			g_T(A_i) = g_T\big(\Helsq(q_i,p_{2,i})+\Helsq(p_{1,i},p_{2,i})\big) \leq g_T\big(\Helsq(q_i,p_{2,i})\big) + g_T\big(\Helsq(p_{1,i},p_{2,i})\big).
		\end{equation}
		The term $g_{T}(\Helsq(p_{1,i}, p_{2,i}))$ is absorbed into the leading term of condition~\eqref{eq:condition-b}. Therefore \eqref{eq:hel-gT-split} implies that \eqref{eq:condition-b} holds with
		$r_{2,i} = g_{T}\!\left(\Helsq(q_i, p_{2,i})\right)$,
		which in turn establishes~\eqref{eq:r-hel}.

		\medskip
		\noindent \emph{Bound on $r_1$}. Applying \cref{lem:reweighted-KL-hellinger}, we obtain
		\begin{equation}\label{eq:hel-r1-step1}
			\E_{X\sim\mu_i} t_{\pi,i}(X) \leq \sqrt{2}\,\Hel(q_i, p_{2,i})\,\sqrt{\int q_i\log^2\left(\frac{p_{1,i}}{p_{2,i}}\right) + \int p_{2,i}\log^2\left(\frac{p_{1,i}}{p_{2,i}}\right)} - \Helsq(p_{1,i}, p_{2,i}).
		\end{equation}
		Let us bound the terms inside the square root. From \cref{lem:reweighted-Vk-bound} (already applied above) and \cref{lem:Vk-hellinger-bound} with $k = 2$:
		\begin{align*}
			\int q_i\log^2\left(\frac{p_{1,i}}{p_{2,i}}\right) &\lesssim g_T(\Helsq\big(q_i,p_{2,i}\big)) + g_T\big(\Helsq(p_{1,i},p_{2,i})\big), \\
			\int p_{2,i}\log^2\left(\frac{p_{1,i}}{p_{2,i}}\right) &\lesssim g_{T}\big(\Helsq(p_{1,i},p_{2,i})\big),
		\end{align*}
		where $T>0$ is some large enough constant depending on $M_\delta$ and $T_\delta$. Therefore
		\begin{align*}
		\int q_i\log^2\left(\frac{p_{1,i}}{p_{2,i}}\right) + 	\int p_{2,i}\log^2\left(\frac{p_{1,i}}{p_{2,i}}\right) \lesssim g_T\big(\Helsq(q_i,p_{2,i})\big) + g_T\big(\Helsq(p_{1,i},p_{2,i})\big)
		\end{align*}
		for $T>0$ large enough. By \eqref{eq:hel-r1-step1}, we then have
		\begin{align*}
		\E_{X\sim\mu_i} t_{\pi,i}(X) &\lesssim \Hel(q_{i},p_{2,i})\,\sqrt{g_T\big(\Helsq(q_i,p_{2,i})\big)+g_T\big(\Helsq(p_{1,i},p_{2,i})\big)} - \Helsq(p_{1,i},p_{2,i}) \\ &\lesssim g_T\big(\Helsq(q_i,p_{2,i})\big) - \Helsq(p_{1,i},p_{2,i}),
		\end{align*}
	where the last inequality follows from  \cref{lem:mixed-entropy-ineq}. This establishes condition~\eqref{eq:condition-a} with $r_{1,i}$ as in~\eqref{eq:r-hel}.
	\end{proof}
\subsection{Proofs of auxiliary lemmas}\label{sec:auxiliary-lemmas}

    \begin{proof}[Proof of \cref{lem:log-mixture-bound}]
		By the triangle inequality, we have
		\[
		\left| \log \left( \int \phi\!\left(x-\theta;\sigma\right) d\pi(\theta) \right) \right| \leq \left| \log(\sqrt{2\pi}\sigma) \right| + \left| \log \int \exp\!\left(-\frac{1}{2\sigma^2}(x-\theta)^2\right) d\pi(\theta) \right|.
		\]
		Clearly, $\int \exp\!\left(-\frac{1}{2\sigma^2}(x-\theta)^2\right) d\pi(\theta) \leq 1$. Also note that for any $\theta \in [-M,M]$, we have
		\[
		\exp\!\left(-\frac{1}{2\sigma^2}(x-\theta)^2\right) \geq \exp\!\left(-\frac{1}{2\sigma^2}x^2 - \frac{1}{2\sigma^2}M^2 - \frac{1}{\sigma^2}|x\theta|\right) \geq \exp\!\left(-\frac{1}{2\sigma^2}(|x|+M)^2\right).
		\]
		This completes the proof.
	\end{proof}

    \begin{proof}[Proof of \cref{lem:sup-log-bound}]
		Arguing as in the proof of \cref{lem:log-mixture-bound}, for any $\theta \in [-M,M]$, we have
		\[
		\inf_{|x| \leq T} \exp\!\left(-\frac{1}{2\sigma^2}(x-\theta)^2\right) \geq \inf_{|x| \leq T} \exp\!\left(-\frac{1}{2\sigma^2}(|x|+M)^2\right) \geq \exp\!\left(-\frac{1}{\sigma^2}(T^2 + M^2)\right).
		\]
		The conclusion now follows using the inequality $\log(1+x) \leq x$ for $x \geq 0$.
	\end{proof}

    \begin{proof}[Proof of \cref{lem:gT-subadditivity}]
		We compute
		\[
		g_T'(z) = \log^2(T/z) - 2\log(T/z), \qquad g_T''(z) = \frac{2\bigl(1 - \log(T/z)\bigr)}{z}.
		\]
		For $z \in (0, 2)$ and $T > 2e$, we have $T/z > e$, so $\log(T/z) > 1$, giving $g_T''(z) < 0$. Hence $g_T$ is strictly concave on $(0, 2)$.

		Since $g_T(0^+) = \lim_{z \to 0^+} z\log^2(T/z) = 0$, we extend $g_T$ continuously to $[0, 2)$ with $g_T(0) = 0$. For $0 < x, y < 1$, note that $0 < x + y < 2$. By concavity,
		\[
		g_T(x) = g_T\!\left(\frac{x}{x+y}(x+y) + \frac{y}{x+y}\cdot 0\right) \geq \frac{x}{x+y}\,g_T(x+y) + \frac{y}{x+y}\,g_T(0) = \frac{x}{x+y}\,g_T(x+y).
		\]
		The identical argument with $y$ in place of $x$ gives $g_T(y) \geq \frac{y}{x+y}\,g_T(x+y)$. Adding these two inequalities yields
		\[
		g_T(x) + g_T(y) \geq g_T(x+y). \qedhere
		\]
	\end{proof}

	\begin{proof}[Proof of \cref{lem:mixed-entropy-ineq}]
		Write $L_a := \log^2(T/a^2)$, $L_b := \log^2(T/b^2)$, and $S := a^2 L_a + b^2 L_b$. Since $0 < a, b < 1$ and $T > e^2$, we have $L_a, L_b \geq \log^2 T > 1$. Adding $b^2$ to both sides, we obtain the equivalent inequality
		\[
		\sqrt{2}\,a\sqrt{S} \leq 2a^2 L_a + \tfrac{1}{2}\,b^2.
		\]

		\medskip
		\noindent\emph{Case 1: $b^2 L_b \leq a^2 L_a$.} Then $S \leq 2a^2 L_a$, so
		\[
		\sqrt{2}\,a\sqrt{S} \leq 2a^2\sqrt{L_a} \leq 2a^2 L_a,
		\]
		where the last inequality uses $\sqrt{L_a} \leq L_a$, valid since $L_a > 1$.

		\medskip
		\noindent\emph{Case 2: $b^2 L_b > a^2 L_a$.} For $z\in [0,1]$, from the proof of \cref{lem:gT-subadditivity}, we observe that
        $$g_T'(z)=\big(\log{(T/z)}-1\big)^2-1>0$$
        for $T> e^2$. Therefore the function $g_T(\cdot)$ is increasing on $(0, 1)$. The assumption $b^2 L_b > a^2 L_a$ reads $g_T(b^2) > g_T(a^2)$, so $b > a$, and consequently $L_b \leq L_a$.

		From $S < 2b^2 L_b$ we obtain $\sqrt{2}\,a\sqrt{S} < 2ab\sqrt{L_b}$. An application of Young's inequality then yields
		\[
		2ab\sqrt{L_b} \leq 2a^2 L_b + \tfrac{1}{2}\,b^2 \leq 2a^2 L_a + \tfrac{1}{2}\,b^2,
		\]
		where the last inequality uses $L_b \leq L_a$. This completes the proof.
	\end{proof}

    We now present proofs of the information-theoretic \cref{lem:Vk-hellinger-bound,lem:reweighted-Vk-bound,lem:reweighted-KL-hellinger}, which relate reweighted and unweighted Kullback-Leibler divergence. This requires a preparatory result which is presented with proof below.

	\begin{lemma}\label{lem:hellinger-integral-bound}
		Given two Lebesgue densities $p_1$ and $p_2$ supported on $\R$, we have
		\[
		\int_{p_1 > 4p_2} p_1 \leq 8\,\Helsq(p_1, p_2).
		\]
	\end{lemma}

	\begin{proof}
		On the set $p_1 > 4p_2$, we have:
		\[
		\sqrt{p_1} > 2\sqrt{p_2} \quad \Longrightarrow \quad 2(\sqrt{p_1} - \sqrt{p_2}) \geq \sqrt{p_1}.
		\]
		Therefore,
		\[
		\int_{p_1 > 4p_2} p_1 \leq \int_{p_1 > 4p_2} \left(2(\sqrt{p_1} - \sqrt{p_2})\right)^2 \leq 4 \int (\sqrt{p_1} - \sqrt{p_2})^2 = 8\,\Helsq(p_1, p_2).
		\]
		This completes the proof.
	\end{proof}

	\begin{proof}[Proof of \cref{lem:Vk-hellinger-bound}]
		There exists a constant $C_k$ such that for $x \geq \frac{1}{4}$,
		\[
		\abs{\log x}^k \leq C_k (\sqrt{x} - 1)^2.
		\]
		Therefore, on the set $p_1 \leq 4p_2$, we have
		\[
		\abs{\log(p_1/p_2)}^k \leq C_k (\sqrt{p_2/p_1} - 1)^2.
		\]
		For any $M > 4 \vee e^{k/\delta}$, we can now write
		\begin{align*}
			V_k(p_1 \| p_2) &= \int_{p_1/p_2 \leq 4} p_1 |\log(p_1/p_2)|^k + \int_{4 < \frac{p_1}{p_2} \leq M} p_1 |\log(p_1/p_2)|^k + \int_{p_1/p_2 > M} p_1 |\log(p_1/p_2)|^k \\
			&\leq C_k \int p_1 \left(\sqrt{\frac{p_2}{p_1}} - 1\right)^2 + \log^kM \int_{p_1 > 4p_2} p_1 + \int_{p_1/p_2 > M} p_1 \left(\frac{p_1}{p_2}\right)^\delta \cdot \frac{\log^k\!\left(\frac{p_1}{p_2}\right)}{\left(\frac{p_1}{p_2}\right)^\delta} \\
			&\leq (C_k + 8\log^kM)\,\Helsq(p_1, p_2) + \int_{p_1/p_2 > M} p_1 \left(\frac{p_1}{p_2}\right)^\delta \cdot \frac{\log^k\!\left(\frac{p_1}{p_2}\right)}{\left(\frac{p_1}{p_2}\right)^\delta},
		\end{align*}
		where the last inequality follows from \cref{lem:hellinger-integral-bound}. To bound the last term in the above display, we observe that the function $x\mapsto (\log^k x)/x^\delta$ is decreasing for $x \geq e^{k/\delta}$. This implies
		\[
		\int_{p_1/p_2 > M} p_1 \left(\frac{p_1}{p_2}\right)^\delta \cdot \frac{\log^k\!\left(\frac{p_1}{p_2}\right)}{\left(\frac{p_1}{p_2}\right)^\delta} \leq \frac{\log^k M}{M^\delta} \cdot M_\delta.
		\]
		Combining the two displays above, we get:
		\[
		V_k(p_1 \| p_2) \leq (C_k + 8\log^k M)\,\Helsq(p_1, p_2) + \frac{\log^kM}{M^\delta} \cdot M_\delta.
		\]
		Next we choose $M := 8 \vee e^{k/\delta} \vee e^{C_k^{1/k}} \vee \left(M_\delta/\{5\,\Helsq(p_1, p_2)\}\right)^{1/\delta}$. Let $a_{k,\delta} := 8 \vee e^{k/\delta} \vee e^{C_k^{1/k}}$.
		Then we have
		\begin{align*}
			V_k(p_1 \| p_2) &\leq \log^kM\left(5\,\Helsq(p_1, p_2) + M^{-\delta} \cdot M_\delta\right) \\
			&\leq 10\left(\left|\log a_{k,\delta}\right|^k \vee \delta^{-k} \left|\log\!\left(\frac{M_\delta}{5\,\Helsq(p_1, p_2)}\right)\right|^k\right) \Helsq(p_1, p_2).
		\end{align*}
		This completes the proof.
	\end{proof}

	\begin{proof}[Proof of \cref{lem:reweighted-Vk-bound}]
		Let $M > 4 \vee e^{k/\delta}$ be chosen large enough (to be fixed later). We split the integral in $V_k(q;\, p_1, p_2)$ into four parts.

		\medskip
		\noindent\emph{Case 1.} Consider the event $\{p_1/p_2 > M\}$. Note that
		\[
		\int_{p_1/p_2 > M} q\, |\log \tfrac{p_1}{p_2}|^k = \int_{p_1/p_2 > M} q \left(\frac{p_1}{p_2}\right)^\delta \cdot \frac{\log^k\!\left(\frac{p_1}{p_2}\right)}{\left(\frac{p_1}{p_2}\right)^\delta} \leq \frac{\log^kM}{M^\delta} \cdot T_\delta,
		\]
		where the last inequality uses the fact that $x\mapsto (\log^k x)/x^\delta$ is decreasing in $x$ for $x \geq e^{k/\delta}$.

		\medskip
		\noindent\emph{Case 2}. Consider the event $\{p_1/p_2 < \frac{1}{M}\}$. By the same argument as in Case~1, we have:
		\[
		\int_{p_1/p_2 < 1/M} q\, |\log \tfrac{p_1}{p_2}|^k = \int_{p_1/p_2 < 1/M} q \left(\frac{p_2}{p_1}\right)^\delta \cdot \frac{\log^k\!\left(\frac{p_2}{p_1}\right)}{\left(\frac{p_2}{p_1}\right)^\delta} \leq \frac{\log^kM}{M^\delta} \cdot T_\delta.
		\]

		\medskip
		\noindent\emph{Case 3.} Consider the event $\{\frac{1}{4} < \frac{p_1}{p_2} \leq M\}$. Note that for $x \geq 1/4$, there exists a constant $C_k$ such that $|\log x|^k \leq C_k(\sqrt{x} - 1)^2$. As a result, we note that
		\begin{align*}
			\int_{1/4 \leq p_1/p_2 \leq M} q\, |\log \tfrac{p_1}{p_2}|^k &= \int_{\substack{q \leq 4p_2,\\ 1/4 \leq p_1/p_2 \leq M}} q\, |\log \tfrac{p_1}{p_2}|^k + \int_{\substack{q > 4p_2,\\ 1/4 \leq p_1/p_2 \leq M}} q\, |\log \tfrac{p_1}{p_2}|^k \\
			&\leq 4C_k \int p_2 \left(\sqrt{\frac{p_1}{p_2}} - 1\right)^2 + \log^kM \int_{q > 4p_2} q \\
			&\leq 4C_k\,\Helsq(p_1, p_2) + 8\log^kM\,\Helsq(q, p_2).
		\end{align*}
		Here the last inequality follows from \cref{lem:hellinger-integral-bound}.

		\medskip
		\noindent\emph{Case 4.} Consider the event $\{\frac{1}{M} \leq \frac{p_1}{p_2} < \frac{1}{4}\}$. By an argument similar to that in Case~3, we have:
		\begin{align*}
			\int_{1/M \leq p_1/p_2 < 1/4} q\, |\log \tfrac{p_1}{p_2}|^k &= \int_{\substack{4 < p_2/p_1 \leq M}} q\, \log^k\!\left(\frac{p_2}{p_1}\right) \\
			&= \int_{\substack{q \leq 4p_1,\\ 4 < p_2/p_1 \leq M}} q\, \log^k\!\left(\frac{p_2}{p_1}\right) + \int_{\substack{q > 4p_1,\\ 4 < p_2/p_1 \leq M}} q\, \log^k\!\left(\frac{p_2}{p_1}\right) \\
			&\leq 4C_k \int p_1 \left(\sqrt{\frac{p_2}{p_1}} - 1\right)^2 + \log^kM \int_{q > 4p_1} q \\
			&\leq 4C_k\,\Helsq(p_1, p_2) + 8\log^k M\,\Helsq(q, p_1) \\
			&\leq (4C_k + 16\log^kM)\,\Helsq(p_1, p_2) + 16\log^kM\,\Helsq(q, p_2).
		\end{align*}
		\medskip
		Combining the observations from Cases~1 through~4, we get:
		\[
		V_k(q;\, p_1, p_2) \leq 2T_\delta \cdot \frac{\log^kM}{M^\delta} + 16(C_k + \log^kM)\,\Helsq(p_1, p_2) + 24\log^kM\,\Helsq(q, p_2).
		\]
		Next we choose
		\[
		M > 24 \vee e^{k/\delta} \vee e^{C_k^{1/k}} \vee \left(\frac{T_\delta}{\Helsq(p_1, p_2) + \Helsq(q, p_2)}\right)^{1/\delta}
		\]
		and set $a_{k,\delta} := 24 \vee e^{k/\delta} \vee e^{C_k^{1/k}}$. We then have:
		\begin{align*}
			&\;\;\;\; V_k(q;\, p_1, p_2) \nonumber \\ &\leq 48\log^kM\left(\frac{T_\delta}{M^\delta} + \Helsq(p_1, p_2) + \Helsq(q, p_2)\right) \\
			&\leq 96 \left(|\log(a_{k,\delta})|^k \vee \delta^{-k} \left|\log\!\left(\frac{T_\delta}{\Helsq(p_1, p_2) + \Helsq(q, p_2)}\right)\right|^k\right) \left(\Helsq(p_1, p_2) + \Helsq(q, p_2)\right).
		\end{align*}
		This completes the proof.
	\end{proof}

	\begin{proof}[Proof of \cref{lem:reweighted-KL-hellinger}]
		Note that
		\begin{align*}
			\KL(q;\, p_1 | p_2) &= \int (\sqrt{q} - \sqrt{p_2})(\sqrt{q} + \sqrt{p_2})\,\log \frac{p_1}{p_2} + \int p_2 \log \frac{p_1}{p_2} \\
			&\leq \left(\int (\sqrt{q} - \sqrt{p_2})^2\right)^{1/2} \left(2\int (q + p_2)\log^2 \frac{p_1}{p_2}\right)^{1/2} + \int p_2 \log \frac{p_1}{p_2} \\
			&= \sqrt{2}\,\Hel(q, p_2)\,\sqrt{V_2(q;\, p_1, p_2) + V_2(p_2 \| p_1)} + \int p_2 \log \frac{p_1}{p_2}.
		\end{align*}
		Next we note that $\log(x) \leq 2(\sqrt{x} - 1)$ for $x \geq 0$. As a result,
		\[
		\int p_2 \log \frac{p_1}{p_2} \leq 2\int p_2 \left(\sqrt{\frac{p_1}{p_2}} - 1\right) = -2\int (1 - \sqrt{p_1 p_2}) = -2\Helsq(p_1, p_2).
		\]
		This completes the proof.
	\end{proof}

\section{Proofs of results in Section~\ref{sec:indep_measurements}}\label{sec:pfapp}

\subsection{Edgeworth expansions of smooth functions}

Let
\begin{equation}\label{eq:xi_avg_rev}
  X_i = \theta_i + \frac{1}{J_i}\sum_{j=1}^{J_i} \varepsilon_{ij},
\end{equation}
where $\E[\varepsilon_{ij} \mid \theta_i] = 0$ and the $\varepsilon_{i1}, \varepsilon_{i2}, \ldots , \varepsilon_{i J_i}$'s are jointly independent conditional on $\theta_i$.  Suppose $\theta_i \overset{iid}{\sim} \pst$ with
$\pst \in \mathcal{P}([-M, M])$ for some fixed $M > 0$.

\begin{prop}[Moment matching and approximate normal convolution]
\label{prop:moment-normal-convolution}
Assume $J_i \ge 2$.
Under model \eqref{eq:xi_avg_rev}, suppose that
$\theta_i\overset{\mathrm{iid}}{\sim}\pistar$ and that
$\varepsilon_{i1},\ldots,\varepsilon_{iJ_i}$ are jointly independent conditional
on $\theta_i$. Suppose also that
$\Var[\varepsilon_{ij}\mid\theta_i]=\sigma_i^2J_i$, where the deterministic
$\sigma_i^2$ satisfy \cref{assump:lowerbound}, and that
$\varepsilon_{ij}/\sqrt{J_i}$ conditional on $\theta_i$ are uniformly
sub-Gaussian. Suppose further that, for a fixed integer $k\ge2$ and every
$1\le \ell\le k$,
\begin{equation*}
  \E\!\left[\left(\frac{\varepsilon_{ij}}{\sqrt{J_i}}\right)^\ell
  \mid \theta_i\right]
  =\sigma_i^\ell\E(Z^\ell),\qquad Z\sim\cN(0,1).
\end{equation*}
Then there exist fixed constants $T\ge1$, $c_1,c_2>0$, and
$C<\infty$, uniform in $i$ and $n$, such that every $\mu_i$ is an
$(r_{1,i},r_{2,i})$-approximate normal convolution of $\pistar$ at scale
$\sigma_i$ in the sense of \cref{asn:assumeindep}, where $r_{1,i}$ and
$r_{2,i}$ may be chosen to satisfy
\begin{equation*}
  r_{1,i}+r_{2,i}\le C {\frac{\log^2 J_i}{J_i^{k-1}}}.
\end{equation*}
\end{prop}

We prove \cref{prop:moment-normal-convolution} using the derivative bounds
and Edgeworth expansion below.

	\begin{lemma}[Derivative and cumulant bounds]\label{lem:deriv-cumulant-bound}
		Let $\pi \in \cP([-M,M])$ and $\sigma > 0$. For each $x \in \R$, let $\kappa_k(\theta \mid X = x)$ denote the $k$-th cumulant of the posterior distribution of~$\theta$ given $X = x$.
		\begin{enumerate}[(a)]
			\item For all $x \in \R$ and $k \geq 1$,
			\[
			\frac{d^k}{dx^k}\log f_{\pi,\sigma}(x) = \begin{cases}
				\displaystyle\frac{\E[\theta \mid X = x] - x}{\sigma^2} & \text{if } k = 1,\\[8pt]
				\displaystyle\frac{\Var(\theta \mid X = x) - \sigma^2}{\sigma^4} & \text{if } k = 2,\\[8pt]
				\displaystyle\frac{\kappa_k(\theta \mid X = x)}{\sigma^{2k}} & \text{if } k \geq 3.
			\end{cases}
			\]
			\item For $k \geq 2$, the derivatives are uniformly bounded in~$x$. For $k = 2$,
			\[
			\sup_{x \in \R}\left|\frac{d^2}{dx^2}\log f_{\pi,\sigma}(x)\right| \leq \frac{M^2}{\sigma^4} + \frac{1}{\sigma^2},
			\]
			and for $k \geq 3$,
			\[
			\sup_{x \in \R}\left|\frac{d^k}{dx^k}\log f_{\pi,\sigma}(x)\right| \leq \frac{k!}{2}\left(\frac{2M}{\sigma^2}\right)^k.
			\]
		\end{enumerate}
	\end{lemma}

	\begin{proof}[Proof of \cref{lem:deriv-cumulant-bound}]
		Factoring the Gaussian kernel,
		\[
		\phi(x - \theta;\sigma) = \frac{1}{\sqrt{2\pi}\,\sigma}\exp\!\left(-\frac{x^2}{2\sigma^2}\right)\exp\!\left(\frac{x\theta}{\sigma^2}\right)\exp\!\left(-\frac{\theta^2}{2\sigma^2}\right),
		\]
		so that $f_{\pi,\sigma}(x) = \frac{1}{\sqrt{2\pi}\,\sigma}\,e^{-x^2/(2\sigma^2)}\,\widetilde{M}(x/\sigma^2)$, where $\widetilde{M}(t) := \int e^{t\theta}\,d\tilde{\pi}(\theta)$. Here $\widetilde{M}(t)$ can be viewed as the moment-generating function of the tilted measure $d\tilde{\pi}(\theta) \propto e^{-\theta^2/(2\sigma^2)}\,d\pi(\theta)$ up to some normalizing constant (which does not depend on $x$). Taking logarithms,
		\begin{equation}\label{eq:log-fpi-cgf}
			\log f_{\pi,\sigma}(x) = -\frac{x^2}{2\sigma^2} + \Lambda(x/\sigma^2) + C_0,
		\end{equation}
		where $\Lambda(t)$ is the cumulant-generating function of~$\tilde{\pi}$ and $C_0$ is a constant free of $x$.

		The exponentially tilted measure $e^{t\theta-\Lambda(t)}\,d\tilde{\pi}(\theta)$ evaluated at $t = x/\sigma^2$ is precisely the posterior $\pi(\theta \mid X = x)$, so $\Lambda^{(k)}(x/\sigma^2) = \kappa_k(\theta \mid X = x)$. Differentiating~\eqref{eq:log-fpi-cgf} $k$ times with respect to~$x$, we find that the quadratic term $-x^2/(2\sigma^2)$ contributes $-x/\sigma^2$ when $k = 1$, $-1/\sigma^2$ when $k = 2$, and zero when $k \geq 3$, while the chain rule gives $\frac{d^k}{dx^k}\Lambda(x/\sigma^2) = \sigma^{-2k}\,\kappa_k(\theta \mid X = x)$. This yields part~(a).

		For part~(b), the posterior $\theta \mid X = x$ is supported on $[-M, M]$ for every~$x$. Let $Y := \theta - \E[\theta \mid X = x]$ denote the centered posterior variable, so $|Y| \leq 2M$, $\E[Y] = 0$, and $\kappa_j(Y) = \kappa_j(\theta \mid X = x)$ for $j \geq 2$. Write $m_j := \E[Y^j \mid X = x]$, and note $m_0 = 1$, $m_1 = 0$, and $|m_j| \leq (2M)^j$ for all $j \geq 0$.

		We prove $|\kappa_k| \leq \frac{k!}{2}\,(2M)^k$ for all $k \geq 2$ by induction using a standard recursive moment--cumulant relation, see e.g.,~\citet[Theorem 2]{morris1982natural},
		\begin{equation}\label{eq:moment-cumulant}
			\kappa_k = m_k - \sum_{i=1}^{k-1}\binom{k-1}{i}\,\kappa_{k-i}\,m_i.
		\end{equation}
		For the base case $k = 2$: $\kappa_2 = m_2$, so $|\kappa_2| \leq (2M)^2 = \frac{2!}{2}\,(2M)^2$. In fact we also note that $\kappa_3=m_3$, so $|\kappa_3|\le (2M)^3\le \frac{3!}{2}(2M)^3$.

		For the inductive step, assume $|\kappa_j| \leq \frac{j!}{2}\,(2M)^j$ for $3 \leq j \leq k - 1$. Since $m_1 = 0$ and $\kappa_1 = 0$, the terms $i = k - 1$ and $i = 1$ in~\eqref{eq:moment-cumulant} vanish, giving
		\begin{align*}
		|\kappa_k| &\leq |m_k| + \sum_{i=2}^{k-2}\binom{k-1}{i} \;|\kappa_{k-i}|\;|m_i| \\ &\leq (2M)^k + \sum_{i=2}^{k-2}\binom{k-1}{i}\,\frac{(k-i)!}{2}\,(2M)^{k-i}\cdot (2M)^i \\ &= (2M)^k\!\left(1 + \frac{1}{2}\sum_{i=2}^{k-2}\binom{k-1}{i}(k-i)!\right).
		\end{align*}
		To complete the induction, we need $1 + \frac{1}{2}\sum_{i=2}^{k-2}\binom{k-1}{i}(k-i)! \leq \frac{k!}{2}$. Simplifying the binomial coefficient:
		\[
		\sum_{i=2}^{k-2}\binom{k-1}{i}(k-i)! = (k-1)!\sum_{i=2}^{k-2}\frac{k-i}{i!} \leq (k-1)!\,k\sum_{i=2}^{\infty}\frac{1}{i!} = k!\,(e - 2).
		\]
		For $k \geq 4$, the bound gives $1 + \frac{k!(e-2)}{2}$, and this is at most $\frac{k!}{2}$ provided $1 \leq \frac{k!(3-e)}{2}$, which holds since $\frac{4!(3-e)}{2} > 3$. Hence $|\kappa_k| \leq \frac{k!}{2}\,(2M)^k$ for all $k \geq 2$.

		For the full derivative bound when $k \geq 3$, part~(a) gives $\left|\frac{d^k}{dx^k}\log f_{\pi,\sigma}(x)\right| = |\kappa_k|/\sigma^{2k} \leq \frac{k!}{2}(2M/\sigma^2)^k$. The bound for $k=2$ is immediate from part (a).
	\end{proof}

For a function $f\in C^r(\R)$, $p\geq0$, and $0<\alpha\leq1$, define the following
weighted H\"older seminorm
\begin{equation*}
  \mathcal L_{r,p,\alpha}(f)
  :=\sup_{x\neq y}
  \frac{|f^{(r)}(x)-f^{(r)}(y)|}
       {|x-y|^\alpha(1+|x|^p+|y|^p)}.
\end{equation*}
We recall a result of \citet[case~(ii) of the first Theorem, Eq.~(8), and the subsequent
remark] {barbour1986asymptotic}, reproduced in slightly different notation. For a scaled
sample mean $S$ of independent terms and its normally distributed counterpart $Z$, this
result expands $\E[f(S)]-\E[f(Z)]$ into terms involving cumulants of $S$ and
$\E[H_d(Z)f(Z)]$, for $H_d(\cdot)$ the order-$d$ Hermite polynomial.

\begin{thm}[Edgeworth expansion for smooth $f$]
\label{thm:barbour-smooth-edgeworth}
Let $Y_1,\ldots,Y_N$ be independent, centered random variables such that
$\sum_{i=1}^N\E Y_i^2=1$, and write $S=\sum_{i=1}^NY_i$.  Fix an integer
$a\geq3$, $p\geq0$, and $0<\alpha\leq1$.  Suppose that
$f\in C^{a-1}(\R)$,
$\mathcal L_{a-1,p+1,\alpha}(f)<\infty$, and
$\E|Y_i|^{a+p+\alpha}<\infty$ for every $i$.  If $\kappa_\ell(S)$ is the
$\ell$th cumulant of $S$ and $H_d$ is the degree-$d$ probabilists' Hermite
polynomial, then
\begin{align}
  \E f(S)
  ={}&\E f(Z)
  +\sum_{\substack{m_3,\ldots,m_a\geq0\\
                    1\leq\sum_{\ell=3}^a(\ell-2)m_\ell\leq a-2}}
    \left\{
      \prod_{\ell=3}^a
      \frac{1}{m_\ell!}
      \left(\frac{\kappa_\ell(S)}{\ell!}\right)^{m_\ell}
    \right\}
    \E\!\left[
      H_{\sum_{\ell=3}^a\ell m_\ell}(Z)f(Z)
    \right] +R_a,
  \label{eq:barbour-edgeworth-expansion}
\end{align}
where $Z\sim\cN(0,1)$ and
\begin{equation}
  |R_a|
  \leq C\mathcal L_{a-1,p+1,\alpha}(f)
  \sum_{i=1}^N
  \left\{\E|Y_i|^{a+\alpha}+\E|Y_i|^{a+p+\alpha}\right\}.
  \label{eq:barbour-edgeworth-remainder}
\end{equation}
Here $C$ depends only on $a,p$, and $\alpha$.
\end{thm}

\begin{remark}[First-order Edgeworth expansion]
The generality of Theorem~\ref{thm:barbour-smooth-edgeworth} is useful in establishing higher-order adaptation (see~\cref{asub:symmetric}). Our main result in the independent measurement regime uses the special case with $a=3$, $p=0$, and $\alpha=1$. In that case, the inequality of Theorem~\ref{thm:barbour-smooth-edgeworth} simplifies to the following inequality, which we informally previewed in~\eqref{eq:steinexpansion},
\[
  \left|
    \E f(S)-\E f(Z)
    -\frac{\kappa_3(S)}{6}\E[H_3(Z)f(Z)]
  \right|
  \le C\mathcal L_{2,1,1}(f)
       \sum_{i=1}^N \E|Y_i|^4,
\]
where $H_3(z)=z^3-3z$, $\kappa_3(S)=\sum_{i=1}^N\E Y_i^3$,
and $C$ is a universal constant.
\end{remark}

\begin{proof}[{Proof of \cref{prop:moment-normal-convolution}}]

Let $Z_i \mid \theta_i \sim \Norm(\theta_i, \sigma_i^2)$.
Fix $i$ and set $\xi_j:=\varepsilon_{ij}/\sqrt{J_i}$, so that
$X_i=\theta_i+J_i^{-1/2}\sum_{j=1}^{J_i}\xi_j$.  We first record a direct
consequence of \cref{thm:barbour-smooth-edgeworth} that will be applied to
both $h=t_{\pi,i}$ and $h=t_{\pi,i}^2$.  For any $h\in C^{2k-1}(\R)$ such
that
\[
  |h^{(2k-1)}(x)|\leq L(1+|x|),\qquad x\in\R,
\]
condition on $\theta_i=\theta$ and define
\[
  Y_j:=\frac{\xi_j}{\sigma_i\sqrt{J_i}},
  \quad S_i:=\sum_{j=1}^{J_i}Y_j,
  \quad f_{\theta,h}(x):=h(\theta+\sigma_i x).
\]
The variables $Y_j$ are conditionally independent and centered, with
$\sum_j\E(Y_j^2\mid\theta_i)=1$.  Conditional on $\theta_i=\theta$,
$f_{\theta,h}(S_i)=h(X_i)$, while $f_{\theta,h}(Z)$ has the same law as
$h(Z_i)$.  The bound on $|h^{(2k-1)}|$, compact support for $\theta$, and
\cref{assump:lowerbound} give
$\mathcal L_{2k-2,1,1}(f_{\theta,h})\lesssim L$.
Apply \cref{thm:barbour-smooth-edgeworth} conditionally with $a=2k-1$,
$p=0$, and $\alpha=1$ and simplify to obtain:
\[
\mathbb E\!\left[h(X_i)\mid\theta_i=\theta\right]
-\mathbb E\!\left[h(\theta+\sigma_i Z)\right] =
\sum_{\ell=k+1}^{2k-1}
\frac{\kappa_\ell(S_i\mid\theta_i=\theta)}{\ell!}
\mathbb E\!\left[
H_\ell(Z)h(\theta+\sigma_i Z)
\right]
+R_{2k-1}.
\]
Uniform conditional
sub-Gaussianity and additivity of cumulants give
\[
  |\kappa_\ell(S_i\mid\theta_i)|\lesssim J_i^{1-\ell/2},
  \quad
  |R_{2k-1}|\lesssim \frac{L}{J_i^{k-1}}.
\]
Finally, Cauchy--Schwarz and $\E H_\ell(Z)^2=\ell!$ bound each Hermite
coefficient by
\[
  \big|\E\{H_\ell(Z)f_{\theta,h}(Z)\mid\theta_i\}\big|
  \leq \sqrt{\ell!}\,
  \{\E[h^2(\theta_i+\sigma_iZ)\mid\theta_i]\}^{1/2}.
\]
Integrating over $\theta_i$ and using Jensen's inequality yields
\begin{equation}
  |\E h(X_i)-\E h(Z_i)|
  \lesssim
  \frac{\{\E h^2(Z_i)\}^{1/2}}{J_i^{(k-1)/2}}
  +\frac{L}{J_i^{k-1}}.
  \label{eq:moment-barbour-comparison}
\end{equation}
By \eqref{eq:ratio-growth}, compact support, and
\cref{assump:lowerbound}, the expression \eqref{eq:exponential_ratio} in
\cref{lem:Vk-hellinger-bound} is bounded uniformly over $i$ and $\pi$. Accordingly, take
one fixed $T\ge {e^4}$ in all applications below for $g_T(\cdot)$.

We verify \eqref{eq:condition-a} and \eqref{eq:condition-b}
uniformly over $1\le i\le n$.
For $h=t_{\pi,i}$, \cref{lem:deriv-cumulant-bound} gives the
derivative bound required by \eqref{eq:moment-barbour-comparison}, uniformly
over $\pi$ and $i$, while \cref{lem:Vk-hellinger-bound} bounds its second
moment under the Gaussian convolution. Consequently,
\begin{align}
  \left|
	  \E_{X_i\sim\mu_i}t_{\pi,i}(X_i)-\E t_{\pi,i}(Z_i)
  \right|   &\lesssim
  \frac{1}{J_i^{(k-1)/2}}
  \sqrt{g_T\!\left(
    \Helsq(f_{\pi,\sigma_i},f_{\pistar,\sigma_i})
  \right)}
  +\frac{1}{J_i^{k-1}}. \notag
\\ &\le
  \eta\Helsq(f_{\pi,\sigma_i},f_{\pistar,\sigma_i})
  +C_\eta {\frac{\log^2 J_i}{J_i^{k-1}}}. \label{eq:edgeworth-log-ratio-young}
\end{align}
For the last inequality, we break $\Helsq(f_{\pi,\sigma_i},f_{\pistar,\sigma_i})$ at $J_i^{1-k}$.
If $u:= \Helsq(f_{\pi,\sigma_i},f_{\pistar,\sigma_i})\le J_i^{1-k}$, we bound the increasing function $u\mapsto \sqrt u\log(T/u)$
by  $J_i^{1-k}$.
Otherwise if $u > J_i^{1-k}$, we take $\log(T/u)\le\log(TJ_i^{k-1})\lesssim\log J_i$ and
apply Young's inequality.

For $\E[t_{\pi,i}(Z_i)]$, observe that
	\begin{equation}\label{eq:negterm}
	  \E\,t_{\pi,i}(Z_i)
  = \int f_{\pst,\sigma_i}(x)\log\frac{f_{\pi,\sigma_i}(x)}{f_{\pst,\sigma_i}(x)}\,dx
  \le -\Hel^2\!\bigl(f_{\pi,\sigma_i},\, f_{\pst,\sigma_i}\bigr),
\end{equation}
which follows by invoking \cref{lem:reweighted-KL-hellinger} with
$q = f_{\pst,\sigma_i} = p_2$ and $p_1 = f_{\pi,\sigma_i}$.
Combining \eqref{eq:edgeworth-log-ratio-young} with
\eqref{eq:negterm}, and fixing $\eta$ sufficiently small, verifies
\eqref{eq:condition-a} for $r_{1i} \lesssim
{\frac{\log^2 J_i}{J_i^{k-1}}}$.

For \eqref{eq:condition-b}, take $h=t_{\pi,i}^2$ in
\eqref{eq:moment-barbour-comparison}.  By \eqref{eq:ratio-growth} and
\cref{lem:deriv-cumulant-bound},
\[
  \left|(t_{\pi,i}^2)^{(2k-1)}(x)\right|\lesssim 1+|x|,
\]
as required by our application of \cref{thm:barbour-smooth-edgeworth}. Also,
\cref{lem:Vk-hellinger-bound} gives\footnote{When the squared Hellinger distance is zero,
the right-hand side is zero.}
\begin{equation*}
	  \E t_{\pi,i}^4(Z_i)
  \lesssim
  \Helsq(f_{\pi,\sigma_i},f_{\pistar,\sigma_i})
  \log^4\!\left(
    \frac{T}{\Helsq(f_{\pi,\sigma_i},f_{\pistar,\sigma_i})}
  \right).
\end{equation*}
We then analogously have
\begin{align*}\left|
	  \E_{X_i\sim\mu_i}t_{\pi,i}^2(X_i)-\E t_{\pi,i}^2(Z_i)
  \right|
 \le
  \eta g_T\!\left(
    \Helsq(f_{\pi,\sigma_i},f_{\pistar,\sigma_i})
  \right)
  +C_\eta {\frac{\log^2 J_i}{J_i^{k-1}}}.
\end{align*}
This inequality follows analogously to \eqref{eq:edgeworth-log-ratio-young}, where we compare
$u:=u:= \Helsq(f_{\pi,\sigma_i},f_{\pistar,\sigma_i})$ to $J_i^{-(k-1)/2}\sqrt u\log^2(T/u)$.

{A further application of \cref{lem:Vk-hellinger-bound} bounds
$\E t_{\pi,i}^2(Z_i)$ by a constant multiple of
$g_T\{\Helsq(f_{\pi,\sigma_i},f_{\pistar,\sigma_i})\}$. Thus
$r_{2,i}$ may be chosen to satisfy the asserted bound and, for a fixed
$c_2>0$,}
\begin{equation*}
  \E_{X_i\sim\mu_i}t_{\pi,i}^2(X_i)
  \le c_2\left\{g_T\!\left(\Helsq(f_{\pi,\sigma_i},f_{\pistar,\sigma_i})\right)
  +r_{2,i}\right\}.
\end{equation*}
This verifies \eqref{eq:condition-b} and completes the proof.
\end{proof}

\subsection{Estimated variance and approximate NPMLE}

We now verify that the approximate-oracle likelihood condition in
\eqref{eq:likelihood-bound} holds when we maximize likelihood by plugging in
the estimated variance for $\sigma_i^2$ in \eqref{eq:approximate_mle}.  Write the
unclipped sample variance and its estimation error as
\begin{equation}
  \hat\varsigma_i^2
  :=
  \frac{1}{J_i(J_i-1)}
  \sum_{j=1}^{J_i}(X_{ij}-X_i)^2,
  \quad
  \Delta_i:=\hat\varsigma_i^2-\sigma_i^2.
  \label{eq:estimated-variance-definition}
\end{equation}
Recall from \cref{sec:indep_measurements} that $\hat\sigma_i^2$ is the
projection of $\hat\varsigma_i^2$ onto $[k^2/2,2K^2]$.

We shall analyze the likelihood perturbation
\begin{align*}
  U_\sigma(\pi)
  :=\avgn\bigg\{&
    \log f_{\pi,\hat\sigma_i}(X_i)
    -\log f_{\pi,\sigma_i}(X_i)
 -
    \log f_{\pistar,\hat\sigma_i}(X_i)
    +\log f_{\pistar,\sigma_i}(X_i)
  \bigg\}
\end{align*}

\begin{prop}[Estimated variances and the oracle likelihood]
\label{prop:estimated-variance-oracle-likelihood}
Under the preceding conditions, for every fixed sufficiently small $\eta>0$,
with probability at least
$1-n^{-2}$,
\begin{equation*}
  \sup_{\pi\in\cP([-M,M])}
  \left\{
    |U_\sigma(\pi)|
    -\eta\Helbarstar{f_\pi}{f_{\pistar}}
  \right\}
  \leq
  C_\eta\log^4 n
  \left\{\frac1n+\avgn\frac1{J_i}\right\}.
\end{equation*}
Consequently, any $\pihat$ satisfying
\eqref{eq:approximate_mle} is a $b_n$-approximate
oracle NPMLE in the sense of \eqref{eq:likelihood-bound}, where one may take
\begin{equation}
  b_n
  :=
  \log^4 n
  \left\{\frac1n+\avgn\frac1{J_i}\right\}.
  \label{eq:estimated-variance-bn}
\end{equation}
The constant $C_\eta$ depends only on $\eta$, $M$, the sub-Gaussian constant $C$ in \cref{thm:main_rates} and $(k,K)$ in \cref{assump:lowerbound}.
\end{prop}

\begin{proof}[Proof of \cref{prop:estimated-variance-oracle-likelihood}]

Define the first-order effect of perturbing $\sigma_i$ by
\begin{equation}
  S_{\pi,i}(x)
  :=
  \left.
  \partial_v
  \left\{
    \log f_{\pi,\sqrt v}(x)
    -\log f_{\pistar,\sqrt v}(x)
  \right\}
  \right|_{v=\sigma_i^2}.
  \label{eq:estimated-variance-score}
\end{equation}
Taylor expansion in $v=\sigma^2$ gives, for every
$\pi\in\cP([-M,M])$,
\begin{align}
  U_\sigma(\pi)
  &=
  \avgn(\hat\sigma_i^2-\sigma_i^2)
  S_{\pi,i}(X_i)
  \label{eq:u1}\\
  &\quad+
  \avgn(\hat\sigma_i^2-\sigma_i^2)^2
  \int_0^1(1-a)
  \bigg[
    \partial_v^2\log f_{\pi,\sqrt v}(X_i)
    -\partial_v^2\log f_{\pistar,\sqrt v}(X_i)
  \bigg]_{
    v=\sigma_i^2+a(\hat\sigma_i^2-\sigma_i^2)
  }\,da.
  \label{eq:estimated-variance-taylor}
\end{align}
We further decompose \eqref{eq:u1}
\begin{align}
  \eqref{eq:u1}&=\underbrace{\avgn\left[
      \Delta_iS_{\pi,i}(X_i)
      -\E\{\Delta_iS_{\pi,i}(X_i)\}
    \right]}_{\text{\cref{lem:estimated-variance-centered}}}
  +\avgn\E\{\Delta_iS_{\pi,i}(X_i)\}
  \nonumber\\
  &\quad+
  \avgn
  (\hat\sigma_i^2-\hat\varsigma_i^2)
  S_{\pi,i}(X_i).
  \label{eq:estimated-variance-first-order}
\end{align}
The proof proceeds by analyzing and bounding each term, delegated to lemmas in the
sequel:
\begin{enumerate}[wide]

\item The most nontrivial term is the second term in
\eqref{eq:estimated-variance-first-order}, due to the fact that $\Delta_i$ and $X_i$ are
not independent. For that term, we first apply  a Taylor expansion and an
Edgeworth expansion for smooth functions in \cref{thm:barbour-smooth-edgeworth} to show
that \[
	\E[\Delta_i S_{\pi,i}(X_i) \mid \theta_i] \approx \frac{1}{\sqrt{J_i}} \underbrace{\E
		\bk{\pr{\frac{\varepsilon_
			{ij}}{\sqrt{J}}}^3 \mid \theta_i}}_{\kappa_i(\theta_i)} \E\bk{S_{\pi,i}'(\theta_i +
			\sigma_i Z) \mid \theta_i}
		\quad Z \mid \theta_i \sim \Norm(0,1)
\] Write $X^*_i = \theta_i + \sigma_i Z$ as the normal version of $X_i$. After further
 integration over $\theta_i \sim \pistar$, the marginal expectation can then be written
 as an integral under the normal mixture $f_ {\pistar,
\sigma_i}$:
\begin{align*}
\E_ {\pistar, Z}
[\kappa_i (\theta_i) S_ {\pi,i}' (\theta_i +
			\sigma Z)] &= \E[\underbrace{\E[\kappa_i(\theta_i) \mid X_i^*]}_{=:\chi_i(X_i^*)} S_
			{\pi,i}' (X_i^*)] = \int \chi_i(x)  f_
			{\pistar, \sigma_i}(x) S_{\pi, i}'(x) \,dx  \\
			&= -\int (\chi_i' + \chi_i f_{\pistar, \sigma_i}'/ f_{\pistar, \sigma_i}) f_
			{\pistar, \sigma_i} S_{\pi, i} \tag{Integration by parts}
\end{align*}
		Finally, \cref{lem:estimated-variance-score} relates this last integral to the
		Hellinger distance between $f_{\pistar, \sigma_i}$ and $f_{\pi, \sigma_i}$, by
		further integrations by parts and relating it to $\norm{f_\pi - f_{\pistar}}_1$. This
		sequence of arguments is
		formalized in the key \cref{lem:estimated-variance-score-mean}.

		Applying Young's inequality to the conclusion of \cref
		{lem:estimated-variance-score-mean} then gives
\begin{align*}
  \abs[\bigg]{\avgn\E\{\Delta_iS_{\pi,i}(X_i)\}}
  &\leq
  \frac\eta2\Helbarstar{f_\pi}{f_{\pistar}}
  +C_\eta\log^3 n
  \left\{\frac1n+\avgn\frac1{J_i}\right\}.
\end{align*}

	\item For \eqref{eq:estimated-variance-taylor}, projection onto the clipping interval gives
$|\hat\sigma_i^2-\sigma_i^2|\leq|\Delta_i|$. We integrate $v$ in \eqref
{eq:estimated-variance-taylor} over a subset of $[k^2/2,2K^2]$.  Hence
\eqref{eq:estimated-variance-second-derivative} and
\eqref{eq:estimated-variance-event}---consequences of calculations with normal convolution
 densities and sub-Gaussianity---yield
\begin{align*}
  \sup_{\pi\in\cP([-M,M])}
  \abs{\eqref{eq:estimated-variance-taylor}} \leq
  C\log n\avgn\Delta_i^2
  \lesssim
  \frac{\log^4 n}{n}+\log n\avgn\frac1{J_i}.
\end{align*}

\item The first term in \eqref{eq:estimated-variance-first-order} is handled by
\cref{lem:estimated-variance-centered}, which applies a standard maximal inequality to
 control the empirical process.

\item Finally, for the third term in \eqref{eq:estimated-variance-first-order}, on the
event in
\eqref{eq:estimated-variance-event}, which occurs with probability
at least $1-Cn^{-3}$ by sub-Gaussianity, we have
\begin{align*}
  \left|
    \avgn
    (\hat\sigma_i^2-\hat\varsigma_i^2)
    S_{\pi,i}(X_i)
  \right|
  \leq
  \max_i|S_{\pi,i}(X_i)|
  \avgn|\hat\sigma_i^2-\hat\varsigma_i^2|\lesssim\sqrt{\log n}
  \left\{\frac{\log^2 n}{n}+\avgn\frac1{J_i}\right\}.
\end{align*}
\end{enumerate}
Combining the bounds for both terms of the Taylor expansion, with probability at
least $1-Cn^{-3}$, uniformly over $\pi\in\cP([-M,M])$, we obtain
\[
\abs{U_\sigma(\pi)} \le \eta \Helbarstar{f_\pi}{f_{\pistar}}
  +C_\eta\log^4 n
  \left\{\frac1n+\avgn\frac1{J_i}\right\}.
\]
For all sufficiently large $n$, this event has probability at least $1-n^{-2}$.

On this event, we thus have
\begin{align*}
  \avgn\log\frac{f_{\pihat,\sigma_i}(X_i)}
                       {f_{\pistar,\sigma_i}(X_i)}
  &\geq
  -\frac{C\log n}{n} -U_\sigma(\pihat) \geq
  -\eta\Helbarstar{f_{\pihat}}{f_{\pistar}}
  -C_\eta b_n.
\end{align*}
The fixed optimization error $C\log n/n$ is absorbed by
\eqref{eq:estimated-variance-bn}.  This is exactly
\eqref{eq:likelihood-bound}.
\end{proof}

This first lemma shows that $\hat\varsigma_i^2$ concentrates around $\sigma_i^2$.

\begin{lemma}[Bounds related to $\hat\varsigma_i^2$]
\label{lem:estimated-variance-concentration}
Recall that $\Delta_i = \hat\varsigma_i^2 - \sigma_i^2$.  Uniformly over $i$ and $u>0$,
\begin{align}
  \Prob\left(
    |\Delta_i|>u\mid\theta_i
  \right)
  \leq
  2\exp\{-cJ_i\min(u^2,u)\},
  \label{eq:estimated-variance-tail}
\\
  \E(\Delta_i^2\mid\theta_i)\lesssim J_i^{-1},
  \qquad
  \E(\Delta_i^4\mid\theta_i)\lesssim J_i^{-2}.
  \label{eq:estimated-variance-moments}
\end{align}
For a sufficiently large fixed $t$, with probability at least $1-Cn^{-3}$,
\begin{align}
  \max_{1\leq i\leq n}|X_i|
  &\leq M+2t\sqrt{\log n},
  \qquad
  \max_{1\leq i\leq n}|\Delta_i|
  \leq C\log n,
   \numberthis \label{eq:quick_consequence}\\
  \avgn\Delta_i^2
  &\leq
  C\left\{\frac{\log^3 n}{n}+\avgn\frac1{J_i}\right\},
  \nonumber\\
  \avgn
  |\hat\sigma_i^2-\hat\varsigma_i^2|
  &\lesssim
  \frac{\log^2 n}{n}+\avgn\frac1{J_i}.
  \label{eq:estimated-variance-event}
\end{align}
\end{lemma}

\begin{proof}
Rearrange \eqref{eq:estimated-variance-definition} to obtain
\begin{align*}
  \Delta_i
  =\frac1{J_i-1}\bigg\{&
    \sum_{j=1}^{J_i}
    \left(\frac{\varepsilon_{ij}^2}{J_i}-\sigma_i^2\right)
    -\left((X_i-\theta_i)^2-\sigma_i^2\right)
  \bigg\}.
\end{align*}
Conditional
on $\theta_i$, the terms
$\varepsilon_{ij}^2/J_i-\sigma_i^2$ are independent, centered, and uniformly
sub-exponential, while $X_i-\theta_i$ is uniformly sub-Gaussian.  Applying
the sub-exponential Bernstein inequality to the first term in $\Delta_i$ and applying the sub-Gaussian
tail bound to $X_i-\theta_i$ give
\eqref{eq:estimated-variance-tail}; see, for example,
\citet[Chapter~2]{vershynin2018high}.  Consequently, for $p=2,4$,
\begin{align*}
  \E(|\Delta_i|^p\mid\theta_i)
  &=p\int_0^\infty u^{p-1}
    \Prob(|\Delta_i|>u\mid\theta_i)\,du
  \lesssim J_i^{-p/2},
\end{align*}
which proves \eqref{eq:estimated-variance-moments}. This and sub-Gaussianity of $X_i$
immediately prove \eqref{eq:quick_consequence}.

For $\avgn \Delta_i^2$, apply Bernstein's inequality across
$i$ to
$\Delta_i^2\wedge C\log^2n$.  By
\eqref{eq:estimated-variance-moments}, the sum of the variances of these
truncated variables is at most $C\sum_iJ_i^{-2}$.  A union bound using
\eqref{eq:estimated-variance-tail} shows that the truncation does not bind with
probability at least $1-Cn^{-4}$. Thus, with
probability at least $1-Cn^{-4}$,
\begin{align*}
  \avgn\Delta_i^2
  &\leq
  C\avgn\frac1{J_i}
  +C\left\{
    \frac{\log n}{n}\avgn\frac1{J_i^2}
  \right\}^{1/2}
  +C\frac{\log^3n}{n} \lesssim \avgn\frac1{J_i}+\frac{\log^3n}{n}.
\end{align*}
The last inequality uses $J_i\geq2$ and
$\sqrt{ab}\leq(a+b)/2$.

For \eqref{eq:estimated-variance-event}, let $0 < d = \min_i\min(|\sigma_i^2 - k^2/2|,
|\sigma_i^2 - 2K^2|)$ be the minimum distance of the oracle
variances $\sigma_i^2$ from  the two endpoints of $[k^2/2,2K^2]$.  We have
\begin{equation}
  |\hat\sigma_i^2-\hat\varsigma_i^2|
  \leq
  |\Delta_i|\mathbf 1\{|\Delta_i|\geq d\}.
  \label{eq:estimated-variance-clipping-residual}
\end{equation}
Equations \eqref{eq:estimated-variance-tail} and
\eqref{eq:estimated-variance-clipping-residual} imply for $p = 1,2$
\begin{align*}
  &\E|\hat\sigma_i^2-\hat\varsigma_i^2|^p \leq
  d^p\Prob(|\Delta_i|\geq d)
  +\int_d^\infty pu^{p-1}\Prob(|\Delta_i|>u)\,du
  \leq Ce^{-cJ_i}.
\end{align*}
With these moment bounds, apply Bernstein's inequality to
$|\hat\sigma_i^2-\hat\varsigma_i^2|\wedge C\log n$: With probability at least
$1-Cn^{-4}$,
\begin{align*}
  \avgn
  (|\hat\sigma_i^2-\hat\varsigma_i^2|\wedge C\log n)
  &\leq
  C\avgn e^{-cJ_i}
  +C\left\{
    \frac{\log n}{n}\avgn e^{-cJ_i}
  \right\}^{1/2}
  +C\frac{\log^2n}{n}\\
  &\lesssim
  \frac{\log^2 n}{n}+\avgn\frac1{J_i},
\end{align*}
where we used $e^{-cJ_i}\lesssim J_i^{-1}$ and Young's inequality.
Another application of \eqref{eq:estimated-variance-tail} gives
$  \Prob\left\{
    \max_i|\hat\sigma_i^2-\hat\varsigma_i^2|>C\log n
  \right\}
  \leq n^{-4}
$
when $C$ is sufficiently large. Thus the above holds for $|\hat\sigma_i^2-\hat\varsigma_i^2|$ with high probability as well.  This proves
\eqref{eq:estimated-variance-event}.
\end{proof}

The next lemma gives useful results for $S_{\pi, i}$ defined in \eqref{eq:estimated-variance-score}.

\begin{lemma}[Bounds for the $S_{\pi,i}$]
\label{lem:estimated-variance-score}
Recall $S_{n,t}=[-M-2t\sqrt{\log n},M+2t\sqrt{\log n}]$.
For every sufficiently large fixed $t$, uniformly over $i$,
$\pi\in\cP([-M,M])$, and $0\leq r\leq4$,
\begin{equation}
  \sup_{x\in S_{n,2t}}|S_{\pi,i}^{(r)}(x)|
  \lesssim\sqrt{\log n},
  \qquad
  |S_{\pi,i}^{(r)}(x)|\leq C(1+|x|),\quad x\in\R.
  \label{eq:estimated-variance-score-envelope}
\end{equation}
If $w$ is twice continuously differentiable and supported on $S_{n,2t}$, then
\begin{equation}
  \left|
    \int w(x)f_{\pistar,\sigma_i}(x)S_{\pi,i}(x)\,dx
  \right|
  \lesssim
  \left\{\log n\,\|w\|_\infty+\|w''\|_\infty\right\}
  \Hel(f_{\pi,\sigma_i},f_{\pistar,\sigma_i}).
  \label{eq:estimated-variance-score-pairing}
\end{equation}
Moreover, for every $v\in[k^2/2,2K^2]$,
\begin{equation}
  |\partial_v^2\log f_{\pi,\sqrt v}(x)|
  \leq C(1+x^2).
  \label{eq:estimated-variance-second-derivative}
\end{equation}
\end{lemma}

\begin{proof}
The Gaussian heat equation gives
\begin{equation}
  \partial_v f_{\pi,\sqrt v}(x)
  =\frac12\partial_x^2f_{\pi,\sqrt v}(x).
  \label{eq:estimated-variance-heat}
\end{equation}
Using \eqref{eq:estimated-variance-heat}, we may compute
\begin{equation*}
  \left.\partial_v\log f_{\pi,\sqrt v}(x)\right|_{v=\sigma_i^2}
  =
  \frac{1}{2f_{\pi,\sigma_i}(x)}
  \int\left\{
    \frac{(x-\theta)^2}{\sigma_i^4}-\frac1{\sigma_i^2}
  \right\}
  \phi(x-\theta;\sigma_i)\,d\pi(\theta).
\end{equation*}
Expanding the square and subtracting the analogous expression for $\pistar$ gives
\begin{align*}
  S_{\pi,i}(x)
  =\frac1{2\sigma_i^4}\bigl\{
    -2x\bigl(h_\pi(x;\sigma_i)-h_{\pistar}(x;\sigma_i)\bigr) +\E_{\pi,\sigma_i}[\theta^2\mid x]
    -\E_{\pistar,\sigma_i}[\theta^2\mid x]\bigr\},
\end{align*}
where expectations are under the corresponding normal posteriors.
For every bounded function $a$ on $[-M,M]$,
\[
\frac{d}{dx}\E_{\pi,\sigma_i}[a(\theta)\mid x]
=\sigma_i^{-2}\operatorname{Cov}_{\pi,\sigma_i}(a(\theta),\theta\mid x).
\]
Repeated differentiation would express every  derivative of these posterior
expectations in terms of moments of bounded functions of $\theta$. All such derivatives
are uniformly bounded by compact support and \cref{assump:lowerbound}.
Applying this observation to the expression for $S_{\pi,i}$ proves
\eqref{eq:estimated-variance-score-envelope}.

For \eqref{eq:estimated-variance-score-pairing}, add and subtract
$f_{\pi,\sigma_i}\partial_v\log f_{\pi,\sqrt v}|_{v=\sigma_i^2}$:
\begin{align*}
  \int w f_{\pistar,\sigma_i}S_{\pi,i}
  &=
  \int w(f_{\pistar,\sigma_i}-f_{\pi,\sigma_i})
  \left.\partial_v\log f_{\pi,\sqrt v}\right|_{v=\sigma_i^2}+
  \int w\left(
    \left.\partial_v f_{\pi,\sqrt v}\right|_{v=\sigma_i^2}
    -\left.\partial_v f_{\pistar,\sqrt v}\right|_{v=\sigma_i^2}
  \right).
\end{align*}
The first integral has absolute value at most
\begin{equation*}
  \left\|w
  \left.\partial_v\log f_{\pi,\sqrt v}\right|_{v=\sigma_i^2}
  \right\|_\infty
  \|f_{\pi,\sigma_i}-f_{\pistar,\sigma_i}\|_1
  \lesssim\log n\,\|w\|_\infty
  \|f_{\pi,\sigma_i}-f_{\pistar,\sigma_i}\|_1.
\end{equation*}
By \eqref{eq:estimated-variance-heat} and two integrations by parts, the
second integral has absolute value
\begin{equation*}
  \abs[\bigg]{\frac12\int w''(x)
  \{f_{\pi,\sigma_i}(x)-f_{\pistar,\sigma_i}(x)\}\,dx} \le \frac12 \| w'' \|_\infty
\| f_{\pi,\sigma_i}-f_{\pistar,\sigma_i}\|_1.
\end{equation*}
Since $\|p-q\|_1\lesssim\Hel(p,q)$ under our normalization,
\eqref{eq:estimated-variance-score-pairing} follows.

Finally, direct differentiation gives
\[
\partial_v^2\log f_{\pi,\sqrt v}(x)
=\frac1{2v^2}
-\frac{\E_{\pi,\sqrt v}[(x-\theta)^2\mid x]}{v^3}
+\frac{\operatorname{Var}_{\pi,\sqrt v}((x-\theta)^2\mid x)}{4v^4}.
\]
Since
$\operatorname{Var}_{\pi,\sqrt v}((x-\theta)^2\mid x)
=\operatorname{Var}_{\pi,\sqrt v}(-2x\theta+\theta^2\mid x)\lesssim1+x^2$,
the bounds on $\theta$ and $v$ imply \eqref{eq:estimated-variance-second-derivative}.
\end{proof}

The following lemma bounds the expectation in
\eqref{eq:estimated-variance-first-order}.

\begin{lemma}
\label{lem:estimated-variance-score-mean}
For every sufficiently large fixed $t$, uniformly over
$\pi\in\cP([-M,M])$,
\begin{equation*}
\begin{aligned}
  \left|\avgn\E\{\Delta_iS_{\pi,i}(X_i)\}\right| \lesssim
  (\log n)^{3/2}\bigg[
    \left\{\avgn\frac1{J_i}\right\}^{1/2}
    \left\{\Helbarstar{f_\pi}{f_{\pistar}}\right\}^{1/2}
  +\avgn\frac1{J_i}+n^{-4}
  \bigg].
\end{aligned}
\end{equation*}
\end{lemma}

\begin{proof}
We first consider a generic smooth $h$ and $\E[\Delta_i h(X_i) \mid \theta_i]$.
Fix $i$, condition on $\theta_i=\theta$, and write $J=J_i$ and
\begin{equation*}
  \xi_j:=\frac{\varepsilon_{ij}}{\sqrt J},
  \quad
  T:=J^{-1/2}\sum_{j=1}^J\xi_j,
  \quad
  T_{-j}:=J^{-1/2}\sum_{\ell\ne j}\xi_\ell,
\end{equation*}
such that $X_i = \theta + T = \theta + T_{-j} + \xi_j/\sqrt{J}$.
\Cref{eq:estimated-variance-definition} rearranges to
\[
  \Delta_i
  =
  \frac1{J-1}
  \left\{
    \sum_{j=1}^J(\xi_j^2-\sigma_i^2)
    -(T^2-\sigma_i^2)
  \right\}.
\] Therefore, letting $h$ be a smooth function where $h^{(r)} \lesssim (\log n)^C$ for $r
= 0,\ldots, 4$,
we then have \begin{align*}
\E[\Delta_i h(X_i) \mid \theta_i] = \frac{1}{J-1} \br{-\E[(T^2-\sigma_i^2)h(\theta + T)\mid \theta_i] +
\sum_j \E [(\xi_j^2 - \sigma_i^2) h (\theta + T)\mid \theta_i]}.
\numberthis\label{eq:decompproduct}
\end{align*}
The first term above satisfies
\begin{align}
  &\frac1{J-1}
  \left|
    \E\{(T^2-\sigma_i^2)h(\theta+T)\mid\theta_i\}
  \right| \leq
  \frac{\|h\|_\infty}{J-1}
  \left\{
    \E\{(T^2-\sigma_i^2)^2\mid\theta_i\}
  \right\}^{1/2}
  \lesssim (\log n)^C J^{-1}.
  \label{eq:estimated-variance-T2}
\end{align}
The last inequality follows from the sub-Gaussian bound for
$T$.

\paragraph{Second term in \eqref{eq:decompproduct}.} For each $j$, Taylor's formula
around $\theta + T_{-j}$ gives
\begin{align*}
  h(\theta+T)
  &=h(\theta+T_{-j})
  +J^{-1/2}\xi_jh'(\theta+T_{-j}) +
  J^{-1}\xi_j^2
  \int_0^1(1-a)
  h''(\theta+T_{-j}+aJ^{-1/2}\xi_j)\,da.
\end{align*}
Plugging in, observing that $\E[\xi_j \mid \theta_i, T_{-j}] = 0$ and $\E[\xi_j^2 -
\sigma_i^2 \mid \theta, T_ {-j}] = 0$, we have
\[
  \E\bk{(\xi_j^2-\sigma_i^2)h(\theta+T)\mid\theta_i}
  =
 \frac{1}{\sqrt{J}}\E[\xi_j^3\mid\theta_i]
  \E[h'(\theta_i+T_{-j})\mid\theta_i]
  +O\pr{\frac{(\log n)^C}{J}}.
\]
We shall apply \cref{thm:barbour-smooth-edgeworth}.
For $\ell\ne j$, set
\[
  Y_\ell:=\frac{\xi_\ell}{\sigma_i\sqrt{J-1}},
  \quad S_{-j}:=\sum_{\ell\ne j}Y_\ell,
  \quad
  f_\theta(x):=h'\!\left(\theta+\sigma_i\sqrt{1-J^{-1}}\,x\right).
\]
Conditional on $\theta_i=\theta$, the variables $Y_\ell$ are independent and
centered and their variances sum to one, and
$f_\theta(S_{-j})=h'(\theta+T_{-j})$.  Apply
\cref{thm:barbour-smooth-edgeworth} with $a=3$, $p=0$, and $\alpha=1$:
\[
  \E\{f_\theta(S_{-j})-f_\theta(Z)\mid\theta\}
  =\frac{\kappa_3(S_{-j}\mid\theta)}{6}
    \E\{H_3(Z)f_\theta(Z)\}+R_3.
\]
We have $|\kappa_3(S_{-j}\mid\theta)| = O(J^{-1/2})$  by additivity of cumulants and uniform
conditional sub-Gaussianity. To check the conditions of \cref{thm:barbour-smooth-edgeworth}, note that the bounds on $h^{(4)}$ give
\[
  \mathcal L_{2,1,1}(f_\theta)\lesssim(\log n)^C,
  \quad
  |R_3|\lesssim(\log n)^C
  \sum_{\ell\ne j}\E(|Y_\ell|^4\mid\theta)
  \lesssim\frac{(\log n)^C}{J}.
\]
Gaussian integration by parts also gives
$|\E\{H_3(Z)f_\theta(Z)\}|=|\E f_\theta^{(3)}(Z)|\lesssim(\log n)^C$.
Thus
\begin{align*}
  &\left|
    \E\{h'(\theta+T_{-j})\mid\theta_i=\theta\}
    -\E h'(\theta+\sigma_i\sqrt{1-J^{-1}}Z)
  \right|
  \lesssim (\log n)^CJ^{-1/2},
\end{align*}
where $Z\sim\cN(0,1)$. We can apply the mean-value theorem to remove $\sqrt{1-J^{-1}}$:
Since $\norm{h''}_\infty \lesssim (\log n)^C$, $
  \left|
    \E h'(\theta+\sigma_i\sqrt{1-J^{-1}}Z)
    -\E h'(\theta+\sigma_iZ)
  \right|
   \lesssim (\log n)^CJ^{-1}.
$
Therefore \[
	\abs{\E[h'(\theta+T_{-j})\mid\theta_i=\theta] - \E[h'(\theta+ \sigma_i Z)\mid\theta_i]}
	\lesssim \frac{(\log n)^C}{\sqrt{J}}.
\]
Since $\kappa_i(\theta) := \E[\xi_j^3 \mid \theta_i=\theta]$ is uniformly bounded, \[
	\E\bk{(\xi_j^2-\sigma_i^2)h(\theta+T)\mid\theta_i} = \frac{\kappa_i(\theta_i)}{
	\sqrt{J}} \E[h'(\theta_i + \sigma_i Z) \mid \theta_i] + O\pr{\frac{(\log
	n)^C}
	{J}}.
\]
\paragraph{Bounding \eqref{eq:decompproduct}} Summing over $j$ and combining with \eqref{eq:estimated-variance-T2} gives \[
	\eqref{eq:decompproduct} = \frac{1}{\sqrt{J_i}} \kappa_i(\theta_i) \E[h'(\theta_i +
	\sigma_i Z) \mid \theta_i] + O\pr{\frac{(\log n)^C}{J_i}}.
	\numberthis \label{eq:estimated-variance-weighted-replacement}
\]

\paragraph{Bounding $\E[\Delta_i S_{\pi,i}]$}
To apply \eqref{eq:estimated-variance-weighted-replacement}, choose a smooth function
$\psi_{n,t}$ which equals one on $S_{n,t}$, vanishes outside $S_{n,2t}$, and has its first
four derivatives bounded by a constant.

Set $h=\psi_{n,t}S_{\pi,i}$. For this $h$, we can choose $C = 1/2$ in the remainder of \eqref{eq:estimated-variance-weighted-replacement}. Truncation by $\psi$ in $h$ incurs a small error:
\eqref{eq:estimated-variance-score-envelope},
\eqref{eq:estimated-variance-moments}, and the sub-Gaussianity of $X$ give, after
choosing $t$ sufficiently large,
\begin{equation}
  \sup_{\pi\in\cP([-M,M])}
  \E\left[
    |\Delta_iS_{\pi,i}(X_i)|
    \mathbf 1\{X_i\notin S_{n,t}\}
  \right]
  \leq Cn^{-4}.
  \label{eq:estimated-variance-score-tail}
\end{equation}
Now, for this choice of $h$, average the leading term in
\eqref{eq:estimated-variance-weighted-replacement} over
$\theta_i\sim\pistar$.  Define $
  \chi_i(x)
  :=
  \frac{
    \int\kappa_i(\theta)\phi(x-\theta;\sigma_i)\,d\pistar(\theta)
  }{f_{\pistar,\sigma_i}(x)}.$
Then, by integration by parts,
\begin{align*}
  \E_{\theta_i,Z}\left[
    \kappa_i(\theta_i)
    (\psi_{n,t}S_{\pi,i})'(\theta_i+\sigma_iZ)
  \right]&=
  \int \chi_i(x)f_{\pistar,\sigma_i}(x)
  (\psi_{n,t}S_{\pi,i})'(x)\,dx\\
  &= -\int  \underbrace{\psi_{n,t}(x)
 \left\{
        \chi_i'(x)
        +\chi_i(x)\frac{f'_{\pistar,\sigma_i}(x)}
                             {f_{\pistar,\sigma_i}(x)}
      \right\}}_w  f_{\pistar,\sigma_i}(x)S_{\pi,i}(x)\,dx.
\end{align*}
The weight function $w$ is supported on $S_{n,2t}$ and has its first two derivatives
bounded by $C\sqrt{\log n}$: fixed-order derivatives of $\chi_i$ and of
$h_{\pistar}(\cdot;\sigma_i)$ are uniformly bounded, while
$f'_{\pistar,\sigma_i}/f_{\pistar,\sigma_i}
=(h_{\pistar}(\cdot;\sigma_i)-x)/\sigma_i^2$ grows at most linearly.

Applying \cref{lem:estimated-variance-score},
\eqref{eq:estimated-variance-weighted-replacement}, and
\eqref{eq:estimated-variance-score-tail} therefore gives
\begin{equation}
  |\E\{\Delta_iS_{\pi,i}(X_i)\}|
  \leq
  C(\log n)^{3/2}
  \left\{
    J_i^{-1/2}\Hel(f_{\pi,\sigma_i},f_{\pistar,\sigma_i})
    +J_i^{-1}+n^{-4}
  \right\}.
  \label{eq:estimated-variance-score-mean-pointwise}
\end{equation} Averaging \eqref{eq:estimated-variance-score-mean-pointwise} concludes the
 proof, after an application of Cauchy--Schwarz.
\end{proof}

\begin{lemma}
\label{lem:estimated-variance-centered}
With probability at least $1-Cn^{-3}$, simultaneously for every
$\pi\in\cP([-M,M])$,
\begin{equation}
  \left|
    \avgn\left[
      \Delta_iS_{\pi,i}(X_i)
      -\E\{\Delta_iS_{\pi,i}(X_i)\}
    \right]
  \right|
  \leq
  C\log^4 n
  \left\{\frac1n+\avgn\frac1{J_i}\right\}.
  \label{eq:estimated-variance-centered}
\end{equation}
\end{lemma}

\begin{proof}
Fix a sufficiently large $t$. Define the following truncated and centered version of
$\Delta_i S_
{\pi,i}$: \[
  Y_{i,\pi}
  :=%
  \Delta_iS_{\pi,i}
  \mathbf 1\{X_i\in S_{n,t},\ |\Delta_i|\leq C\log n\} -
  \E\left[
    \Delta_iS_{\pi,i}
    \mathbf 1\{X_i\in S_{n,t},\ |\Delta_i|\leq C\log n\}
  \right].
\]
The $Y_{i,\pi}$ are independent and centered.  By
\eqref{eq:estimated-variance-score-envelope} and
\eqref{eq:estimated-variance-moments},
\begin{equation}
  \max_i\|Y_{i,\pi}\|_\infty\lesssim(\log n)^{3/2},
  \quad
  \E Y_{i,\pi}^2\lesssim J_i^{-1}.
  \label{eq:estimated-variance-centered-envelope}
\end{equation}
The variance bound follows from
$\E[\Delta_i^2S_{\pi,i}(X_i)^2]
\leq(\E\Delta_i^4)^{1/2}(\E S_{\pi,i}(X_i)^4)^{1/2}\lesssim J_i^{-1}$,
using the global bound in \eqref{eq:estimated-variance-score-envelope} and
sub-Gaussianity of $X_i$.

For fixed $A > 0$, define the following pseudometric on $\mathcal P([-M,M])$ by \[
	d_t(\pi_1, \pi_2) = \max_{1\leq i\leq n}\sup_{x\in S_{n,t}}
  \bigg\{
    |f_{\pi_1,\sigma_i}(x)-f_{\pi_2, \sigma_i}(x)|
  +
  \left|
    \left.\partial_vf_{\pi_1,\sqrt v}(x)\right|_{v=\sigma_i^2}
    -\left.\partial_vf_{\pi_2,\sqrt v}(x)\right|_{v=\sigma_i^2}
  \right|
  \bigg\}
\] An application of Proposition SM6.1 in \citet{chen2024empirical} (applied with $q=2$
and $|x| \le M+2t\sqrt{\log n}$, together
 with standard covering-packing inequalities and \cref{assump:lowerbound}) shows that \[\log N (n^
 {-A}, \mathcal P([-M,M]), d_t) \lesssim_{A, t} (\log n)^2.
\]
Thus let $\pi_1, \ldots, \pi_N$ be such an $n^{-A}$-cover in $d_t$.
Given $\pi$, if $d_t(\pi_j, \pi) \le n^{-A}$, we then have that \[
	|Y_{i,\pi} - Y_{i,\pi_j}| \lesssim \log n \sup_{x \in S_{n,t}} |S_{\pi,i}(x) - S_
	{\pi_j, i}(x)| \lesssim n^{-5}, \numberthis\label{eq:Y_projection_cost}
\]
by choosing sufficiently large $A$.

Bernstein's inequality, \eqref{eq:estimated-variance-centered-envelope},
and a union bound give that
\begin{align}
  \Prob\bigg\{
    \max_{1\leq j\leq N}
    \left|\avgn Y_{i,\pi_j}\right|
    >C \epsilon
  \bigg\}
  \leq
  2N\exp\{-(\log N+4\log n)\}
  \leq2n^{-4},
  \label{eq:estimated-variance-bernstein}
\end{align}
where
\begin{align*}
\epsilon&:=
\left\{\frac{\log N+4\log n}{n}\avgn\frac1{J_i}\right\}^{1/2}
+(\log n)^{3/2}\frac{\log N+4\log n}{n}\\
&\lesssim
\left\{\frac{\log^2 n}{n}\avgn\frac1{J_i}\right\}^{1/2}
+\frac{(\log n)^{7/2}}{n}.
\end{align*}
Young's inequality bounds this by
$C\log^4 n\{1/n+\avgn J_i^{-1}\}$.
Projecting to the cover and absorbing \eqref{eq:Y_projection_cost} gives the same
bound for $\sup_\pi|\avgn Y_{i,\pi}|$, with failure probability at most $2n^{-4}$.

It remains to remove the truncation in $Y_{i, \pi}$.
The proof of \cref{lem:estimated-variance-concentration} establishes
\begin{equation*}
  \Prob\left\{
    X_i\in S_{n,t},\ |\Delta_i|\leq C\log n
    \text{ for every }i
  \right\}
  \geq1-Cn^{-3}.
\end{equation*}
On this event, the uncentered part of $Y_{i,\pi}$ equals
$\Delta_iS_{\pi,i}(X_i)$. We also have that by calculations with
\eqref{eq:estimated-variance-score-envelope} and \eqref{eq:estimated-variance-tail} \[
	\sup_{\pi\in\cP([-M,M])}\avgn
  \E\big[|\Delta_iS_{\pi,i}(X_i)|
  \mathbf 1\{X_i\notin S_{n,t}
             \text{ or }|\Delta_i|>C\log n\}\big]
  \leq Cn^{-4}.
\]
This means that $
	\abs{\frac{1}{n}\sum_i Y_{i,\pi} - \frac{1}{n} \sum_i (\Delta_i S_{\pi,i} - \E[\Delta_i S_{\pi,i}])}
	\lesssim n^{-4}
$
with probability at least $1-Cn^{-3}$. Combining these facts with
\eqref{eq:estimated-variance-bernstein} proves
\eqref{eq:estimated-variance-centered}.
\end{proof}

\subsection{Proof of \cref{thm:main_rates}}

\begin{proof}[Proof of \cref{thm:main_rates}]

We will apply \cref{thm:main} for the Hellinger bound. Observe that under our setup, compact support of $\pistar$ and uniform conditional
    sub-Gaussianity imply \cref{assump:sub-Gauss}. It remains to calculate $\delta_n$. We use \cref{prop:moment-normal-convolution} to calculate $\bar r_1 + \bar r_2$.
    Observe that the first two conditional moments of
    $\varepsilon_{ij}/\sqrt{J_i}$ agree with those of $\cN(0,\sigma_i^2)$.
    Thus \cref{prop:moment-normal-convolution}, with $k=2$, gives
    \[
    \rbar{1}+\rbar{2}\lesssim
    {\avgn\frac{\log^2 J_i}{J_i}\lesssim\frac{\log^2 J}{J}} \lesssim \frac{\log^2 n}{J} + \frac{1}{n}
    \]
    The second-to-last inequality follows by applying Jensen's inequality to
    $x\mapsto x\log^2(e/x)$, which is concave on $(0,1]$.

    Next, \eqref{eq:approximate_mle} and
    \cref{prop:estimated-variance-oracle-likelihood} show that $\pihat$ is a
    $b_n$-approximate oracle NPMLE with
    $b_n\lesssim\log^4 n(1/n+1/J)$. Thus we can take \[
    \delta_n^2=C\log^4 n\left(\frac1n+\frac1J\right),
    \quad
    \E\!\left[\Helbarstar{f_{\pihat}}{f_{\pistar}}\right]\lesssim\delta_n^2,
    \]
    which proves the Hellinger bound.

We will apply \cref{thm:denoising-normal} for the regret bound. To check its conditions, note that the clipping ensures that $\hat\sigma_i$ is uniformly bounded
above and away from zero. Since the clipping interval contains $\sigma_i^2$,
projection onto that interval and \eqref{eq:estimated-variance-moments} give
\[
\avgn\E\!\left(\hat\sigma_i^2-\sigma_i^2\right)^2
\lesssim\avgn J_i^{-1}=J^{-1}\lesssim\delta_n^2.
\]
This verifies \eqref{eq:estimated-variance-regret-condition}.

Next we compute $\mathcal W_n^2$. Conditionally on $\theta_i$, apply \citet[Theorem~1.1]{bobkov2018berry} with $p=2$
to the independent summands $\varepsilon_{ij}/(J_i\sigma_i)$, whose variances sum
to one. Rescaling gives
\[
\Wtwo^2\!\left(\mathrm{Law}(X_i\mid\theta_i),\cN(\theta_i,\sigma_i^2)\right)
\lesssim
\frac{\E[(\varepsilon_{i1}/\sqrt{J_i})^4\mid\theta_i]}{\sigma_i^2J_i}
\lesssim J_i^{-1},
\]
where the last inequality uses uniform conditional sub-Gaussianity and
\cref{assump:lowerbound}. Averaging over $\theta_i$ and $i$ yields
$\mathcal W_n^2\lesssim J^{-1}$.

Finally, \cref{thm:denoising-normal} and
$\max\{1,\log(\delta_n^{-1})\}\lesssim\log n$ give
\[
(\mathrm{Regret}_n)_+^2
\lesssim\delta_n^2\max\{1,\log(\delta_n^{-1})\}+\mathcal W_n^2
\lesssim\log^5 n\left(\frac1n+\frac1J\right).
\qedhere
\]
\end{proof}

\begin{proof}[Proof of \cref{cor:inference}]
Define
$
  L:=1+\log\frac1{1/n+1/J}.
$
By \cref{thm:main_rates} and Markov's inequality,
\[
  \Prob\left\{
  \Helbarstar{f_{\pihat}}{f_{\pistar}}
  >(\log n)^2\left(\frac1n+\frac1J\right)^{1/2}
  \right\}
  \lesssim
  (\log n)^2\left(\frac1n+\frac1J\right)^{1/2}.
\]
The proof of \cref{thm:deconvolution}, together with bounded support, therefore gives
\begin{align*}
\E W_2^2(\pihat,\pistar)
&\lesssim
\left[1+\log\left\{1+
\frac{(1/n+1/J)^{-1/4}}{\log n}\right\}\right]^{-1}
+(\log n)^2\left(\frac1n+\frac1J\right)^{1/2}
\lesssim L^{-1}.
\end{align*}
The last inequality uses the growth condition in the statement, with
$C_0$ sufficiently large.  Jensen's inequality then proves the Wasserstein result.

We now turn to the coverage result. Define $\Pi_{\pi,s}$ as the normal posterior CDF under
prior $\pi$ and data variance $s^2$ and $Q_{\pi, s}$ as its corresponding quantile function:
\begin{align*}
\Pi_{\pi,s}(t\mid x)
  &:=\frac{\int_{-M}^t\phi(x-\theta;s)\,d\pi(\theta)}
          {f_{\pi,s}(x)},\\
Q_{\pi,s}(a\mid x)
  &:=\inf\{t:\Pi_{\pi,s}(t\mid x)\geq a\}.
\end{align*}
For the normal oracle, its posterior credible interval is \[
  C_i^0(x;\alpha)
  :=\left[Q_{\pistar,\sigma_i}(\alpha/2\mid x),
  Q_{\pistar,\sigma_i}(1-\alpha/2\mid x)\right],
\]
while the estimated interval is
\[
  \hat\Theta_i(X_1,\ldots,X_n;\alpha)
  =\left[Q_{\pihat,\hat\sigma_i}(\alpha/2\mid X_i),
  Q_{\pihat,\hat\sigma_i}(1-\alpha/2\mid X_i)\right].
\]
For $(\theta_i, X_i)$ drawn from the true DGP, define the random variable $U_i := \Pi_
{\pistar, \sigma_i}
(\theta_i \mid X_i)$. Because $\pistar$ has a strictly positive density, we have
\[
  \br{\theta_i \in C_i^0(X_i; \alpha)} = \br{U_i \in [\alpha/2, 1-\alpha/2]}.
\]
Accordingly, we may decompose \begin{align*}
&\left|
  \Prob\{\theta_I\in\hat\Theta_I(X_1,\ldots,X_n;\alpha)\}-(1-\alpha)
  \right|
\\&\leq
  \left|\Prob\{\theta_I\in\hat\Theta_I(X_1,\ldots,X_n;\alpha)\}
  -\Prob\{U_I\in[\alpha/2, 1-\alpha/2]\}\right|
\\&\quad +
  \left|\Prob\{U_I\in[\alpha/2, 1-\alpha/2]\}-(1-\alpha)\right|.
\numberthis \label{eq:coverage-main-decomposition}
\end{align*}

Let us begin with the second term in \eqref{eq:coverage-main-decomposition}. Let $Z_i
\mid \theta_i \sim \Norm(\theta_i, \sigma_i^2)$.  Berry--Esseen, applied to $X_i\mid
\theta_i$ and $Z_i \mid \theta_i$, and the uniform sub-Gaussian assumption give,
uniformly in $i$ and $\theta_i$,
\begin{equation*}
  \sup_x\left|
  \Prob(X_i\leq x\mid\theta_i)
  -\Prob(Z_i\leq x\mid\theta_i)
  \right|\lesssim J_i^{-1/2}.
\end{equation*}
For fixed $\theta$, $x\mapsto\Pi_{\pistar,\sigma_i}(\theta\mid x)$ is nonincreasing. Thus,
for all $\theta \in (-M,M)$ and $u\in (0,1)$ fixed,
\[
  \br{x : \Pi_{\pistar, \sigma_i}(\theta \mid x) \le u} = [c(i,\theta,u), \infty)
\]
for some $c(i,\theta,u)$. Therefore \begin{align*}
\Prob(U_i \le u) &= \E\Prob(U_i \le u \mid \theta_i) = \E \Prob\pr{
    \Pi_{\pistar, \sigma_i}(\theta_i \mid X_i) \le u \mid \theta_i
  } \\
  &= \E\Prob( X_i \ge c(i, \theta_i, u) \mid \theta_i ) \\
  &= \E\Prob( Z_i \ge c(i, \theta_i, u) \mid \theta_i ) + O(J_i^{-1/2}) \\
  &= \underbrace{\E\Prob( \Pi_{\pistar, \sigma_i}(\theta \mid Z_i) \le u \mid \theta_i
  )}_u + O (J_i^
  {-1/2}) \numberthis \label{eq:couple_U_to_Z}
\end{align*}
where the $O(J_i^{-1/2})$ term is bounded by  $ C J_i^{-1/2}$ for some constant $C$.
This holds similarly for $\Prob(U_i \ge u).$ Hence \[
  \abs{\Prob\{U_I\in[\alpha/2, 1-\alpha/2]\} - \underbrace{\Prob( \theta_I \in C^0_I(Z_I;
  \alpha) )}_{1-\alpha}}
  \lesssim
  \frac{1}{n} \sum_{i=1}^n J_i^{-1/2} \le J^{-1/2}.
\]
Returning to the first term in \eqref{eq:coverage-main-decomposition}, let $d_{\mathrm K}
(\pihat,\pistar)$ be the Kolmogorov distance between the priors and let  $D_i(x)$ be the
Kolmogorov distance between the corresponding posteriors:
\begin{align*}
d_{\mathrm K}(\pihat,\pistar)
  &:=\sup_t
  |\pihat(({-}\infty,t])-\pistar(({-}\infty,t])|,\\
D_i(x)
  &:=\sup_t\left|
  \Pi_{\pihat,\hat\sigma_i}(t\mid x)
  -\Pi_{\pistar,\sigma_i}(t\mid x)
  \right|.
\end{align*}
Fix $0 < r < \alpha/2$. We claim that
\begin{align*}
\left|
  \mathbf 1\{\theta_i\in\hat\Theta_i\}
  -\mathbf 1\{U_i\in[\alpha/2,1-\alpha/2]\}
  \right|\leq
  \mathbf 1\{D_i(X_i)>r\}
  +\sum_{c\in\{\alpha/2,1-\alpha/2\}}\mathbf 1\{|U_i-c|\leq r\}.
\numberthis
\label{eq:coverage-rank-bracketing}
\end{align*}
To see this, if $D_i(x) \le r$, then we have for $c\in [\alpha/2,1-\alpha/2]$, \[
  Q_{\pistar, \sigma_i}(c-r \mid x) \le Q_{\pihat, \hat\sigma_i} (c \mid x) \le Q_
  {\pistar, \sigma_i} (c+r \mid x).
\]
Hence, on $D_i(X_i) \le r$, \[
  \left|
  \mathbf 1\{\theta_i\in\hat\Theta_i\}
  -\mathbf 1\{U_i\in[\alpha/2,1-\alpha/2]\}
  \right| \le \one(|U_i - \alpha/2| \le r) + \one(|U_i - (1-\alpha/2)| \le r),
\]
giving \eqref{eq:coverage-rank-bracketing}. By \eqref{eq:couple_U_to_Z}, we further have
that $\Prob(|U_i-c|\leq r) \lesssim r + J_i^{-1/2}$.

Hence it remains to bound $\Prob(D_i(X_i) > r)$. We claim that
\begin{equation}
  D_i(x)\leq Ce^{C|x|}d_{\mathrm K}(\pihat,\pistar)
  +C(1+|x|)|\hat\sigma_i-\sigma_i|.
\label{eq:coverage-posterior-cdf-global}
\end{equation}
Assuming the claim \eqref{eq:coverage-posterior-cdf-global} for now, note that a
standard result implies that for distributions with a uniformly bounded density, $d_
{\mathrm K}^2 (\pihat,\pistar)\lesssim W_1 (\pihat,\pistar) \le W_2(\pihat,\pistar)$. At
the
same time,
projection onto the clipping
interval, \eqref{eq:estimated-variance-moments}, and the lower bound on $\sigma_i$ also
give
\[
  \avgn\E|\hat\sigma_i-\sigma_i|^2\lesssim\frac1J.
\]
Plugging these bounds into \eqref{eq:coverage-posterior-cdf-global} yields, after applying
the first part of the statement on controlling $W_2$,
\[
  \E D_I(X_I) \lesssim
  \bigl\{\E W_2(\pihat,\pistar)\bigr\}^{1/2}+J^{-1/2} \lesssim L^{-1/4} + J^{-1/2}.
\]
By Markov's inequality applied to the expectation of
\eqref{eq:coverage-rank-bracketing}, we obtain \[
  \left|\Prob\{\theta_I\in\hat\Theta_I(X_1,\ldots,X_n;\alpha)\}
  -\Prob\{U_I\in[\alpha/2, 1-\alpha/2]\}\right| \lesssim \frac{L^{-1/4} + J^{-1/2}}{r} + r + J^
  {-1/2} \lesssim L^{-1/8}
\]
by taking $r = \sqrt{L^{-1/4} + J^{-1/2}} < \alpha/2$. Putting things together bounds
\eqref{eq:coverage-main-decomposition} by $L^{-1/8}$.

It remains to conclude the proof by verifying \eqref{eq:coverage-posterior-cdf-global}. We

first show that for all $s^2 \in [k/2, 2K]$, \[
  |\Pi_{\pihat, s} (t \mid x) - \Pi_{\pistar, s}(t\mid x)| \le C e^{C|x|} d_{\mathrm{K}}
  (\pihat,\pistar). \numberthis \label{eq:intermediatePidist}
\]
Observe that
\[
\Pi_{\pi,s}(t\mid x)
=
\frac{\int_{[-M,t]} w_{x,s}(\theta)\,d\pi(\theta)}
{\int_{[-M,M]} w_{x,s}(\theta)\,d\pi(\theta)}, \quad w_{x,s}(\theta) = \exp\br{x\theta/s^2
-
\theta^2/(2s^2)}
\]
The differences across $\pihat$ and $\pistar$ for both the denominator and the numerator
in $\Pi_{\cdot}(t \mid x)$ are bounded as follows: For any $h$ of bounded variation, we
have
\begin{align*}
\int h(\theta) d(\hat\pi - \pi) \le d_{\mathrm{K}}(\pihat, \pistar) \mathrm{TV}(h)
\end{align*}
by integration by parts. For both $h = w_{x,s}$ and $h(\theta) = w_{x,s}(\theta)\one
(\theta \le t)$, we have that their total variation is bounded by $C\sup_{\theta \in
[-M,M]} w_{x,s}$. Hence the differences of the numerator and denominator in $\Pi$ are
bounded by $C d_{\mathrm{K}}(\pihat, \pistar) \sup_{\theta \in
[-M,M]} w_{x,s}$. Finally, the
denominator itself is lower bounded by $\inf_{\theta \in[-M, M]} w_{x,s}(\theta)$. Taking
the ratio and observing that $\sup_{\theta} w_{x,s}(\theta) / \inf_\theta w_{x,s}(\theta)
\lesssim e^{C|x|}$ yields \eqref{eq:intermediatePidist}.

Note too that \[
  \diff{}{s} \Pi_{\pi, s}(t \mid x) = \cov_{\theta \sim \Pi_{\pi, s}(\cdot \mid x)}
  \pr{\one (\theta \le t), \frac{
  (x-\theta)^2} {s^3}} \lesssim 1+|x|
\]
for a constant uniform over the relevant values of $s$. The mean value theorem, combined with
\eqref{eq:intermediatePidist}, then yields
\eqref{eq:coverage-posterior-cdf-global}.
\end{proof}

\subsection{Adaptation to matching higher moments}
\label{asub:symmetric}

When the variance $\sigma_i^2$ is known, we observe that the NPMLE automatically adapts to
higher moments of $\varepsilon_{ij}$.

For example, loosely speaking, if the variance is known and
$\varepsilon_{ij}$ is symmetrically distributed given $\theta_i$, then the dependence on $J$
in the Hellinger bound becomes ${1/J^2}$
instead of $1/J$, since higher-order Edgeworth expansions are available. To be precise, here we work with a different aggregate sample size notion.  For $p<0$, let $J_p=(\avgn J_i^p)^{1/p}$ be the power mean of order $p$ of the group sizes, so the harmonic mean $J$ in the main text is $J_{-1}$.

\begin{thm}\label{thm:moment_adapt}
  Consider the same setup as in \cref{thm:main_rates} with known $\hat\sigma_i^2 = \sigma_i^2$. We further assume that
  for $1 \le \ell \le k$, with a fixed integer $k>2$, it holds that
  \begin{equation*}
    \E\!\left[\left(\frac{\varepsilon_{ij}}{\sqrt{J_i}}\right)^{\ell} \mid \theta_i\right]
    = \sigma_i^\ell\, \E(Z^\ell), \qquad Z \sim \cN(0,1).
  \end{equation*}
  Then we have the following conclusions:
  \begin{equation*}
    \E\bk{\frac{1}{n}\sum_{i=1}^{n}
      \Hel^2\!\left(f_{\hat{\pi}_n,\sigma_i},\, f_{\pst,\sigma_i}\right)
    }
    \lesssim
    {\frac{\log^4 n}{n}+\frac{\log^2 J_{1-k}}{J_{1-k}^{k-1}}}.
  \end{equation*}
  Moreover, assume in addition the following
  uniform conditional Cram\'er condition: there exist constants $t_0>0$ and
  $\rho\in(0,1)$, independent of $n$, $i$, and $\theta$, such that
  \begin{equation*}
    \sup_{|t|\ge t_0}
    \left|
      \E\!\left[
        \exp\!\left(i t\frac{\varepsilon_{i1}}{\sigma_i\sqrt{J_i}}\right)
        \mid\theta_i=\theta
      \right]
    \right|\le\rho
  \end{equation*}
  for every $n$, $1\le i\le n$, and $\pistar$-almost every $\theta$.
  Then
  \begin{equation*}
    \bigl(\mathrm{Regret}_n^+\bigr)^2
    \lesssim
    {\log n\left(\frac{\log^4 n}{n}+\frac{\log^2 J_{1-k}}{J_{1-k}^{k-1}}\right)}.
  \end{equation*}
  The constant in this regret bound may also depend on
  $t_0$ and $\rho$.
\end{thm}

Therefore the NPMLE requires fewer observations per group if more conditional moments of the
observed data match those of the approximating Gaussian convolution.

\begin{proof}[Proof of \cref{thm:moment_adapt}]
We shall apply \cref{thm:main,thm:denoising-normal}.
By \cref{prop:moment-normal-convolution},
\begin{equation*}
  \rbar{1}+\rbar{2}\lesssim
  {\avgn\frac{\log^2 J_i}{J_i^{k-1}}\le\frac{\log^2 J_{1-k}}{J_{1-k}^{k-1}}}.
\end{equation*}
The last inequality follows from Jensen's inequality applied to
$x\mapsto x\log^2(1/x)$ at $x=J_i^{1-k}$, since this function is concave on
$(0,1/4]$ and $J_i\ge2$, $k\ge3$.

Apply the argument of \citet[Theorem~1.3]{bobkov2018berry}
conditionally on $\theta_i$ under the additional Cram\'er condition, with $p=2$ and $\ell=k+1$: With a constant independent of $n$, $i$, and $\theta_i$,
\[
  \mathcal W_n^2\lesssim\frac1n\sum_{i=1}^n J_i^{-(k-1)}=J_{1-k}^{1-k}.
\]
Moreover, \eqref{eq:approximate_mle} with $\hat\sigma_i=\sigma_i$
implies \eqref{eq:likelihood-bound} with $b_n=\log n/n$.
This term is absorbed by $\log^4 n/n$ in \cref{thm:main}.
{The regret bound then follows from \cref{thm:denoising-normal} since
$\delta_n^2\gtrsim\log^4 n/n$ implies $\max\{1,\log(\delta_n^{-1})\}\lesssim\log n$.}
\end{proof}

\subsection{A regret lower bound for the  NPMLE}
\label{asub:npmle-lower-bound}

\begin{prop}
Fix $M>0$ and $0<k<\sigma<K$. Let $\pihat$ be a (measurable) maximizer of \eqref{eq:npmle} with the clipped sample variances $\hat\sigma_i^2$.
For all $n\ge 2$ and all sufficiently large $J$,
\[
  \sup_P (\mathrm{Regret}_n^+)^2 \ge c\left(\frac1n+\frac1J\right),
\]
where the supremum is over distributions of $(X_i, \hat \sigma_i^2)_{i=1}^n$ satisfying the assumptions of \cref{thm:main_rates}
with $\pistar=\delta_0$, $J_i=J$, $\sigma_i=\sigma$. The constant $c>0$ does not depend on $n,J$.
\end{prop}

\begin{proof}
Denote the estimate of $\theta_i$ as $\hat\theta_i=h_{\pihat}(X_i;\hat\sigma_i)$. Since $\theta_i=0$ almost surely, the risk of the normal oracle is zero.
Hence  $(\mathrm{Regret}_n^+)^2=\avgn\E\hat\theta_i^2$. We seek a lower bound of this quantity.

For $b\in\{-M,M\}$, consider the following perturbation of the log-likelihood as a function of $t$:
\[
  \ell_b(t)=\avgn\log\int
  \phi\left(X_i-[(1-t)\theta+tb];\hat\sigma_i\right)\,d\pihat(\theta).
\]
By $\hat\pi_n$ being the maximizer, we have $\ell_b(t)\le\ell_b(0)$ for $0\le t\le1$.
Differentiation gives
\[
  0\ge \ell_b'(0+)=\avgn
  \frac{bX_i-(X_i+b)\hat\theta_i+
  \hat\E_{\pihat}[\theta^2\mid X_i,\hat\sigma_i]}{\hat\sigma_i^2} \ge \avgn
  \frac{bX_i-(X_i+b)\hat\theta_i}{\hat\sigma_i^2}.
\]
Rearrange to obtain \[
b\left(\frac1n\sum_i\frac{X_i}{\hat\sigma_i^2}\right)
\le \frac1n\sum_i\frac{(X_i+b)\hat\theta_i}{\hat\sigma_i^2}
\le \frac1n\sum_i\frac{(|X_i|+M)|\hat\theta_i|}{\hat\sigma_i^2}.
\]
Since this is true for both $b \in \br{\pm M}$, we have $
  M\left|\avgn\frac{X_i}{\hat\sigma_i^2}\right|
  \le \avgn\frac{(|X_i|+M)|\hat\theta_i|}{\hat\sigma_i^2}.
$ This gives, after taking expectations,
\[
  M \E \left|\avgn\frac{X_i}{\hat\sigma_i^2}\right|
  \le \E\bk{\avgn\frac{(|X_i|+M)|\hat\theta_i|}{\hat\sigma_i^2}}
  \le \frac{2(\sigma+M)}{k^2}
  \left(\frac1n\sum_{i=1}^n\mathbb E\hat\theta_i^2\right)^{1/2},
\]
where the last inequality uses Cauchy--Schwarz.
Rearrange to obtain
\begin{equation}
  (\mathrm{Regret}_n^+)^2
  \ge c\left(\E\left|\avgn\frac{X_i}{\hat\sigma_i^2}\right|\right)^2.
  \label{eq:npmle-regret-weighted-mean}
\end{equation}
It remains to show that the RHS of \eqref{eq:npmle-regret-weighted-mean} is bounded below by $1/J+1/n$ for suitable choices of the law of $X_i$. For the $1/J$ dependence, we can take $\varepsilon_{ij}$ to have a two-point distribution with mean zero, variance $J\sigma^2$, and third moment $\kappa_3 \sigma^3 J^{3/2} > 0$. By standard arguments we can show \[
\E \frac{X_i}{\hat \sigma_i^2} = - \frac{\kappa_3}{\sigma \sqrt{J}} + O(J^{-1}).
\]
This proves the $\gtrsim 1/J$ lower bound by \eqref{eq:npmle-regret-weighted-mean}.
For the $1/n$ term, we can take $\varepsilon_{ij}$ to be normal. Then $X_i$
are independent $\Norm(0,\sigma^2)$ and are independent of the sample variances, so
\[
  \E\left[\left|\avgn\frac{X_i}{\hat\sigma_i^2}\right|
       \mid\hat\sigma_1,\ldots,\hat\sigma_n\right]
  =\sqrt{\frac{2}{\pi}}\,\frac{\sigma}{n}
     \left(\sum_{i=1}^n\hat\sigma_i^{-4}\right)^{1/2}
  \ge \sqrt{\frac{2}{\pi}}\,\frac{\sigma}{2K^2\sqrt n}.
\]
Thus \eqref{eq:npmle-regret-weighted-mean} gives $c/n$. Taking the larger
of the two lower bounds proves the claim.
\end{proof}

\section{Proofs of results in \cref{sec:linear_measurements}}

\begin{thm}
\label{thm:quantitative-dependence}
Suppose $A$ is deterministic and has full row rank, $m>n+2$,
$\theta_i\iid\pistar\in\cP([-M,M])$, $\theta$ is independent of $\eta$, and
\eqref{eq:dependence-weighted-representation}--\eqref{eq:dependence-scale-normalization}
and \cref{assump:dependence-lsi} hold.  For every fixed $D>0$, there is a constant $C>0$
such that
\begin{equation}
  \Prob\left\{
    \Helbarstar{f_{\pihat}}{f_{\pistar}}
    >C\log^4 m
       \left(\frac1{n_{*,A}}+\frac1{J_A}\right)
  \right\}
  \leq Cm^{-D}.
  \label{eq:quantitative-dependence-rate}
\end{equation}
Moreover,
\[
  \E\Helbarstar{f_{\pihat}}{f_{\pistar}}
  \leq C\log^4 m
       \left(\frac1{n_{*,A}}+\frac1{J_A}\right).
\]
If the $\eta_t$ are standard Gaussian, $J_A^{-1}$ may be omitted from both
bounds.
\end{thm}

The proof follows the same localization argument as \cref{thm:main}.

The following high-level condition, in the vein of \cref{asn:assumeindep}, characterizes
approximate normality in this setting. Here, we say that the marginal law
of $X_i$ is an \emph{$(r_{1,i}, r_{2,i}, r_{3,i})$-approximate normal convolution} if both
conditions in \cref{asn:assumeindep} hold and, moreover, for every
$\pi\in\cP([-M,M])$, \begin{equation}
{\E\left[
    \left\{
    \left.\sigma_i\frac{\partial}{\partial x}
    \log\frac{f_{\pi,\sigma_i}(x)}{f_{\pistar,\sigma_i}(x)}
    \right|_{x=X_i}
    \right\}^2
  \right]}
  \leq c_3\left[
    g_T\!\left\{\Helsq(f_{\pi,\sigma_i},f_{\pistar,\sigma_i})\right\}
    +r_{3,i}
  \right]
  \label{eq:dependence-score-condition}
\end{equation}

This condition enables a local maximal inequality, which then implies the conclusion of
\cref{thm:quantitative-dependence} through a peeling argument, analogous to
\cref{prop:local-maximal,thm:main}. We then verify that under the preceding assumptions,
our setup satisfies this notion of
{$(r_{1,i}, r_{2,i}, r_{3,i})$-approximate normal convolution}. We do so via the Edgeworth
expansion for smooth functions in
\cref{thm:barbour-smooth-edgeworth}. The Edgeworth expansion here is analogous to
\cref{prop:moment-normal-convolution}, but more complex since
\eqref{eq:dependence-weighted-representation} involves weights $c_{it}$.

For use in the proofs below, write
$t_{\pi,i}(x)=\log\{f_{\pi,\sigma_i}(x)/f_{\pistar,\sigma_i}(x)\}$.

\subsection{A local maximal inequality under dependence}

The following consequence of \eqref{eq:dependence-one-dimensional-lsi} will be used
conditionally on $\theta$. It says that, under \eqref{eq:dependence-one-dimensional-lsi},
if $G$ has bounded curvature, then $G(\eta) = G (\eta_1,\ldots, \eta_m)$ satisfies a Bernstein-type
concentration inequality.
\begin{lemma}
\label{lem:dependence-second-order-lsi}
Let $G:\R^m\to\R$ be twice continuously differentiable and suppose
\[
  \sup_x\|\nabla^2G(x)\|_{\mathrm{op}}\leq H.
\]

For every $x > 0$,
\[\mathbb P\!\left( |G(\eta)-\mathbb E_\eta G|>x
\right)
\le
2\exp\!\left[
-c\min\left\{
\frac{x^2}{C_{\mathrm{LS}}\mathbb E_\eta\|\nabla G\|_2^2},
\frac{x}{C_{\mathrm{LS}}H}
\right\}
\right].\]
Here $\Prob$ and $\E_\eta$ are with respect to the product law of $\eta$.
\end{lemma}

\begin{proof}
The proof follows from \citet[Theorem 3.19]{adamczak2017moment} with $f = G$ and $k=2$. Here $\|U\|_p=(\E_\eta |U|^p)^{1/p}$.

We claim that for every locally Lipschitz $g : \R^m \to \R$ and $p \ge 2$,  \[
\norm{g - \E g}_p \le L \sqrt{p} \norm{\norm{\nabla g}_2}_p \numberthis \label{eq:as_moment_lsi}
\]
for $L \le C \sqrt{C_{\mathrm{LS}}} $.
Under this claim, condition (3.18) in \citet{adamczak2017moment} holds so that we may apply Theorem 3.19. We thus have for every $p \ge 2$, \[
\norm{G - \E G}_p \le L \sqrt{p} \E \norm{\nabla G}_2 + L^2 p \norm{\norm{\nabla^2 G}_{\mathrm{op}}}_{p} \lesssim \sqrt{C_{\mathrm{LS}} p} \br{\E \norm{\nabla G}_2^2}^{1/2} + C_{\mathrm{LS}} p H.
\]
The first inequality follows from \citet{adamczak2017moment} and the second inequality follows from our assumptions.
Applying Markov's inequality to this bound and optimizing over $p\ge 2$ gives the stated tail inequality.

\Cref{eq:as_moment_lsi} is a standard consequence of \citet{aida1994moment}. By tensorization, we have for every smooth $h(\eta_1,\ldots, \eta_m)$, \[
\mathrm{Ent}(h^2) \le 2 C_{\mathrm{LS}} \E \norm{\nabla h}_2^2.
\]
Thus the joint law of $(\eta_1,\ldots, \eta_m)$ satisfies a log-Sobolev inequality as well. We can then invoke \citet{aida1994moment} (see, for instance, p.~468 of \citet{adamczak2017moment}).
\end{proof}

Define
\begin{equation*}
  L_n(\pi)=\avgn t_{\pi,i}(X_i),
  \quad
  \rbar{3}=\avgn r_{3,i}.
\end{equation*}

\begin{prop}[Local maximal inequality]
\label{prop:dependence-local-maximal}
Suppose $\theta_i\iid\pistar$, $\theta$ is independent of $\eta$,
\eqref{eq:dependence-weighted-representation}, \cref{assump:dependence-lsi}, the two
conditions in \cref{asn:assumeindep}, \eqref{eq:dependence-score-condition}, and the
scale bounds implied by \eqref{eq:dependence-scale-normalization} hold.  For every fixed
$c,D>0$, there is a sufficiently large $c^*>0$ such that, whenever
\begin{equation}
  c^*\left\{\frac{\log^4m}{n_{*,A}}+\rbar{2}+\rbar{3}\right\}
  \leq r\leq1,
  \label{eq:dependence-local-radius}
\end{equation}
and
\[
  \mathcal G(M,r)\subseteq
  \left\{\pi\in\cP([-M,M]):
    \Helbarstar{f_\pi}{f_{\pistar}}\leq r
  \right\},
\]
we have
\begin{equation*}
  \Prob\left\{\sup_{\pi\in\mathcal G(M,r)}
  \{L_n(\pi)-\E L_n(\pi)\}\geq cr\right\}
  \leq Cm^{-D}+2\exp\left\{-C\frac{n_{*,A}r}{\log^2m}\right\}.
\end{equation*}
\end{prop}

\begin{proof}

For fixed $\pi$, write
\begin{align*}
  L_n(\pi)-\E L_n(\pi)
  &=L_n(\pi)-\E_\eta\{L_n(\pi)\mid\theta\}
  +\avgn\left[
    \E_\eta\{t_{\pi,i}(X_i)\mid\theta_i\}
    -\E t_{\pi,i}(X_i)
  \right] \nonumber \\
  &=: L_{n1}(\pi \mid \theta)  + L_{n2}(\pi) .
\end{align*}
For fixed $\theta = (\theta_1,\ldots, \theta_n)$, we use
\cref{lem:dependence-second-order-lsi} to study the distribution of $L_ {n1} (\pi
\mid \theta) \mid \theta$.
In particular, the analysis of $L_{n1}$ uses the additional approximate normality
requirement
\eqref{eq:dependence-score-condition} in \eqref{eq:dependence-average-b}. We then use a
standard Bernstein inequality to analyze $L_ {n2} (\pi)$. These combine to give a
pointwise-in-$\pi$ concentration inequality
\eqref{eq:dependence-pointwise-tail}.  We then upgrade the pointwise result through a
covering argument.

\paragraph{Bounding $L_{n1}(\pi)$}
We apply \cref{lem:dependence-second-order-lsi} conditionally on $\theta$ to $L_{n1}$. We
let \[G(\eta) = \frac{1}{n}\sum_i t_{\pi,i}\pr{\theta_i + \sigma_i \sum_{t} c_
{it} \eta_t}.\] We can calculate that $\nabla G = \nabla_\eta L_n(\pi) = \frac1n C_A'
    \big(\sigma_i t_{\pi,i}'(X_i)\big)_{i=1}^n$ and \[
	\nabla^2 G = \nabla_\eta^2L_n(\pi)
  =\frac1n C_A'
    \operatorname{diag}\big(\sigma_i^2t_{\pi,i}''(X_i)\big)_{i=1}^n C_A,
\]
for $C_A := (c_{it})_{i\leq n,t\leq m}$ with
\begin{equation*}
  C_AC_A'=\operatorname{Cor}(AA'),
  \qquad
  \|C_A\|_{\mathrm{op}}^2=\frac{n}{n_{*,A}}.
\end{equation*}
By \cref{lem:deriv-cumulant-bound}, compact support of the priors and
\cref{assump:lowerbound} imply $\sup_{i,\pi,x}|\sigma_i^2t_{\pi,i}''(x)|\leq C.$
Therefore, we have a bound on $\E[\norm{\nabla G}_2^2 \mid \theta]$ and $\norm{\nabla^2
G}_ {\mathrm{op}}$:
\begin{align*}
  \E_\eta\{\|\nabla_\eta L_n(\pi)\|_2^2\mid\theta\}
  &\leq\frac1{n_{*,A}}\avgn
     \underbrace{\E_\eta\{(\sigma_it_{\pi,i}'(X_i))^2\mid\theta_i\}}_{=:b_i(\theta_i)},
  \\
  \sup_\eta\|\nabla_\eta^2L_n(\pi)\|_{\mathrm{op}}
  &\leq\frac{C}{n_{*,A}}.
\end{align*}
For every $p\geq1$, \cref{lem:dependence-second-order-lsi} therefore gives, conditionally
on $\theta$, with probability at least $1-2e^{-p}$,
\[
  |L_{n1}(\pi\mid\theta)|
  \leq C\left[
    \sqrt{\frac{p}{n_{*,A}}\avgn b_i(\theta_i)}
    +\frac{p}{n_{*,A}}\right].
\]
To convert this bound to an unconditional one, it remains to bound $\avgn b_i(\theta_i)$.
We do so with Bernstein's inequality. Since the support of posteriors is compact, $b_i$ is
bounded: $0 \le b_i \le C$. By assumption, since the marginal law of $X_i$ is an
approximate normal convolution,
\eqref{eq:dependence-score-condition} gives: \begin{align}
  \avgn\E b_i(\theta_i)
  &\lesssim \left\{
    \avgn g_T\!\left(
      \Helsq(f_{\pi,\sigma_i},f_{\pistar,\sigma_i})
    \right)+\rbar{3}
  \right\}
   \lesssim \left\{
    g_T(\Helbarstar{f_\pi}{f_{\pistar}})+\rbar{3}
  \right\}
  \label{eq:dependence-average-b}
\end{align}
(choosing $T$ sufficiently large so that $g_T$ is concave on $[0,2]$ and increasing on
$[0,1]$). The $b_i(\theta_i)$ are independent, nonnegative, and bounded. Thus,
$\Var\{b_i(\theta_i)\}\leq C\E b_i(\theta_i)$. Bernstein's inequality and
\eqref{eq:dependence-average-b} give, with probability at least $1-e^{-p}$,
\[
  \avgn b_i(\theta_i)
  \leq C\left\{g_T(\Helbarstar{f_\pi}{f_{\pistar}})+\rbar{3}+\frac{p}{n}\right\}.
\]
Substituting this bound in the conditional tail and averaging over $\theta$ gives,
with probability at least $1-3e^{-p}$,
\[
  |L_{n1}(\pi\mid\theta)|
  \leq C\left[
    \sqrt{\frac{p}{n_{*,A}}\{g_T(\Helbarstar{f_\pi}{f_{\pistar}})+\rbar{3}\}}
    +\frac{p}{n_{*,A}}\right],
\]
since $p/\sqrt{n_{*,A}n}\leq p/n_{*,A}$.

\paragraph{Bounding $L_{n2}(\pi)$} Returning to $L_n - \E L_n$, recall that \[
	L_{n2}(\pi) = \avgn\left[
    \E_\eta\{t_{\pi,i}(X_i)\mid\theta_i\}
    -\E t_{\pi,i}(X_i)
  \right] =: \avgn d_i(\theta_i) - \E d_i(\theta_i).
\]
We apply Bernstein's inequality to $d_i$. The functions $d_i$ are uniformly bounded
because $|d_i| \le \E [|t_ {\pi,i} (X_i)| \mid
\theta_i]
\lesssim \E[1+|X_i| \mid \theta_i] \le 1+|\theta_i| + \sigma_i \E_\eta \abs{\sum_t c_
{it} \eta_t} \lesssim 1$ by \eqref{eq:ratio-growth} and \cref{assump:lowerbound}. The
approximate normal convolution condition, in particular \eqref{eq:condition-b}, bounds the
variance of $d_i$:
\begin{equation*}
  \avgn\Var\{d_i(\theta_i)\}
  \leq C\left\{
    g_T(\Helbarstar{f_\pi}{f_{\pistar}})+\rbar{2}
  \right\}
\end{equation*}
(again choosing $T$ large enough so that $g_T$ is concave on $[0,2]$). Bernstein's inequality now gives, with probability at least $1-2e^{-p}$,
\[
  |L_{n2}(\pi)|\leq C\left[
    \sqrt{\frac{p}{n}\{g_T(\Helbarstar{f_\pi}{f_{\pistar}})+\rbar{2}\}}
    +\frac{p}{n}\right].
\]
Apply the bounds for $L_{n1}$ and $L_{n2}$ at $p+\log 5$ and take a union bound.
Since $n_{*,A}\leq n$, for every fixed $\pi$ and every $p\geq1$,
\begin{equation}
  \Prob\left\{|L_n(\pi)-\E L_n(\pi)|
  >C\left[\sqrt{\frac{p}{n_{*,A}}}
  \{g_T(\Helbarstar{f_\pi}{f_{\pistar}})+\rbar{2}+\rbar{3}\}^{1/2}
  +\frac{p}{n_{*,A}}\right]\right\}\leq e^{-p}.
  \label{eq:dependence-pointwise-tail}
\end{equation}

\paragraph{Covering argument}

Since $\eta_1,\ldots, \eta_m$ satisfy the LSI condition \cref{assump:dependence-lsi}, it
is well-known that $c'\eta$ for $\norm{c}_2=1$ is uniformly sub-Gaussian
\citep[][Theorem 5.3]{ledoux2001concentration}. Because
$\theta_i$ are uniformly bounded, this implies that the marginal distributions of $X_i$
are uniformly sub-Gaussian. Hence, with probability at least $1-C m^{-D}$, \[
	\max_i |X_i| \le C_D \sqrt{\log m},
\]
for some $D > 0$ to be chosen.

On this event, repeat the covering argument in the proof of
\cref{prop:local-maximal}.  Take an $m^{-s}$ cover of $\mathcal G(M,r)$ under the
pseudometric
\[
  d(\pi, \tilde\pi) := \max_{i\leq n}\sup_{|x|\leq C_D\sqrt{\log m}}
  |f_{\pi,\sigma_i}(x)-f_{\tilde\pi,\sigma_i}(x)|.
\]
By \citet[Lemma~4]{Jiang2020} and the packing-covering inequality in
\citet[(2.6)]{ghosal2001entropies}, its cardinality $N$ satisfies
$\log N\leq C_s \log^2m$. For $\pi_j$ a member of the cover, $s$ can be chosen
sufficiently large such that with probability at least $1-Cm^{-D}$ \[
	\sup_{\pi\in\mathcal G(M,r)}\min_{j\leq N}
  \left\{|L_n(\pi)-L_n(\pi_j)|+|\E L_n(\pi)-\E L_n(\pi_j)|\right\}
  \leq m^{-D-4}.
\]
Therefore, for $p \ge 4$, we have the following uniform statement: \[
	\sup_{\pi \in \mathcal G(M,r)} \abs{L_n(\pi) - \E L_n(\pi)} \ge C\left[\sqrt{\frac{p}{n_
	{*,A}}}
  \{g_T(r)+\rbar{2}+\rbar{3}\}^{1/2}
  +\frac{p}{n_{*,A}}\right] + m^{-D-4} \numberthis \label{eq:final_bound}
\]
with probability at most $
	\exp\pr{-p + \log N} + C m^{-D}.
$

To finish, we choose $p$. Let $ p=c_0\frac{n_{*,A}}{\log^2m} r \ge c_0 c^* \log^2 m \ge 4
	$ by \eqref{eq:dependence-local-radius} (requiring that $c_0 c^* \ge 4$). Condition
\eqref{eq:dependence-local-radius} ensures both that
$\rbar{2}+\rbar{3}\leq r/c^*$ and, because
$n_{*,A}\leq n\leq m$,
\[
  r\geq c^*\frac{\log^4m}{m},
  \quad
  \log(T/r)\leq C\log m,
  \quad
  g_T(r)\leq Cr\log^2m,
  \quad
  p\geq c_0c^*\log^2m.
\]
Consequently,
\begin{align*}
  \sqrt{\frac p{n_{*,A}}}
  \{g_T(r)+\rbar{2}+\rbar{3}\}^{1/2}
  \leq C\sqrt{c_0}\,r, \quad
  \frac p{n_{*,A}}=\frac{c_0r}{\log^2m}\leq c_0r, \quad r\gtrsim \frac{\log^4(m)}{m}
  \gtrsim m^
  {-D-4}.
\end{align*}
Choose $c_0$ sufficiently small and $c^*$ sufficiently large, maintaining $c_0 c^* \ge 4$,
so that $\log N\leq p/2$. This gives  $\exp\pr{-p + C_s \log^2 m} \le 2\exp\left\{-C
\frac{n_{*,A}r}{\log^2m}\right\}.$
Plugging these choices into \cref{eq:final_bound} proves the claim.
\end{proof}

We now obtain the Hellinger bound under \cref{asn:assumeindep} and
\eqref{eq:dependence-score-condition}.  This is the analogue of \cref{thm:main} for the
dependent likelihood.

\begin{thm}
\label{thm:dependence-approximate-normality}
Suppose $\theta_i\iid\pistar$, $\theta$ is independent of $\eta$,
\eqref{eq:dependence-weighted-representation}--
\eqref{eq:dependence-scale-normalization}, \cref{assump:dependence-lsi}, the two
conditions in \cref{asn:assumeindep}, and \eqref{eq:dependence-score-condition} hold.
Let $\pihat$ be  the composite NPMLE in \eqref{eq:dependence-npmle}.  For every fixed $s>0$,
take $C$ sufficiently large and, within this theorem and its proof, define
\begin{equation*}
  \delta_n^2
  :=C\left\{\frac{\log^4m}{n_{*,A}}+\rbar{1}+\rbar{2}+\rbar{3}\right\}.
\end{equation*}
Then
\begin{equation*}
  \Prob\{\Helbarstar{f_{\pihat}}{f_{\pistar}}>\delta_n^2\}\leq Cm^{-s},
  \qquad
  \E\Helbarstar{f_{\pihat}}{f_{\pistar}}\leq C\delta_n^2.
\end{equation*}
\end{thm}

\begin{proof}
The proof exactly repeats the peeling argument in the proof of \cref{thm:main}. In
particular, we can run the same proof by replacing $b_n \gets 0$, $\eta \gets 0$, $S_{nt}
\gets \R$, $A_
{nt}(\pi) \gets L_n(\pi)$, $\log^4n  / n + \bar r_2 \gets \frac{\log^4m}{n_{*,A}} + \bar
r_2 + \bar r_3$, and replacing \cref{prop:local-maximal} with
\cref{prop:dependence-local-maximal}. Choose $D=s+2$ in \cref{prop:dependence-local-maximal}.
\end{proof}
\subsection{Verification via Edgeworth expansion for smooth functions}

\begin{lemma}
\label{lem:dependence-weighted-gaussian-comparison}
Let $\eta_1,\ldots,\eta_m$ satisfy \cref{assump:dependence-lsi}, let $\Theta\in[-M,M]$ be
independent of the $\eta_t$, and let $V_1,\ldots,V_m$ be independent standard Gaussians
that are independent of $(\eta_1,\ldots,\eta_m,\Theta)$.  For deterministic coefficients
$c_1,\ldots,c_m$ satisfying $\sum_t c_t^2=1$, define
\[
  X=\Theta+\sigma\sum_{t=1}^m c_t\eta_t,
  \quad
  W=\Theta+\sigma\sum_{t=1}^m c_tV_t,
\]
where $0 < \sigma \lesssim 1$ is a fixed constant.
Suppose $g$ has six
continuous derivatives and, for fixed $q \ge 0$ and $L\geq1$,
\[
  |g^{(j)}(x)|\leq L(1+|x|^q),
  \qquad j=4,5,6.
\]
Then
\begin{equation}
  |\E g(X)-\E g(W)|
  \leq C\left|\sum_{t=1}^m c_t^3\right||\E g^{(3)}(W)|
       +CL\sum_{t=1}^m c_t^4.
  \label{eq:dependence-weighted-gaussian-comparison}
\end{equation}
\end{lemma}

\begin{proof}

Towards applying \cref{thm:barbour-smooth-edgeworth}, set
\[
  Y_t:=c_t\eta_t,
  \quad S:=\sum_{t=1}^mY_t,
  \quad \bar g(x):=\E g(\Theta+\sigma x).
\]
The variables $Y_t$ are independent and centered, and
$\sum_t\E Y_t^2=1$.  The mean value theorem and the assumed bound on
$g^{(4)}$ yield
\[
  \mathcal L_{3,q+1,1}(\bar g)\leq CL.
\]
Apply \cref{thm:barbour-smooth-edgeworth} with $a=4$, $p=q$, and
$\alpha=1$.  The only indices in \eqref{eq:barbour-edgeworth-expansion} are
$m_3=1$, $m_4=1$, and $m_3=2$, so
\begin{align}
  \E\bar g(S)-\E\bar g(Z)
  ={}&\frac{\kappa_3(S)}{6}\E\{H_3(Z)\bar g(Z)\}
     +\frac{\kappa_4(S)}{24}\E\{H_4(Z)\bar g(Z)\}\notag\\
    &+\frac{\kappa_3(S)^2}{72}\E\{H_6(Z)\bar g(Z)\}+R,
  \label{eq:dependence-barbour-expansion}
\end{align}
where \cref{eq:barbour-edgeworth-remainder} and
\cref{assump:dependence-lsi} give
\begin{equation}
  |R|
  \leq CL\sum_{t=1}^m
  \left\{\E|c_t\eta_t|^5+\E|c_t\eta_t|^{5+q}\right\}
  \leq CL\sum_{t=1}^mc_t^4.
  \label{eq:dependence-barbour-remainder}
\end{equation}
Here we used $|c_t|\leq1$ and the uniform moment bounds implied by the LSI.

It remains to bound various objects in the expansion
\eqref{eq:dependence-barbour-expansion}.
The cumulants in \eqref{eq:dependence-barbour-expansion} are
\[
  \kappa_3(S)=\E(\eta_1^3)\sum_{t=1}^mc_t^3,
  \qquad
  \kappa_4(S)=\kappa_4(\eta_1)\sum_{t=1}^mc_t^4.
\]
Independence of $\Theta$ gives $\E\bar g(S)=\E g(X)$.  Also,
$\sum_t c_tV_t\sim\cN(0,1)$, so $\E\bar g(Z)=\E g(W)$, and Gaussian
integration by parts gives, for $r=3,4,6$,
$
  \E\{H_r(Z)\bar g(Z)\}=\sigma^r\E g^{(r)}(W).
$
The LSI and the derivative assumptions imply
$|\E\eta_1^3|+|\kappa_4(\eta_1)|\leq C$ and
$|\E g^{(r)}(W)|\leq CL$ for $r=4,6$.  Finally,
\[
  \left(\sum_{t=1}^mc_t^3\right)^2
  \leq\left(\sum_{t=1}^mc_t^2\right)
       \left(\sum_{t=1}^mc_t^4\right)
  =\sum_{t=1}^mc_t^4.
\]
Substituting the above into \eqref{eq:dependence-barbour-expansion}, together with
\eqref{eq:dependence-barbour-remainder}, proves
\eqref{eq:dependence-weighted-gaussian-comparison}.
\end{proof}

The following bounds are used when applying \cref{lem:dependence-weighted-gaussian-comparison}.
Let $Z_i=\theta_i+\sigma_iV_i$, where $V_i\sim\cN(0,1)$ is independent
of $\theta_i$.
\begin{lemma}
\label{lem:dependence-gaussian-localization}
Uniformly over $i$ and $\pi\in\cP([-M,M])$,
\begin{align}
  \E t_{\pi,i}(Z_i)
  &\leq-2\Helsq(f_{\pi,\sigma_i},f_{\pistar,\sigma_i}),
  \label{eq:dependence-gaussian-drift}\\
  \E t_{\pi,i}(Z_i)^2
  &\leq C g_T\!\left\{\Helsq(f_{\pi,\sigma_i},f_{\pistar,\sigma_i})\right\},
  \nonumber\\[-2mm]
  \E t_{\pi,i}(Z_i)^4
  &\leq C\Helsq(f_{\pi,\sigma_i},f_{\pistar,\sigma_i})
  \log^4\!\left\{
    \frac{e}{\Helsq(f_{\pi,\sigma_i},f_{\pistar,\sigma_i})}
  \right\},
  \label{eq:dependence-gaussian-log-moments}\\
  \E\{\sigma_it_{\pi,i}'(Z_i)\}^2
  &\leq Cg_T\!\left\{\Helsq(f_{\pi,\sigma_i},f_{\pistar,\sigma_i})\right\}.
  \label{eq:dependence-gaussian-score}
\end{align}
The right sides are interpreted as zero when the two densities agree.
\end{lemma}

\begin{proof}
Equation \eqref{eq:dependence-gaussian-drift} follows because the expectation is
$-\operatorname{KL}(f_{\pistar,\sigma_i}\,\Vert\,f_{\pi,\sigma_i})$ and
$\operatorname{KL}\geq2\Helsq$.  The second and fourth moments of $t_{\pi,i}$ in
\eqref{eq:dependence-gaussian-log-moments} are bounded by
\cref{lem:Vk-hellinger-bound} with $k=2,4$.  To use \cref{lem:Vk-hellinger-bound}, we
can verify that the hypothesis of
\cref{lem:Vk-hellinger-bound} holds uniformly over $i, \pi$ since $\theta$ is bounded
and \cref{assump:lowerbound} holds.

Finally,
\[
  \sigma_it_{\pi,i}'(x)
  =\sigma_i\frac{d}{dx}\log\frac{f_{\pi,\sigma_i}(x)}
                                      {f_{\pistar,\sigma_i}(x)}
  =\sigma_i\left\{
    \frac{f_{\pi,\sigma_i}'(x)}{f_{\pi,\sigma_i}(x)}
    -\frac{f_{\pistar,\sigma_i}'(x)}{f_{\pistar,\sigma_i}(x)}
  \right\}.
\]
Let $\pi^{(i)}$ and $\pistar^{(i)}$ be the pushforwards of $\pi$ and $\pistar$ under the
change of variables
$v\mapsto v/\sigma_i$. This change of variables does not affect the Hellinger distance: $\Helsq(f_{\pi,\sigma_i},f_{\pistar,\sigma_i})
   =\Helsq(f_{\pi^{(i)},1},f_{\pistar^{(i)},1})$. We can then
verify \[
	\E\{\sigma_it_{\pi,i}'(Z_i)\}^2 =\int f_{\pistar^{(i)},1}(y)
    \left|\frac{d}{dy}
      \log\frac{f_{\pi^{(i)},1}(y)}{f_{\pistar^{(i)},1}(y)}
    \right|^2dy =: \operatorname{Regret}(\pi^{(i)}\Vert\pistar^{(i)})
\]
where $\operatorname{Regret}(\pi^{(i)}\Vert\pistar^{(i)})$ is defined in
\citet{chen2026sharp}. Since $k < \sigma_i \lesssim 1$, the rescaled priors are supported
in some fixed compact interval. Theorem 1 in \citet{chen2026sharp}, which bounds regret by
Hellinger, then implies
\eqref{eq:dependence-gaussian-score}:\footnote{Technically, Theorem 1 in
\citet{chen2026sharp} holds for sufficiently small
$\Helsq(f_{\pi,\sigma_i},f_{\pistar,\sigma_i})$. If $\Helsq(f_
{\pi,\sigma_i},f_{\pistar,\sigma_i})$ is large, since $\pi^{(i)}$'s are supported in
uniformly bounded intervals, regret is uniformly bounded over $i$ and the same conclusion
holds by enlarging the implicit constant.}
\[
	\E\{\sigma_it_{\pi,i}'(Z_i)\}^2 = \operatorname{Regret}(\pi^{(i)}\Vert\pistar^{(i)})
	\lesssim g_T(\Helsq(f_
	{\pi,\sigma_i},f_{\pistar,\sigma_i})). \qedhere
\]
\end{proof}

\begin{prop}
\label{prop:dependence-gaussian-comparison-verification}
Suppose $\theta_i\iid\pistar\in\cP([-M,M])$, $\theta$ is independent of $\eta$,
\eqref{eq:dependence-weighted-representation}--\eqref{eq:dependence-scale-normalization}
hold, and the $\eta_t$ satisfy \cref{assump:dependence-lsi}.  Then the marginal law of
$X_i$ satisfies the two conditions in \cref{asn:assumeindep} and
\eqref{eq:dependence-score-condition} with nonnegative remainders such that
 \begin{equation}
  r_{1,i}+r_{2,i}+r_{3,i}
  \leq C(\log^4m)\sum_{t=1}^m c_{it}^4.
  \label{eq:dependence-gaussian-comparison-remainders}
\end{equation}
\end{prop}

\begin{proof}

Apply \cref{lem:dependence-weighted-gaussian-comparison} with $c_t=c_{it}$ and,
successively, the
following test functions
\[
  g=t_{\pi,i},\qquad
  g=t_{\pi,i}^2,
  \qquad
  g=(\sigma_it_{\pi,i}')^2.
\]
By \cref{lem:deriv-cumulant-bound}, every positive-order derivative of $t_{\pi,i}$ is
uniformly bounded, while $t_{\pi,i}$ itself has at most linear growth. Thus derivatives
four through six of the first two test functions satisfy the requirements of
\cref{lem:dependence-weighted-gaussian-comparison}. For
\(g= (\sigma_i t_{\pi,i}')^2\), we have
\[
g^{(k)}
=\sigma_i^2\sum_{j=0}^k\binom{k}{j}
t_{\pi,i}^{(j+1)}t_{\pi,i}^{(k-j+1)},
\qquad k=4,5,6.
\]These expressions involve derivatives of \(t_{\pi,i}\) through order seven, all of which are uniformly bounded by \cref{lem:deriv-cumulant-bound}.

For any smooth $g$, Gaussian integration by parts gives
\begin{equation}
  \E g^{(3)}(Z_i)
  =\sigma_i^{-3}\E\left[
    \left\{\left(\frac{\varepsilon_i}{\sigma_i}\right)^3
      -3\frac{\varepsilon_i}{\sigma_i}\right\}g(Z_i)
  \right],
  \label{eq:dependence-third-order-ibp} \quad Z_i = \theta_i + \varepsilon_i,\,
  \varepsilon_i \sim
  \Norm(0, \sigma_i^2)
\end{equation}
Equation \eqref{eq:dependence-third-order-ibp}, Cauchy--Schwarz, and
\cref{lem:dependence-gaussian-localization} imply, respectively, for $u = \Helsq(f_{\pi,\sigma_i},f_{\pistar,\sigma_i})$,
\begin{align}
  |\E t_{\pi,i}^{(3)}(Z_i)|
  &\leq C\sqrt{g_T(u)}, \notag
  \\
  |\E\{(t_{\pi,i}^2)^{(3)}(Z_i)\}|
  &\leq C\sqrt{u}\log^2(e/u), \notag
  \\
  |\E\{((\sigma_it_{\pi,i}')^2)^{(3)}(Z_i)\}|
  \le C \abs{
  	\E[|\sigma_i t'_{\pi,i}|^4]
  }^{1/2} \le C \E[|\sigma_i t'_{\pi,i}|^2]^{1/2}&\leq C\sqrt{g_T(u)},
  \label{eq:dependence-third-derivative-score}
\end{align}
where \eqref{eq:dependence-third-derivative-score} uses compact support, which implies
$|\sigma_it_{\pi,i}'|\leq C$. Plugging these bounds into the conclusion of
\cref{lem:dependence-weighted-gaussian-comparison},
\eqref{eq:dependence-weighted-gaussian-comparison}, yields
the claim after applying Young's inequality and $|\sum_t c_i^3| \le  \sum_t |c_i|^3$.\end{proof}

\begin{lemma}
\label{lem:dependence-posterior-loss}
Suppose the conditions of \cref{prop:dependence-local-maximal} hold.  For
every fixed $D>0$, there are constants $c^*,C>0$ such that, whenever
\[
  c^*\left\{\frac{\log^4m}{n_{*,A}}+\rbar{3}\right\}
  \leq r\leq1,
\]
we have
\begin{equation*}
  \Prob\left\{
  \sup_{\substack{\pi\in\cP([-M,M]):\\
  \Helbarstar{f_\pi}{f_{\pistar}}\leq r}}
  \avgn
  |h_\pi(X_i;\sigma_i)-h_{\pistar}(X_i;\sigma_i)|^2
  >Cr\log^2m
  \right\}
  \leq Cm^{-D}.
\end{equation*}
\end{lemma}
\begin{proof}
Tweedie's formula gives
\[
  h_\pi(x;\sigma_i)-h_{\pistar}(x;\sigma_i)
  =\sigma_i^2t_{\pi,i}'(x).
\]
Thus the statement of this lemma is a local maximal inequality for averages of $q_{\pi,i}(X_i) := \sigma_i^4 t'_{\pi, i}(X_i)^2 $. We thus repeat the proof of \cref{prop:dependence-local-maximal}, which is such a result for averages of $t_{\pi,i}$. Observe that \[
0 \le q_{\pi,i} \le 4M^2, \quad |q'_{\pi,i}|^2 \le C q_{\pi,i}, \quad |q_{\pi,i}''| \le C.
\]

We first note that
\[
\frac1n\sum_{i=1}^n
\mathbb E\!\left[
q_{\pi,i}(X_i)
+ \bigl(\sigma_i q'_{\pi,i}(X_i)\bigr)^2
+ q_{\pi,i}(X_i)^2
\right]
\lesssim
g_T\!\left(\overline{\mathrm{Hel}}^{\,2}(f_\pi,f_{\pi_\star})\right)+\bar r_3.
\]
The bounds on the first term on the LHS follow from \eqref{eq:dependence-score-condition} and $\sigma_i \in [k,K]$. The bounds on the other two terms on the LHS follow from the bounds on $q, q'$ above.

The proof of \eqref{eq:dependence-pointwise-tail} applies analogously to $q_{\pi,i}$ in
place of $t_{\pi,i}$, using the second and third terms on the LHS of the last display. Following through this argument, we obtain that with probability at least \(1-e^{-p}\),
\[
\left|\frac1n\sum_{i=1}^n
\bigl\{q_{\pi,i}(X_i)-\mathbb E q_{\pi,i}(X_i)\bigr\}\right|
\lesssim
\sqrt{\frac{p}{n_{*,A}}
\left\{g_T\!\left(\overline{\mathrm{Hel}}^{\,2}(f_\pi,f_{\pi_\star})\right)+\bar r_3\right\}}
+\frac{p}{n_{*,A}}.
\]
Adding \(\frac1n\sum_i\mathbb E q_{\pi,i}(X_i)\), using the bound above, and using the AM--GM inequality yields \[
\frac{1}{n} \sum_i q_{\pi,i}(X_i) \lesssim g_T\!\left(\overline{\mathrm{Hel}}^{\,2}(f_\pi,f_{\pi_\star})\right)
+\bar r_3+\frac{p}{n_{*,A}}
\]
with probability at least $1-e^{-p}$. Since \(q_{\pi,i}(X_i)=|h_\pi(X_i;\sigma_i)-h_{\pi_\star}(X_i;\sigma_i)|^2\), this is the desired bound pointwise in \(\pi\):
\[
\mathbb P\!\left\{
\frac1n\sum_{i=1}^n
|h_\pi(X_i;\sigma_i)-h_{\pi_\star}(X_i;\sigma_i)|^2
>
C\left[
g_T\!\left(\overline{\mathrm{Hel}}^{\,2}(f_\pi,f_{\pi_\star})\right)
+\bar r_3+\frac{p}{n_{*,A}}
\right]
\right\}
\le e^{-p}.
\]

Finally, by Proposition SM6.1 of Chen (2026), applied with \(q=1\), we repeat the covering
argument in \cref{lem:estimated-variance-centered}, with the additional observation that $f_{\pi, \sigma_i}(x) \ge m^{-C}$
uniformly over \(\pi,i\) and \(|x|\le C_D\sqrt{\log m}\). This covering argument makes the preceding bound uniform and concludes the proof.
\end{proof}

\begin{proof}[Proof of \cref{thm:quantitative-dependence}]
Average \eqref{eq:dependence-gaussian-comparison-remainders} in
\cref{prop:dependence-gaussian-comparison-verification} and use
\eqref{eq:dependence-design-quantities} to obtain
\[
  \rbar{1}+\rbar{2}+\rbar{3}\leq C\frac{\log^4m}{J_A}.
\]
Substituting into
\cref{thm:dependence-approximate-normality} proves
\eqref{eq:quantitative-dependence-rate}. If the errors are Gaussian, $X_i$ has density
$f_{\pistar,\sigma_i}$, so
\eqref{eq:condition-a}, \eqref{eq:condition-b}, and
\eqref{eq:dependence-score-condition} hold with
$r_{1,i}=r_{2,i}=r_{3,i}=0$.  The expectation conclusion follows because $\Helsq\leq1$.
\end{proof}

\begin{proof}[Proof of \cref{thm:quantitative-dependence-maintext}]
The Hellinger rates are directly proved in \cref{thm:quantitative-dependence}. We prove the regret rates by following the proof of \cref{thm:denoising-normal}. Couple each \(X_i\) to \(Z_i\mid\theta_i\sim N(\theta_i,\sigma_i^2)\) optimally in \(W_2\). As in the proof of Theorem 4.3,
\[
(\mathrm{Regret}_n)_+
\le
\left\{
\frac1n\sum_i
\mathbb E\left|
h_{\hat\pi_n}(X_i;\sigma_i)
-h_{\pi_\star}(X_i;\sigma_i)
\right|^2
\right\}^{1/2}
+C\mathcal W_n.
\] \Cref{lem:dependence-posterior-loss}, together with the Hellinger rates and boundedness of the posterior means, shows that
\[
\left\{
\frac1n\sum_i
\mathbb E\left|
h_{\hat\pi_n}(X_i;\sigma_i)
-h_{\pi_\star}(X_i;\sigma_i)
\right|^2
\right\}^{1/2}
\lesssim
(\log m)^3
\left(\frac1{n_{*,A}}+\frac1{J_A}\right)^{1/2}.
\]The Wasserstein CLT argument used in the proof of \cref{thm:main_rates}, applied here to the independent weighted summands \(c_{it}\eta_t\), gives
\[
\mathcal W_n^2
\lesssim
\frac1n\sum_i\sigma_i^2
\sum_t\mathbb E|c_{it}\eta_t|^4
\lesssim \frac1{J_A}.
\] Combining these displays yields
\[
(\mathrm{Regret}_n)_+^2
\lesssim
(\log m)^6
\left(\frac1{n_{*,A}}+\frac1{J_A}\right).
\]For exactly normal errors, \(\mathcal W_n=0\), and the Hellinger bounds and \cref{lem:dependence-posterior-loss} omit the \(J_A^{-1}\) term, so we would also omit $J_A^{-1}$ in the regret bound.
\end{proof}

\section{High-level conditions for EB consistency under approximate normality and dependence}\label{sec:general}

\Cref{sec:linear_measurements} establishes convergence rates when the observations $X$ admit a linear
representation in independent error terms. Empirical Bayes is also applied to estimates from
nonlinear models. For example, \citet{gu2022ranking} fit a Bradley--Terry model to pairwise
comparisons and apply Gaussian empirical Bayes methods to the resulting maximum likelihood
estimates and estimated standard errors. These estimates are dependent. Such applications motivate studying the
properties of the estimates directly.

This section establishes \emph{consistency} under conditions imposed directly on the joint
distribution of the estimates. We allow for dependence among both the observations and the
latent parameters, estimated variances, and weaker tail conditions than the log-Sobolev inequality in \cref{assump:dependence-lsi}. These conditions
yield consistent estimation of the latent distribution and vanishing regret relative to
the marginal normal oracle used in \cref{sec:linear_measurements}, though without explicit convergence rates.

Throughout, we assume that $\theta_i \sim \pistar$ \emph{marginally} for some $\pistar \in
\cP([-M,M])$, with the $\theta_i$ and the $X_i$ no longer necessarily independent across
$i$, and the $X_i$ only approximately normal given $\theta_i$. We write $\mu_i$ for the
marginal law of $X_i$ as in \cref{thm:main}. Our first condition quantifies a notion of
dependence across units: Roughly speaking, we assume that $X_i$ are sufficiently independent that laws
of large numbers still apply to them.

\begin{assume}[Approximate independence]\label{asn:taildef}
  For any fixed sequence of functions $f_i\colon \mathbb{R} \to \mathbb{R}$, $1\le i\le n$, such that $|f_i'(0)|\le 1$ and $\max_{1\le i\le n}\lVert f_i''\rVert_{\infty}\le 1$, we assume that
  \begin{equation*}
    \frac{1}{n}\sum_{i=1}^{n}\bigl(f_i(X_i) - \E\,f_i(X_i)\bigr)
    \pto 0
    \qquad \text{and} \qquad
    \limsup_{T\to\infty}\limsup_{n\to\infty} \frac{1}{n}\sum_{i=1}^{n} \E\bigl[X_i^2 \mathbf{1}(|X_i|\ge T)\bigr] = 0.
  \end{equation*}
\end{assume}

Assumption~\ref{asn:taildef} requires only pointwise---not uniform---convergence over the class of function sequences with second derivatives bounded by $1$. The integrability condition is likewise weaker than the uniform sub-Gaussianity of Assumption~\ref{assump:sub-Gauss}: it only constrains the average second moment of $X_i$. We
use an analogue of the approximate optimality condition in
\eqref{eq:likelihood-bound} and allow the variances $\sigma_i^2$ to be estimated.

\begin{assume}[Approximate optimality with estimated variances]\label{asn:appopt}
  Let $\hat{\sigma}_i \equiv \hat{\sigma}_i(X_1, \ldots, X_n)$ satisfy $\min_i\{\hat{\sigma}_i, \sigma_i\} \geq c$ and $\max_i\{\hat{\sigma}_i, \sigma_i\} \leq C$ almost surely, for fixed constants $c, C > 0$, and let $\pihat \in \cP([-M,M])$ almost surely. We assume that there exists a deterministic sequence $r_n \to 0$ such that $\pihat, \hat{\sigma}_1, \ldots, \hat{\sigma}_n$ satisfy
  \begin{equation*}
    \frac{1}{n}\sum_{i=1}^{n}\log\int \phi(X_i - \theta;\, \hat{\sigma}_i)\,
    d\pihat(\theta)
    \geq
    \frac{1}{n}\sum_{i=1}^{n}\log\int \phi(X_i - \theta;\, \hat{\sigma}_i)\,
    d\pistar(\theta)
    - r_n,
  \end{equation*}
  and
  $
    \frac{1}{n}\sum_{i=1}^{n}\bigl|\hat{\sigma}_i - \sigma_i\bigr| \pto 0.
  $
\end{assume}

The upper and lower bounds on $\hat{\sigma}_i$ can be enforced by clipping standard
variance estimators. The sequence $r_n$ may be taken to be zero when $\pihat$ is the NPMLE
with the estimated standard deviations $\hat{\sigma}_i$ plugged into the misspecified
objective. The average convergence of $\hat{\sigma}_i$ to $\sigma_i$ is a mild assumption.

We now state the main result of this section.
The result is stated in terms of two discrepancies between the observed data and the
normal working model. At the level of the marginal law $\mu_i$ of $X_i$, define the
averaged Wasserstein-1 discrepancy
\[
  \mathcal{V}_n := \avgn W_1\!\left(\mu_i,\, f_{\pistar,\sigma_i}\right),
\]
which controls consistent deconvolution. For the regret guarantee we additionally use
the conditional Wasserstein-2 discrepancy $\mathcal{W}_n$ of Section~\ref{sec:deconvolution}.

\begin{thm}[Consistency under dependence]\label{thm:depcons}
  Suppose Assumptions~\ref{asn:taildef} and~\ref{asn:appopt} hold. If $\mathcal{V}_n \to 0$,
  then $\pihat$ converges weakly to $\pistar$ in probability and
  \begin{equation*}
    \avgn \Helsq\!\left(f_{\pihat,\sigma_i},\, f_{\pistar,\sigma_i}\right)
    \pto 0.
  \end{equation*}
  If additionally \(\mathcal W_n\to0\), then \begin{equation*}
    \mathrm{Regret}_n \to 0,
  \end{equation*} with $\mathrm{Regret}_n$ as in~\eqref{eq:regret-z}.
\end{thm}

Theorem~\ref{thm:depcons} shows that consistent deconvolution and vanishing regret can be achieved by nonparametric maximum likelihood under mild average integrability of the $X_i$, a weak notion of approximate normality, and an averaged notion of approximate independence. To accommodate these weaker conditions, we depart from the proof techniques of~\citet{zhang2009generalized} and~\citet{ghosal2001entropies}, which rely on exponential moment bounds or concentration inequalities and thus require considerably stronger tail conditions. Our argument is instead inspired by recent developments in empirical Bayes for principal component analysis. Specifically,~\citet*{zhong2022empirical} study the nonparametric maximum likelihood estimator with approximately normal quantities that may be correlated. While their analysis focuses on a specific random matrix setting,
we build on their argument to establish consistency under the
general setting considered here.

\subsection{Proof of results from \cref{sec:general}}
\label{sec:pfgeneral}
In order to prove \cref{thm:depcons}, we introduce the following normal-convolution regularized Wasserstein distance
\begin{align}\label{eq:GOT}
    d_{\eta}(\pi_1,\pi_2):=W_1(f_{\pi_1,\eta},f_{\pi_2,\eta})
\end{align}
for any $\eta>0$, where $W_1$ is the $1$-Wasserstein distance. By \citet[Theorem 1]{goldfeld2020gaussian}, $d_{\eta}$ metrizes weak convergence for all fixed $\eta>0$. The proof of \cref{thm:depcons} proceeds through a sequence of lemmas. The first result shows that the KL divergence has curvature with respect to $d_{\eta}$ from \eqref{eq:GOT} for all $\eta>0$.
\begin{lemma}\label{lem:curvaturebd}
    Fix arbitrary $\delta,\eta>0$ and $1\le i\le n$ and assume the uniform bounds on $\sigma_i$ in \cref{asn:appopt} hold. Then there exists a constant $C>0$ (free of $\eta, \delta, i$) such that if $d_{\eta}(\pi_1,\pi_2)\ge \delta$ for some $\pi_1,\pi_2\in\mathcal{P}([-M,M])$, then
    \begin{align}\label{eq:curvaturebd}
    \inf_{1\le i\le n} \operatorname{KL}(f_{\pi_1,\sigma_i} {\,\Vert\,}f_{\pi_2,\sigma_i})\ge \frac{1}{2}\exp\left(-2\frac{C^2}{\delta^2}\right)=:\tau(\delta).
    \end{align}
    Note that $\tau(\delta)\to 0$ as $\delta\to 0$.
\end{lemma}

The next lemma says that the likelihood at any $\pi$ is separated from the likelihood at $\pst$ with high probability, provided $\pi$ and $\pst$ are well separated in the $d_{\eta}(\cdot,\cdot)$ metric.

\begin{lemma}\label{lem:ptwiseconc}
    Fix any $\eta,\delta>0$. Let $\tau\equiv\tau(\delta)$ be defined as in \eqref{eq:curvaturebd}. Suppose \cref{asn:taildef} and $\mathcal{V}_n\to 0$ hold. Then for any $\pi\in\mathcal{P}([-M,M])$ satisfying $d_{\eta}(\pi,\pst)\ge\delta$, we have:
    $$\PP\left(\frac{1}{n}\sum_{i=1}^n \log f_{\pi,\sigma_i}(X_i)-\frac{1}{n}\sum_{i=1}^n \log f_{\pst,\sigma_i}(X_i)\ge -\tau/3\right) \to 0.$$
\end{lemma}

\noindent For the next result, we define a pseudo-metric between $\pi_1,\pi_2\in\mathcal{P}([-M,M])$ as follows: Fix any $\eta>0$, $B>0$ and let $F_{\pi,\eta}$ denote the cumulative distribution function of $\cN(0,\eta^2)*\pi$ for $\pi\in\mathcal{P}([-M,M])$. We define the following pseudo-metric:
$$\tilde{d}_{\eta,B}(\pi_1,\pi_2) := \max_{x:|x|\le B} |F_{\pi_1,\eta}(x)-F_{\pi_2,\eta}(x)|.$$
This next lemma shows that if two distributions $\pi_1,\pi_2\in \cP([-M,M])$ are close in the $\tilde{d}_{\eta,B}(\cdot,\cdot)$ pseudo-metric, then their likelihoods are close too.

\begin{lemma}\label{lem:coverbd}
    Fix any constants $C_0,\iota>0$ and suppose \cref{asn:appopt} holds. Then, on the event $\sum_{i=1}^n X_i^2\le C_0 n$, there exist $\eta,B,L>0$ (depending on $C_0,\iota$) such that for any $\pi_1,\pi_2\in\mathcal{P}([-M,M])$, we have
    $$\bigg|\frac{1}{n}\sum_{i=1}^n \log{f_{\pi_1,\sigma_i}(X_i)} - \frac{1}{n}\sum_{i=1}^n \log{f_{\pi_2,\sigma_i}(X_i)}\bigg|\le L \tilde{d}_{\eta,B}(\pi_1,\pi_2)+\iota.$$
\end{lemma}

\noindent The next result shows that $\sum_{i=1}^n X_i^2\le C_0 n$ holds with high probability for large $C_0$.
\begin{lemma}\label{lem:truncate}
    Under \cref{asn:taildef}, there exists a large enough $C_0>0$ such that
    $$\PP\left(\sum_{i=1}^n X_i^2\ge C_0 n\right)\to 0.$$
\end{lemma}

The final technical lemma shows a high-probability separation between the likelihoods at $\pi$ and $\pst$ uniformly over all $\pi$ such that $d_{\eta}(\pi,\pst)\ge \delta$.
\begin{lemma}\label{lem:unifbd}
    Fix any $\delta>0$ and suppose \cref{asn:taildef,asn:appopt} and $\mathcal{V}_n\to 0$ hold. Then there exist $\tau>0$ (depending on $\delta$) and $\eta>0$ (free of $\delta$) such that
    $$\PP\left(\sup_{\pi\in\mathcal{P}([-M,M]): d_{\eta}(\pi,\pst)\ge \delta}\frac{1}{n}\sum_{i=1}^n \big(\log{f_{\pi,\sigma_i}(X_i)}-\log{f_{\pst,\sigma_i}(X_i)}\big)\ge -\tau\right)\to 0$$
    as $n\to\infty$.
\end{lemma}

The final lemma helps characterize the price paid due to estimating the standard deviations $\sigma_1,\ldots ,\sigma_n$ with $\hat{\sigma}_1,\ldots ,\hat{\sigma}_n$.

\begin{lemma}\label{lem:diffbd}
    Suppose $0<C_1<\sigma,\tilde{\sigma}<C_2<\infty$. Then for any $T\ge 1$ and any $\pi$ supported on $[-M,M]$, the following holds:
    $$\sup_{|x|\le T}\bigg|\log\int\phi\left(x-\theta;\sigma\right)\,d\pi(\theta)-\log\int\phi\left(x-\theta;\tilde{\sigma}\right)\,d\pi(\theta)\bigg|\le \left(\frac{1}{C_1}+\frac{C_2}{C_1^4}(T+M)^2\right)|\sigma-\tilde{\sigma}|.$$
\end{lemma}

\begin{proof}
    Note that
    \begin{align*}
        \phi\left(x-\theta;\tilde{\sigma}\right)=\left(1+\frac{\sigma-\tilde{\sigma}}{\tilde{\sigma}}\right)\exp\left(-\frac{1}{2}\frac{\sigma^2-\tilde{\sigma}^2}{\sigma^2 \tilde{\sigma}^2}(x-\theta)^2\right)\phi\left(x-\theta;\sigma\right).
    \end{align*}
    Therefore, for all $|x|\le T$ and $\pi$ supported on $[-M,M]$, we have:
    \begin{align*}
        \frac{\int \phi\left(x-\theta;\tilde{\sigma}\right)\,d\pi(\theta)}{\int \phi\left(x-\theta;\sigma\right)\,d\pi(\theta)} \vee \frac{\int \phi\left(x-\theta;\sigma\right)\,d\pi(\theta)}{\int \phi\left(x-\theta;\tilde{\sigma}\right)\,d\pi(\theta)} \le \left(1+\frac{|\sigma-\tilde{\sigma}|}{C_1}\right)\exp\left(\frac{C_2}{C_1^4}(T+M)^2 |\sigma-\tilde{\sigma}|\right).
    \end{align*}
    The conclusion follows by the monotonicity of the log function and the elementary inequality $\log{(1+x)}\le x$ for $x>0$.
\end{proof}

\begin{proof}[Proof of \cref{thm:depcons}]

Fix any arbitrary $T > 0$.
By \cref{lem:diffbd}, there exists a deterministic constant $c_T > 0$ such that
\begin{align}\label{eq:step1}
  &\sup_{\pi \in \mathcal{P}([-M,M])}
  \frac{1}{n}\sum_{i=1}^{n}
  \Bigl|
    \log\!\int \phi(X_i - \theta;\,\hat{\sigma}_i)\,d\pi(\theta)
    -
    \log\!\int \phi(X_i - \theta;\,\sigma_i)\,d\pi(\theta)
  \Bigr|
  \mathbf{1}(|X_i| \leq T)
  \nonumber \\ & \leq\;
  \frac{c_T}{n}\sum_{i=1}^{n}|\hat{\sigma}_i - \sigma_i|
  \pto 0,
\end{align}
 where the last line follows by \cref{asn:appopt}.
On the other hand, by \cref{lem:log-mixture-bound}, there exists a
constant $\tilde{C}$ free of $T > 0$ such that for all large enough $T$ we have
\begin{equation}\label{eq:step2}
  \frac{1}{n}\sum_{i=1}^{n}
  \Bigl|
    \log\!\int \phi(X_i - \theta;\,\hat{\sigma}_i)\,d\pihat(\theta)
  \Bigr|
  \mathbf{1}(|X_i| > T)
  \le
  \tilde{C}\cdot\frac{1}{n}\sum_{i=1}^{n} X_i^2\,\mathbf{1}(|X_i| > T)
  \pto 0,
\end{equation}
in the double limit $n\to\infty$ followed by $T\to\infty$, by
\cref{asn:taildef}.
Combining \eqref{eq:step1} and \eqref{eq:step2} we get:
\[
  \frac{1}{n}\sum_{i=1}^{n}
  \Bigl|
    \log\!\int \phi(X_i-\theta;\,\hat{\sigma}_i)\,d\pihat(\theta)
    -
    \log\!\int \phi(X_i-\theta;\,\sigma_i)\,d\pihat(\theta)
  \Bigr|
  \pto 0.
\]
In a similar vein, we have
\[
  \frac{1}{n}\sum_{i=1}^{n}
  \Bigl|
    \log\!\int \phi(X_i-\theta;\,\hat{\sigma}_i)\,d\pst(\theta)
    -
    \log\!\int \phi(X_i-\theta;\,\sigma_i)\,d\pst(\theta)
  \Bigr|
  \pto 0.
\]

Next we fix $\delta > 0$ and let $\tau > 0$ (depending on $\delta$) and $\eta > 0$ (free of
$\delta$) be chosen according to \cref{lem:unifbd}.  The above observations coupled with \cref{asn:appopt}
imply that
\begin{align*}
  &\;\;\;\;\limsup_{n\to\infty}
    \PP\bigl(d_\eta(\pihat,\,\pst) > \delta\bigr)\\
  &\leq
    \limsup_{n\to\infty}
    \PP\Biggl(
      \exists\,\pi\in\mathcal{P}([-M,M]),\;d_\eta(\pi,\pst)>\delta
      \;:\;
      \frac{1}{n}\sum_{i=1}^{n}
      \log
      \frac{
        \int\phi(X_i-\theta;\,\hat{\sigma}_i)\,d\pi(\theta)
      }{
        \int\phi(X_i-\theta;\,\hat{\sigma}_i)\,d\pst(\theta)
      }
      \geq -\frac{\tau}{2}
    \Biggr)\\
  &\leq
    \limsup_{n\to\infty}
    \PP \Biggl(
      \exists\,\pi\in\mathcal{P}([-M,M]),\;d_\eta(\pi,\pst)>\delta
      \;:\;
      \frac{1}{n}\sum_{i=1}^{n}
      \log
      \frac{
        \int\phi(X_i-\theta;\,\sigma_i)\,d\pi(\theta)
      }{
        \int\phi(X_i-\theta;\,\sigma_i)\,d\pst(\theta)
      }
      \geq -\tau
    \Biggr)\\
  &\leq
    \limsup_{n\to\infty}
    \PP\Biggl(
      \sup_{\substack{\pi\in\mathcal{P}([-M,M])\\d_\eta(\pi,\pst)>\delta}}
      \frac{1}{n}\sum_{i=1}^{n}
      \log
      \frac{
        \int\phi(X_i-\theta;\,\sigma_i)\,d\pi(\theta)
      }{
        \int\phi(X_i-\theta;\,\sigma_i)\,d\pst(\theta)
      }
      \geq -\tau
    \Biggr) = 0\,.
\end{align*}
This implies $\pihat$ converges weakly to $\pst$ in probability.
Since both $\pihat,\pst\in\mathcal{P}([-M,M])$, we also have
$W_1(\pihat,\pst)\pto 0$. This observation will be useful for proving Hellinger and regret convergence.

We next prove convergence in average squared Hellinger distance. Observe that
\[
  \Helsq\!\left(\cN(a,\sigma^2),\cN(b,\sigma^2)\right)
  =1-\exp\!\left(-\frac{(a-b)^2}{8\sigma^2}\right)
  \le \frac{(a-b)^2}{8\sigma^2}.
\]
Applying the data-processing inequality to an optimal $W_2$-coupling of the
mixing distributions, and using $\sigma_i\ge c$, gives
\[
  \avgn \Helsq\!\left(f_{\pihat,\sigma_i},f_{\pst,\sigma_i}\right)
  \le \frac{1}{8c^2}W_2^2(\pihat,\pst)
  \le \frac{M}{4c^2}W_1(\pihat,\pst)
  \pto 0.
\]
The second inequality uses the common support $[-M,M]$, on which
$|a-b|^2\le 2M|a-b|$.

Now we move on to the proof of the regret bound.
Recall the definition of $h_\pi(x;\sigma)$ from
\eqref{eq:posterior_mean} and note that by
\eqref{eq:lipschitzmean} the map $x\mapsto h_\pi(x;\sigma)$ is
$(M^2/c^2)$-Lipschitz uniformly over $\pi\in\mathcal{P}([-M,M])$
and $\sigma\in[c,C]$, with $c,C$ as in \cref{asn:appopt}.
Next recall the definition of $\mathrm{Regret}_n$ from \eqref{eq:regret-z}.
Conditioned on $\theta_i$, we couple $(X_i, Z_i)$ using the optimal
$W_2$-distance coupling.  Using the above observations, we get that
\begin{align}\label{eq:regretdecomp}
  &\;\;\;\;\;|\mathrm{Regret}_n| \nonumber
  \\ &=
  \left|
    \sqrt{\frac{1}{n}\sum_{i=1}^{n}\E \bigl(\theta_i - h_{\pihat}(X_i;\hat{\sigma}_i)\bigr)^2}
    \;-\;
    \sqrt{\frac{1}{n}\sum_{i=1}^{n}\E \bigl(\theta_i - h_{\pst}(Z_i;\sigma_i)\bigr)^2}
  \right| \nonumber \\
  &\leq
    \sqrt{\frac{1}{n}\sum_{i=1}^{n}\E \bigl(h_{\pihat}(X_i;\hat{\sigma}_i) - h_{\pihat}(X_i;\sigma_i)\bigr)^2} + \sqrt{\frac{1}{n}\sum_{i=1}^{n}\E \bigl(h_{\pihat}(X_i;\sigma_i) - h_{\pst}(X_i;\sigma_i)\bigr)^2}
    \nonumber \\ &\qquad +
\frac{M^2}{c^2}\sqrt{\frac{1}{n}\sum_{i=1}^{n}\E\,W_2^2\!\bigl(\mathrm{Law}(X_i\mid\theta_i),\mathrm{Law}(Z_i\mid\theta_i)\bigr)}.
\end{align}
As $\mathcal{W}_n\to 0$, the third term converges to $0$ as $n\to\infty$.
Let us now bound the first term. Note that by \cref{asn:taildef}, we get:
\begin{align*}
  &\limsup_{T\to\infty}\,\limsup_{n\to\infty}
    \frac{1}{n}\sum_{i=1}^{n}
    \E\!\Bigl[
      \bigl(h_{\pihat}(X_i;\hat{\sigma}_i) - h_{\pihat}(X_i;\sigma_i)\bigr)^2
      \mathbf{1}(|X_i|>T)
    \Bigr]\\
  &\leq
    4M^2\,
    \limsup_{T\to\infty}\,\limsup_{n\to\infty}
    \frac{1}{n}\sum_{i=1}^{n}
    \PP(|X_i|>T)
  \;=\; 0.
\end{align*}
Observe that for fixed $T > 0$, there exists a constant $c_T > 0$ such that
\begin{align}\label{eq:denlbd}
  \inf_{|x|\leq T}\;
  \inf_{\pi\in\mathcal{P}([-M,M])}\;
  \inf_{\sigma\in[c,C]}\;
  \int\phi(x-\theta;\,\sigma)\,d\pi(\theta)
  \;\geq\; c_T.
\end{align}
By using the above observation, we get that
\begin{align*}
    \sup_{|x|\le T} |h_{\pihat}(x;\hat{\sigma}_i)-h_{\pihat}(x;\sigma_i)| &\le \frac{1}{c_T}\sup_{|x|\le T}\bigg|\int \theta \left(\phi\left(x-\theta;\hat{\sigma}_i\right)-\phi\left(x-\theta;\sigma_i\right)\right)\,d\pihat(\theta)\bigg| \\ & \qquad +\frac{M}{c_T^2 \sqrt{2\pi}\sigma_i}\sup_{|x|\le T}\bigg|\int  \left(\phi\left(x-\theta;\hat{\sigma}_i\right)-\phi\left(x-\theta;\sigma_i\right)\right)\,d\pihat(\theta)\bigg|
\end{align*}
By using \eqref{eq:denlbd}, we then have:
$$\sup_{|x|\le T} |h_{\pihat}(x;\hat{\sigma}_i)-h_{\pihat}(x;\sigma_i)| \lesssim |\hat{\sigma}_i-\sigma_i|$$
where the implied constant depends on $T$. As a result, the following conclusion holds:
\begin{align*}
    \frac{1}{n}\sum_{i=1}^{n}\E \bigl(h_{\pihat}(X_i;\hat{\sigma}_i) - h_{\pihat}(X_i;\sigma_i)\bigr)^2\mathbf{1}(|X_i|\le T) \lesssim \E\left[\frac{1}{n}\sum_{i=1}^n (\hat{\sigma}_i-\sigma_i)^2\right] \lesssim \E\left[\frac{1}{n}\sum_{i=1}^n |\hat{\sigma}_i-\sigma_i|\right]\to 0,
\end{align*}
where the conclusion follows by using \cref{asn:appopt}.

\noindent Next we bound the second term in \eqref{eq:regretdecomp}. Note that by \cref{asn:taildef}, we once again get:
\begin{align*}
  &\limsup_{T\to\infty}\,\limsup_{n\to\infty}
    \frac{1}{n}\sum_{i=1}^{n}
    \E\!\Bigl[
      \bigl(h_{\pihat}(X_i;\sigma_i) - h_{\pst}(X_i;\sigma_i)\bigr)^2
      \mathbf{1}(|X_i|>T)
    \Bigr]\\
  &\leq
    4M^2\,
    \limsup_{T\to\infty}\,\limsup_{n\to\infty}
    \frac{1}{n}\sum_{i=1}^{n}
    \PP(|X_i|>T)
  \;=\; 0.
\end{align*}

\noindent Moreover, both maps
\[
  \theta \mapsto \theta\,\phi(x-\theta;\sigma)
  \qquad\text{and}\qquad
  \theta \mapsto \phi(x-\theta;\sigma)
\]
are uniformly Lipschitz over $|\theta|\leq M$, $\sigma\in[c,C]$, and all
$x\in\R$.
This implies that for a given $T > 0$, there exists a constant $C_T > 0$ such that
\[
  \frac{1}{n}\sum_{i=1}^{n}
  \E\!\Bigl[
    \bigl(h_{\pihat}(X_i;\sigma_i) - h_{\pst}(X_i;\sigma_i)\bigr)^2
    \mathbf{1}(|X_i|\leq T)
  \Bigr]
 \leq
  C_T\,\E\,W_1^2(\pihat,\pst)
  \to 0
\]
Combining the above observation with \eqref{eq:denlbd}, we find that the second term in \eqref{eq:regretdecomp} converges to $0$. This completes the proof.
\end{proof}

\subsection{Proofs of auxiliary lemmas}
\begin{proof}[Proof of \cref{lem:curvaturebd}]
    Assume without loss of generality that $\operatorname{KL}(f_{\pi_1,\sigma_i}\,\Vert\,f_{\pi_2,\sigma_i})\le 1/2$.
    By \cite[Theorem 2, part 2]{Nguyen2013}, there exists a universal constant $C$ (depending on $M$) such that
    $$W_1(\pi_1,\pi_2)\le C\big(-\log{\tv(f_{\pi_1,\sigma_i},f_{\pi_2,\sigma_i})}\big)^{-\frac{1}{2}},$$
    for every $1\le i\le n$.
    By \citet[Theorem 3, part (i)]{goldfeld2020gaussian}, we have $d_{\eta}(\pi_1,\pi_2)\le W_1(\pi_1,\pi_2)$ for all $\eta>0$. Further by Pinsker's inequality, we have $\tv(f_{\pi_1,\sigma_i},f_{\pi_2,\sigma_i})\le \sqrt{2\operatorname{KL}(f_{\pi_1,\sigma_i}\,\Vert\,f_{\pi_2,\sigma_i})}$. Combining the above observations with the fact that $\operatorname{KL}(f_{\pi_1,\sigma_i}\,\Vert\,f_{\pi_2,\sigma_i})\le 1/2$, we get
    $$\delta \le d_{\eta}(\pi_1,\pi_2)\le C\left(-\frac{1}{2}\log{(2\operatorname{KL}(f_{\pi_1,\sigma_i}\,\Vert\,f_{\pi_2,\sigma_i}))}\right)^{-1/2}.$$
    This completes the proof.
\end{proof}

\begin{proof}[Proof of \cref{lem:ptwiseconc}]
    Note that the function
    $$g_i(x):=\log\frac{f_{\pi,\sigma_i}(x)}{f_{\pst,\sigma_i}(x)}$$
    satisfies $\lVert g_i''\rVert_{\infty}\le 2M^2/\sigma_i^4$ and $\lVert g_i'\rVert_{\infty}\le 2M/\sigma_i^2$ as both $\pi$ and $\pst$ are elements of $\cP([-M,M])$. Therefore by \cref{asn:taildef}, we have
    \begin{align*}
        \frac{1}{n}\sum_{i=1}^n g_i(X_i) - \frac{1}{n}\sum_{i=1}^n \E g_i(X_i)=\frac{1}{n}\sum_{i=1}^n \log\frac{f_{\pi,\sigma_i}(X_i)}{f_{\pst,\sigma_i}(X_i)} - \frac{1}{n}\sum_{i=1}^n \E \log\frac{f_{\pi,\sigma_i}(X_i)}{f_{\pst,\sigma_i}(X_i)} \overset{\PP}{\to} 0.
    \end{align*}
    Next we sample $Z_i\sim f_{\pst,\sigma_i}$ and consider the optimal $W_1$-couplings between $\mu_i$ and $f_{\pst,\sigma_i}$, to get:
     \begin{align*}
        \frac{1}{n}\sum_{i=1}^n \E g_i(X_i) - \frac{1}{n}\sum_{i=1}^n \E g_i(Z_i) \le \frac{1}{n}\sum_{i=1}^n \big|\E g_i(X_i)-\E g_i(Z_i)\big|\lesssim \frac{1}{n}\sum_{i=1}^n W_1(\mu_i,f_{\pst,\sigma_i})\to 0.
    \end{align*}
    Next we observe that
    \begin{align*}
        \frac{1}{n}\sum_{i=1}^n \EE g_i(Z_i) = \EE\left[\frac{1}{n}\sum_{i=1}^n \log \frac{f_{\pi,\sigma_i}(Z_i)}{f_{\pst,\sigma_i}(Z_i)}\right] = -\frac{1}{n}\sum_{i=1}^n \operatorname{KL}(f_{\pst,\sigma_i}\,\Vert\,f_{\pi,\sigma_i}) \le -\tau(\delta)\equiv -\tau,
    \end{align*}
    where the last inequality uses \cref{lem:curvaturebd} for any $\pi$ such that $d_{\eta}(\pi,\pst)\ge \delta$. Combining the above observations, we get:
    \begin{align*}
        &\;\;\;\;\PP\left(\frac{1}{n}\sum_{i=1}^n \log\frac{f_{\pi,\sigma_i}(X_i)}{f_{\pst,\sigma_i}(X_i)}\ge -\tau/3\right) \\ &\le \PP\left(\bigg|\frac{1}{n}\sum_{i=1}^n g_i(X_i) - \frac{1}{n}\sum_{i=1}^n \E g_i(X_i)\bigg| + \frac{1}{n}\sum_{i=1}^n \big|\E g_i(X_i) - \E g_i(Z_i) \big| \ge 2\tau/3\right) \to 0.
    \end{align*}
    This completes the proof.
\end{proof}

\begin{proof}[Proof of \cref{lem:coverbd}]
Fix $\eta=c/2$. For $\sigma\in[c,C]$, observe that \[
  f_{\pi,\sigma}
  = \phi\!\left(\cdot;\sqrt{\sigma^2-\eta^2}\right)\star f_{\pi,\eta}
\]
For $U\sim f_{\pi_1,\eta}$, $V \sim f_{\pi_2, \eta}$, couple them optimally. Observe that \[
f_{\pi_1, \sigma}(x) = \E\bk{\phi(x-U; \sqrt{\sigma^2 - \eta^2})}\]
and similarly for $f_{\pi_2,\sigma}$. Therefore, by the mean-value theorem,
\begin{align*}
\sup_x|f_{\pi_1,\sigma}(x)-f_{\pi_2,\sigma}(x)| &\le \pr{\sup_z \abs[\big]{
\partial_z \phi(z ; \sqrt{\sigma^2 - \eta^2})}} \E|U-V| \\ &= \frac{d_\eta(\pi_1, \pi_2)}{\sqrt{2\pi e} (\sigma^2 - \eta^2)} \lesssim d_{\eta}(\pi_1,\pi_2).
\end{align*}
For $|x|\le T$, both densities are at least
$(\sqrt{2\pi}C)^{-1}\exp(-(T+M)^2/(2c^2))$. Thus, for each $T\ge1$, there is a constant $L_T<\infty$ such that
\[
  \sup_{\sigma\in[c,C],\,|x|\le T}
  \big|\log f_{\pi_1,\sigma}(x)-\log f_{\pi_2,\sigma}(x)\big|
  \le L_T d_\eta(\pi_1,\pi_2).
\]

Similarly, we have $
  \big|\log f_{\pi_1,\sigma}(x)-\log f_{\pi_2,\sigma}(x)\big|
  \le \frac{2M}{c^2}(|x|+M).
$
Thus, on the event $\sum_{i=1}^n X_i^2\le C_0n$, splitting at $|X_i|=T$ gives
\begin{align*}
  \avgn\big|\log f_{\pi_1,\sigma_i}(X_i)-\log f_{\pi_2,\sigma_i}(X_i)\big|&\le L_T d_\eta(\pi_1,\pi_2)
  +\frac{2M}{c^2}\avgn (|X_i|+M)\mathbf{1}(|X_i|>T)\\
  &\le L_T d_\eta(\pi_1,\pi_2)
  +\frac{2M}{c^2}\avgn X_i^2\left(\frac1T+\frac{M}{T^2}\right)\\
  &\le L_T d_\eta(\pi_1,\pi_2)
  +\frac{2MC_0}{c^2}\left(\frac1T+\frac{M}{T^2}\right). \numberthis \label{eq:loglikelihood_diff_bound}
\end{align*}

Finally, we have that,
for $B\ge M$ and $Z\sim\cN(0,1)$,
\[
  d_\eta(\pi_1,\pi_2)
  =\int_\R |F_{\pi_1,\eta}(x)-F_{\pi_2,\eta}(x)|\,dx
  \le 2B\tilde d_{\eta,B}(\pi_1,\pi_2)
  +2\E\big[(M+\eta|Z|-B)_+\big].
\]
Choose $T$ so that the second term in \eqref{eq:loglikelihood_diff_bound} is at most
$\iota/2$, and then choose $B$ so that
$2L_T\E[(M+\eta|Z|-B)_+]\le\iota/2$.
The conclusion follows with $L=2BL_T$.
\end{proof}

\begin{proof}[Proof of \cref{lem:truncate}]
    By \cref{asn:taildef}, it follows that there exists $\tilde{C}_0>0$ large enough such that
    $$\frac{1}{n}\sum_{i=1}^n \E X_i^2 \le \tilde{C}_0$$
    for all $n$ large enough. Define $C_0:=3\tilde{C}_0$. By \cref{asn:taildef}, we then have
    \begin{align*}
        \limsup_{n\to\infty}\, \PP\left(\frac{1}{n}\sum_{i=1}^n X_i^2 \ge C_0\right)\le \limsup_{n\to\infty}\, \PP\left(\frac{1}{n}\sum_{i=1}^n X_i^2 - \frac{1}{n}\sum_{i=1}^n \E X_i^2 \ge \tilde{C}_0\right) = 0.
    \end{align*}
    This completes the proof.
\end{proof}

\begin{proof}[Proof of \cref{lem:unifbd}]
    Given $\delta>0$, set $\iota:=\tau/10$, where $\tau\equiv \tau(\delta)$ (see \eqref{eq:curvaturebd}). By \cref{lem:truncate}, there exists $C_0>0$ such that $\sum_{i=1}^n X_i^2\le C_0 n$ outside an event of probability converging to $0$ as $n\to\infty$. With the aforementioned choice of $C_0$ and $\iota$, we get appropriate $\eta,B,L$ from \cref{lem:coverbd}.

    \noindent Next fix any $\upsilon\ge \eta$. Consider the set $\{\pi\in\mathcal{P}([-M,M]):\ d_{\upsilon}(\pi,\pst)\ge \delta\}$ and choose a $(\tau/10L)$-covering set, say $\mathcal{S}(\eta,B,L,\tau)\subseteq \{\pi\in\mathcal{P}([-M,M]):\ d_{\upsilon}(\pi,\pst)\ge \delta\}$,  with respect to the pseudo-metric $\tilde{d}_{\eta,B}(\cdot,\cdot)$. By a standard packing argument, this set has some finite cardinality $\cN(\eta,B,L,\tau)$ as all the cumulative distribution functions $G_{\mu,\eta}$ for $\mu\in\mathcal{P}([-M,M])$ have derivatives uniformly bounded above and below over the interval $[-B,B]$.

    \noindent Now given any $\pi\in\mathcal{P}([-M,M])$ such that $d_{\upsilon}(\pi,\pst)\ge\delta$, there exists $\ell(\pi)\in\mathcal{S}(\eta,B,L,\tau)$ such that $\tilde{d}_{\eta,B}(\pi,\ell(\pi))\le (\tau/10L)$. Therefore, given any $\pi$ such that $d_{\upsilon}(\pi,\pst)\ge \delta$, by \cref{lem:coverbd}, we  have

    $$\frac{1}{n}\sum_{i=1}^n \big(\log{f_{\pi,\sigma_i}(X_i)}-\log{f_{\pst,\sigma_i}(X_i)}\big)\le \frac{\tau}{5}+\frac{1}{n}\sum_{i=1}^n \big(\log{P_{\ell(\pi),i}(X_i)}-\log{f_{\pst,\sigma_i}(X_i)}\big),$$
    on the set $\sum_{i=1}^n X_i^2\le C_0 n$. Therefore, by a direct union bound, we get:
    \begin{align*}
    &\;\;\;\;\limsup_{n\to\infty}\PP\left(\sup_{\pi\in\mathcal{P}([-M,M]): d_{\eta}(\pi,\pst)\ge \delta}\frac{1}{n}\sum_{i=1}^n \big(\log{f_{\pi,\sigma_i}(X_i)}-\log{f_{\pst,\sigma_i}(X_i)}\big)\ge -\frac{2\tau}{15}\right) \\ &\le \sum_{\pi\in\mathcal{S}(\eta,B,L,\tau)} \limsup_{n\to\infty} \PP\left(\frac{1}{n}\sum_{i=1}^n \big(\log{f_{\pi,\sigma_i}(X_i)}-\log{f_{\pst,\sigma_i}(X_i)}\big)\ge -\tau/3\right) = 0.
    \end{align*}
    The last inequality follows from \cref{lem:ptwiseconc}. This completes the proof.
\end{proof}

\section{Examples of heuristic appeals to the CLT in empirical Bayes}
\label{sec:handwave}

\subsection{Papers in economics}
\begin{itemize}
    \item  \citet{gu2018oracle}: ``In this case, the statistic $S_i$ follows a normal mixture distribution only asymptotically under the null and alternative (cf. \citet{cao2011simultaneous}), and therefore the mixture normal set-up considered in the methodology section is slightly misspecified. The impact of misspecification on the properties of deconvolution estimators or multiple testing procedures has not yet been studied in the literature, as far as we know. This appendix takes the first step in examining the impact using Monte Carlo simulations. We experiment with both normal and non-normal errors for different DGPs. Our results suggest that the asymptotic approximation of the base density does not affect the performance of the proposed testing method. A more rigorous theoretical investigation is left for future research.''
    \item \citet{azevedo2019empirical}: ``For each idea, the firm performs an experiment, or A/B test, with $n$ users. The experiment yields an estimated quality $\hat{\Delta}_i$. The estimated quality is an estimated treatment effect. We assume that estimated quality is normally distributed with mean $\Delta_i$ and variance $\sigma^2/n$. This is reasonable because of randomization and because of the large samples used by internet companies.''
    \item \citet{azevedo2020testing}: ``Third, experimental errors are normally distributed. This is a reasonable assumption in our main application because the typical estimator for the unknown quality is a difference between sample means with independently and identically distributed data, and treatment/control groups are in the millions.''
    \item \citet{deng2021postselection}:  ``We consider a standard A/B test with a treatment and a control group of sample sizes $N_T$ and $N_C$ and metric values $Y_T$ and $Y_C$, respectively. A metric could be in the form of an average across i.i.d. samples, but is not limited to it. The central limit theorem entails that when sample sizes are large enough, the estimated treatment effect $\Delta = Y_T-Y_C$ approximately follows a Gaussian distribution
with mean $\mu$ and variance $\sigma_T^2/N_T + \sigma_C^2/N_C$.''
    \item \citet{kline2022systemic}: ``The normality assumption for $z_f$ can be justified by an asymptotic approximation with a growing number of jobs sampled for each firm.''
    \item \citet{gu2023invidious}: ``The classical variance stabilizing transformation for the Poisson brings us back to the Gaussian model.''
    \item \citet{abadie2023estimating}: ``The Gaussianity assumption on the distribution of $\hat{\tau} \mid \tau$ is motivated by approximate Gaussianity of the large sample distributions of many commonly used estimators of treatment effects.''
    \item \citet{gu2022ranking}: ``Another approach to regularization is to treat the unconstrained logistic estimates, \smash{$\hat{\theta}$}, as approximately independent draws from a Gaussian sequence model. When the problem design is unbalanced so the number of matched pairs are not equal, the MLE point estimates will have different precision and off-diagonal elements of their covariance matrix are also more heterogeneous. Initially, we will ignore the latter aspect and treat the estimated maximum likelihood rating parameters as a sample from a Gaussian sequence model with heterogeneous scale parameters.''
    \item \citet{lee2024improving}: ``When estimating average impacts, the central limit theorem applied to the individuals within the site provides the normality or something close to it.''
    \item \citet{wernerfelt2025estimating}: ``For the lower-level distributions, we make a normality assumption. This is simply an appeal to the central limit theorem: We impose the assumption that our experiments are sufficiently large that the treatment effects are normally distributed around the true unobserved value and our variance estimates have converged.''
    \item \citet{yamin2025poverty}: ``Motivated by the central
limit theorem, I assume normality of the estimates and consider them a noisy but unbiased signal of $\mu_i$'', ``The normal distribution assumption is reasonable as long as the population from
which the samples are drawn has a finite variance and the effective sample size for each
administrative region is relatively large.''
\item \citet{moon2026optimal}: ``If program costs are observed without statistical
uncertainty, the Gaussian distribution approximation motivated by the central limit theorem is reasonable.'' [This quote is from arXiv v4. ArXiv v5 instead mentions results from this work.]
\end{itemize}

\subsection{Papers in statistics}

\begin{itemize}
    \item \citet{deb2022twocomponent}: ``Letting $\Delta_i :=  (\mu_i(2)-\mu_i(1))/\sigma_i$ denote the effect size for the $i$'th gene, we can assume that $t_i$'s are approximately normal.''
    \item \citet{imbens2022comment}: ``the approximate normality of the point estimates is arguably reasonable given the sample sizes involved''
    \item \citet{ignatiadis2023empiricala}: ``We could appeal to the central limit theorem to justify knowing that the likelihood is Normal.''

    \item \citet{hoff2025selective}: ``A simple but widely applicable model for studying and developing multipopulation inference procedures is the multiple normal means model, where scalar observations $Z_1, \ldots, Z_{p+1}$ are independently sampled from $p + 1$ potentially different normal populations, so that $Z_j \sim \cN(\mu_j, \psi_j^2)$ independently for $j = 1, \ldots, p+1$, with $\mu_1, \ldots, \mu_{p+1}$ being unknown and $\psi_1^2, \ldots, \psi_{p+1}^2$ (approximately) known. This scenario might arise if the elements of $\mathbf{Z} = (Z_1, \ldots, Z_{p+1})$ are sample averages from $p + 1$ populations with means equal to the corresponding elements of $\boldsymbol{\mu} = (\mu_1, \ldots, \mu_{p+1})$, and a common population variance $\psi^2$ which could be precisely estimated by pooling data across the groups.''
    \item \citet{ling2026empirical}: ``Motivated by the central limit theorem, we model the estimators as exactly normal with known second moments; our theory does not propagate the error from the CLT approximation.'' [This quote is from arXiv v1 (arXiv:2605.08069v1). ArXiv v2 and the camera-ready version instead use results from an earlier preprint of this work.]
    \item \citet{chang2026personalized}: ``Under standard regularity conditions, local empirical processes yield asymptotic normality [\ldots]
We therefore treat $\hat\theta_k$ as approximately [normal], a noisy observation of the client-specific parameter $\theta_k$.''
\end{itemize}

\section{Details on simulations and additional empirical results}
Here we provide some additional details on each simulation and empirical analysis, as well as some additional results.

\subsection{Simulation with Pearson distributions}

We note the following caveats:
\begin{itemize}
    \item Throughout these simulations $\hat\sigma^2$ is the usual sample variance, except clipped to $\ge 10^{-6}$. This binds most often for the Pearson distributions with skewness \(\pm 1\) and excess kurtosis \(-0.75\), affecting about \(14.8\%\) of units at \(J=2\), \(2.3\%\) at \(J=4\), and \(0.064\%\) at \(J=8\).
\end{itemize}

In \cref{fig:pure_loglog} we show a log--log plot for \cref{fig:pure_simulation}, where we only plot NPMLE regret series that are strictly positive for all values of $J$.

\begin{figure}[tbh]
    \centering
    \includegraphics[width=\linewidth]{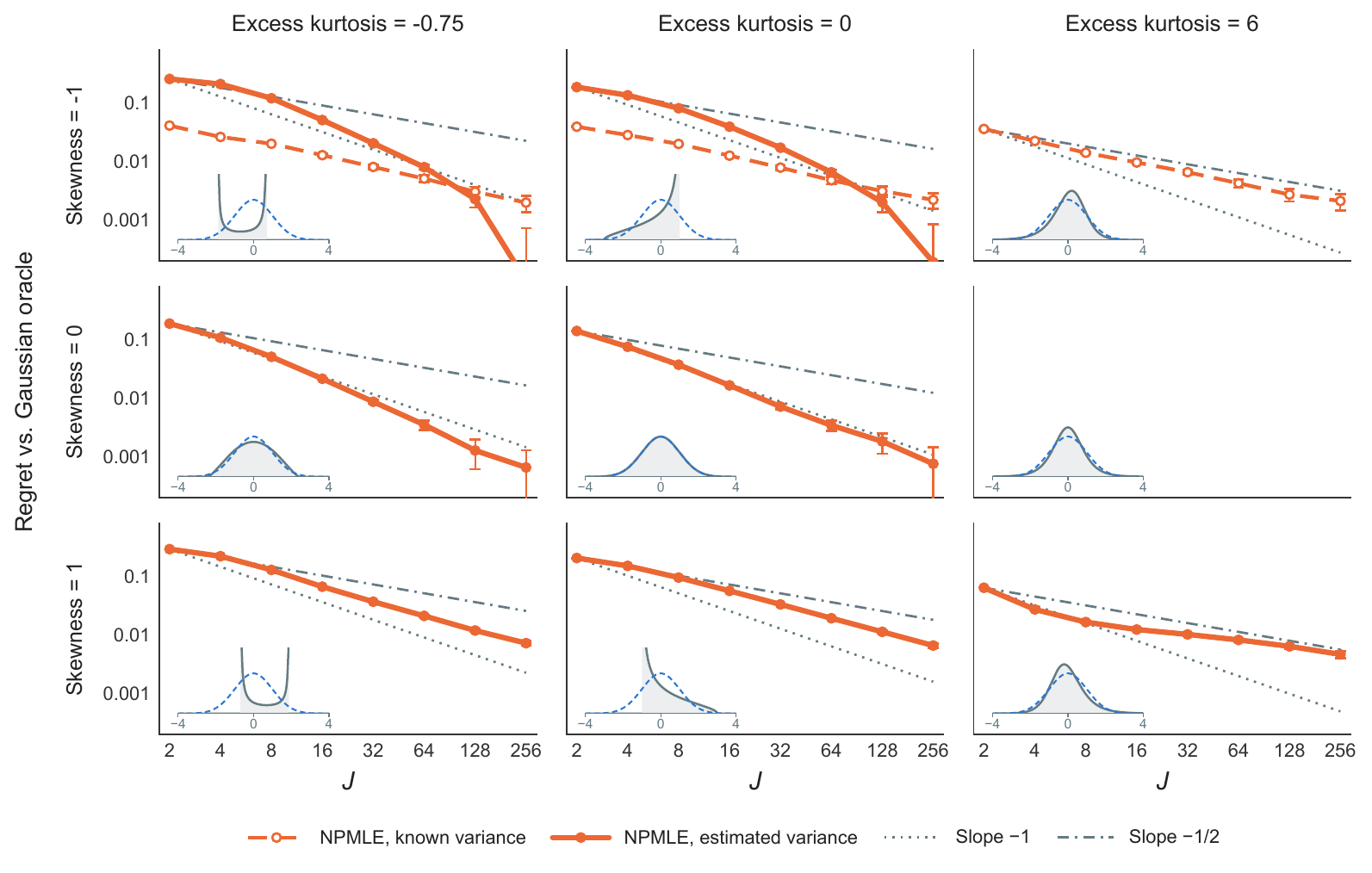}
    \caption{Log--log plot for \cref{fig:pure_simulation}}
    \label{fig:pure_loglog}
\end{figure}

\subsection{Simulation calibrated to \citet{kline2022systemic}}

We note:
\begin{itemize}
    \item Likewise, $\hat\sigma_i^2$ is the raw variance truncated at $10^{-6}$. This binds  for 47.7\% of firms at \(r=1/32\), 30.4\% at \(r=1/16\), and 1.6\% at \(r=1\).

\end{itemize}

We consider a simulation where we set $\theta_i := \E_n X_i^\data + c(X_i^\data - \E_n
X_i^\data)$, for a scaling factor $c$ such that $\var_n(\theta_i)$ equals the estimated
signal variance from $X_i^{\data}$. We then draw $X_i$ following \eqref{eq:local} without
redrawing $\theta_i$. We call this the fixed-effect setting and study the performance of
various EB procedures. In contrast to the random-effect setting above, the fixed-effect setting does not assume precision independence. Since we no longer assume $\theta_i \iid \pistar$, we can no longer
define regret. Instead, we measure performance using the RMSE of estimating $\theta_i$ using $\hat\theta_i = h_{\hat\pi}(X_i; \sigma_i)$, normalized by the RMSE of the naive estimator $\hat\theta_i = X_i$.

\begin{figure}[tbh]
	\centering
	\includegraphics[width=\linewidth]{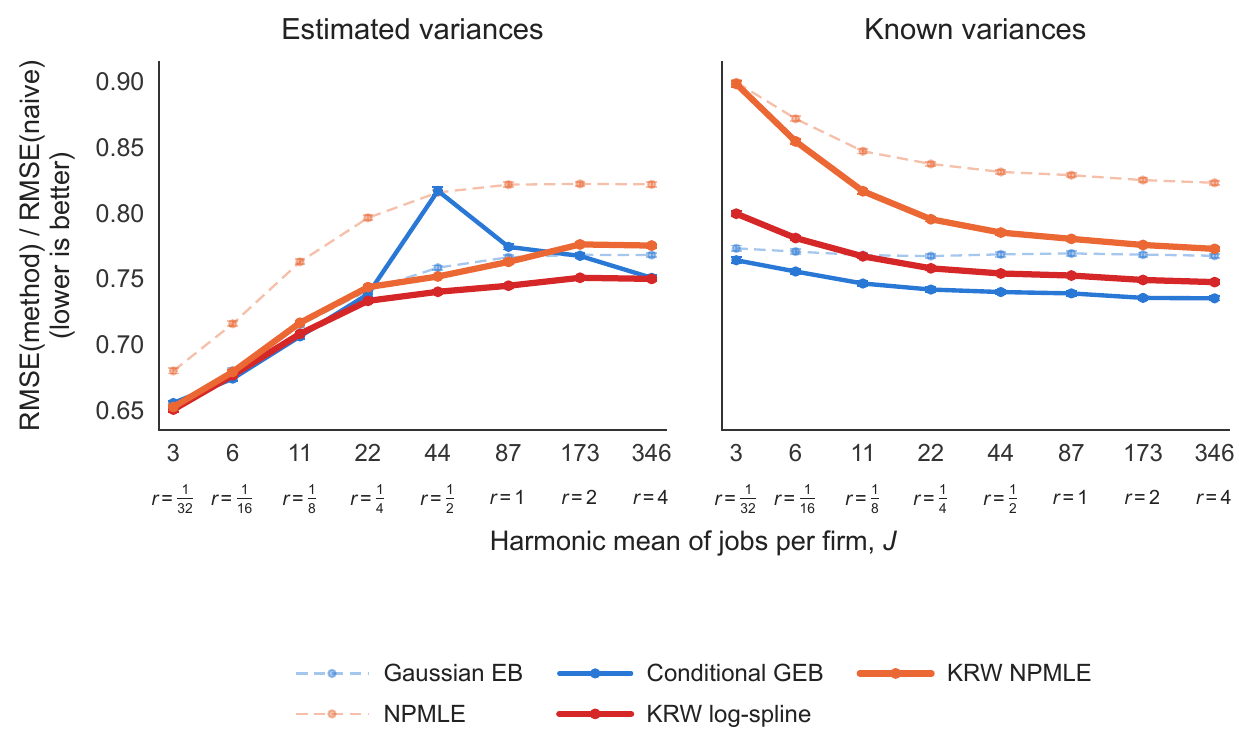}
	\caption{Fixed-effect calibrated simulation of \citet{kline2022systemic}}
	\label{fig:kline_fe}
\end{figure}

The results are shown in \cref{fig:kline_fe}, where we consider EB methods with variants that are robust to precision dependence \citep{chen2024empirical}:
\begin{itemize}[wide]
    \item \textcolor[HTML]{2A78D6}{\emph{Conditional Gaussian EB}}, which models $\theta_i \mid \hat\sigma_i \sim \Norm(m_0(\hat\sigma_i), s_0^2(\hat\sigma_i))$, with $m_0$ linear in $\log \hat \sigma_i$ and $s_0^2$ exponential-linear in $\log \hat\sigma_i$, following \citet{walters2024empirical}.
    \item \citet[\emph{KRW}]{kline2024discrimination}, a model where $\theta_i / \sigma_i^\beta \iid G_0$ for some unknown $\beta$ to be estimated.\footnote{\citet{kline2024discrimination} apply this model to $\theta$ being the difference in \emph{log} callback rates instead. This model is also present in \citet{stephens2017false}.} We consider estimating $G_0$ via a \textcolor[HTML]{D62728}{spline sieve} for the log density\footnote{We approximate the log density of \(G_0\) by a natural cubic spline with five degrees of freedom and four equally spaced interior knots. We fit the spline by penalized maximum likelihood, with a penalty proportional to the \(\ell_2\) norm of the spline coefficients. Following \citet{kline2024discrimination}, we choose the penalty strength to bring the mean and standard deviation of the fitted \(G_0\) close to their GMM estimates.}  \citep{kline2024discrimination,efron2016empirical}, as well as estimating $G_0$ via the \textcolor[HTML]{EB6834}{NPMLE}.
\end{itemize}
Qualitatively, we find behavior similar to that in \cref{fig:kline_re}. Increasing $J$ does not seem to benefit performance much: In fact, as in \cref{fig:kline_re} and \cref{item:kurtosis},  noisy $\hat\sigma_i^2$ with small $J$ appears to \emph{benefit} the performance of EB methods. Across EB methods, the model of \citet{kline2024discrimination} modeling $G_0$ as a spline sieve appears to be most competitive.

\subsection{Results on the VWH dataset}
\label{sub:VWH_details}

We discuss the following:
\begin{itemize}[wide]
    \item (HC2 variance estimation in \cref{fig:veneto})
    Specifically, for mover $w$ and $k=1,\ldots,r$, let
    \[
     Y_{wk}^{(r)}=\theta_{i_2(w)}-\theta_{i_1(w)}+\tau\sqrt r\,\eta_{wk},
    \]
    where we draw the shocks $\eta_{wk}$ independently from $P_\epsilon^{0,h}$. The vector $Y^{(r)}$ is then an average of $ Y_{wk}^{(r)}$ over $k$. The HC2 estimate of the sampling variance of $X_i^{(r)}$ is therefore \[
     \hat\sigma_{i,\mathrm{HC2}}^2
     =\sum_w\frac{a_{iw}^2}{r^2}\sum_{k=1}^r
     \frac{\bigl[Y_{wk}^{(r)}-(X^{(r)}_{i_2(w)}-X^{(r)}_{i_1(w)})\bigr]^2}{1-h_w/r},
    \]
    where $h_w$ is mover $w$'s leverage in the regression of $Y^{(r)}$ on the destination-minus-origin firm indicators.

    This sum over $w$ contains undefined summands at $r=1$ for those with $h_w=1$ (there are 30,714 such movers). For such a mover $w$ we replace its contribution to $\hat\sigma_{i,\mathrm{HC2}}^2$ with
    \[
    a_{iw}^{2}
    \frac{\sum_{v\ \mathrm{a\ mover}}
    \left[Y_v^{(1)}-\left(X_{i_2(v)}^{(1)}-X_{i_1(v)}^{(1)}\right)\right]^2}
    {\#\{\text{movers}\}-(n-1)}.
    \]

    \item (Details on branching \citet{kline2025branching})
    Following \citet{kline2025branching}, we can consider a graph on firms induced by the bipartite graph of firms and workers, whereby two firms are connected if any worker moves between them (we can thus associate each edge in the firm graph with a set of workers). Given a set of firms $S$ and its associated graph, we seek---or ``pack''---two spanning trees that have disjoint edge sets.

    \citet[Sections 6.2--6.4]{kline2025branching} packs spanning trees with the
    Roskind--Tarjan algorithm. We instead use a Python implementation of Edmonds' (1965) matroid partition algorithm to find two edge-disjoint spanning trees.\footnote{This is primarily motivated by not wanting to introduce SageMath as a dependency.} To generate different packings across repetitions, we randomly relabel the firms before each run and process the edges in lexicographic order of their relabeled endpoints.

    Running this algorithm extracts two edge-disjoint spanning trees on the graph associated with $S$. There are some remaining edges not present in either tree. We consider two ways of assigning these remaining edges.
    \begin{itemize}

    \item First, following
       \citet[Section 5]{kline2025branching}, we put the first tree in branch $M_1$ and
       all other edges, including the second tree, in branch $M_2$.

       \item Second, for \emph{balanced-pair}
       branching, we put one tree in each branch, randomly permute the remaining edges,
       and divide them evenly between the branches. This balances the number of edges in each branch.\footnote{We thank Patrick Kline for this suggestion.}
    \end{itemize}
    This partition of edges then produces a partition of movers.

    Given such a partition of movers into branches, we can then fit $X_1^S, X_2^S$ by OLS on the respective branches. We repeat the split 100 times and produce $(X_3^S, X_4^S), \ldots, (X_{199}^S, X_{200}^S)$ analogously. To provide the branching-based denoised estimate of firm effects following \citet[Section 7.3]{kline2025branching}, we stack these into a dataset \[
    \begin{bmatrix}
      X_1^S & X_2^S \\
      X_3^S & X_4^S \\
      \vdots \\
      X_{199}^S & X_{200}^S
    \end{bmatrix}.
    \]
    We then regress the first column nonparametrically on the second, using 5-fold cross-fitting by firm. For each firm, this regression produces 100 estimates of that firm's effect; we average those to obtain \(\hat\theta_{1i}^S\). Reversing the roles of the two columns gives another set of averaged fitted values $\hat\theta_2^S$. The branching-based denoised estimate of firm effects is then $\frac{1}{2}(\hat\theta_1^S + \hat\theta_2^S)$. These regressions may optionally include controls for precision of the firm effect estimates.

We apply this procedure to our calibrated simulation on the VWH data and obtain \cref{fig:branching}.
    \begin{figure}[tbh]
	\centering
	\includegraphics[width=\linewidth]{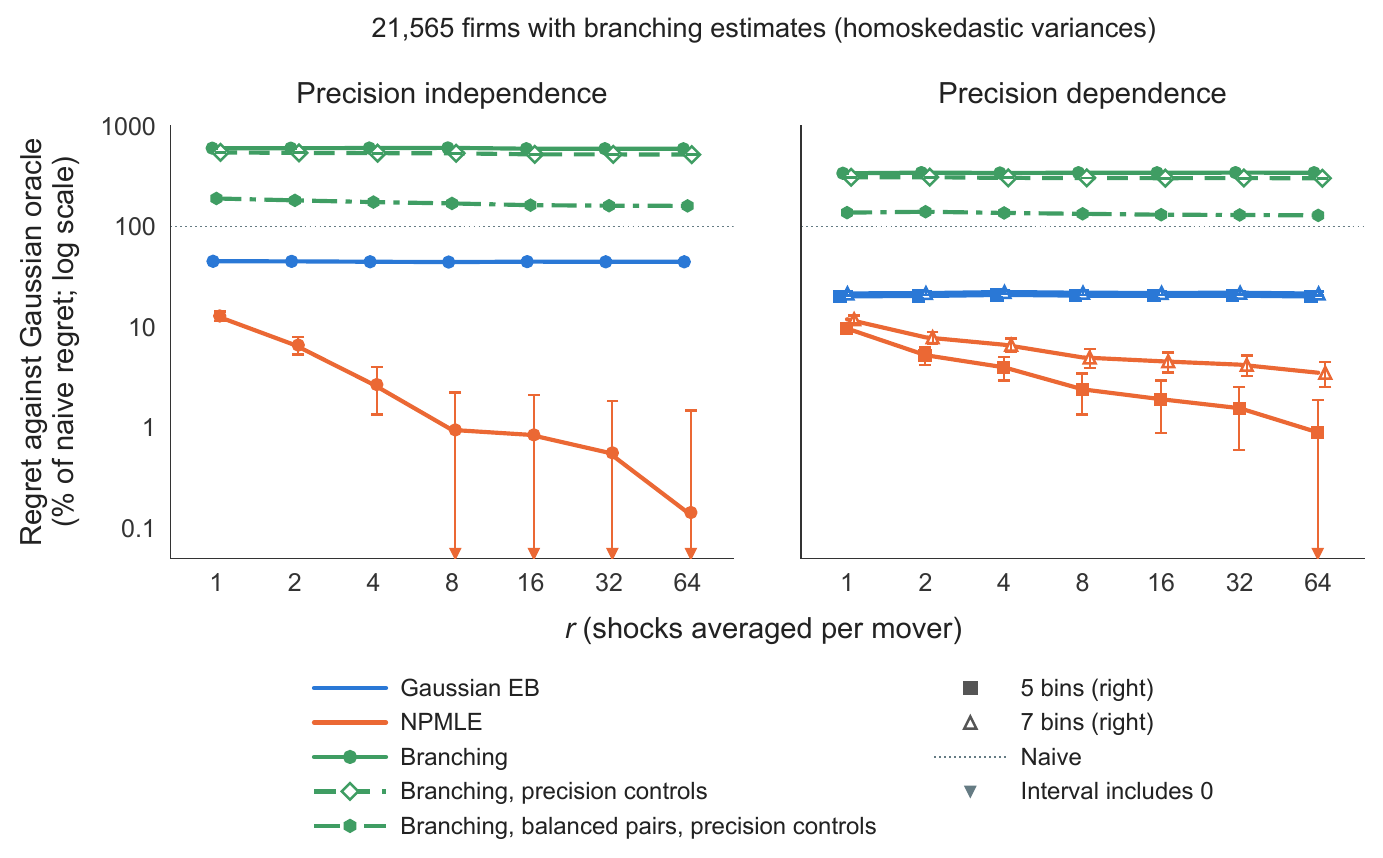}

	\begin{proof}[Notes]

    All methods are \emph{evaluated} on the same 21,565 firms selected by branching. The naive
    and EB estimates use the full set of movers, with estimated homoskedastic variances. Error bars show $\pm 2$ Monte Carlo standard errors over 100
    replications (downward triangles mark intervals that include zero).

    Since $S$ is a subset of all firms, when we set the sum of firm effects to zero, the sum of firm effects in $S$ might not be zero. We thus shift the branching-produced estimates so that their average over $S$ matches $|S|^{-1} \sum_{j\in S} X_j^{(r)}$, where $X_j^{(r)}$ is the OLS coefficient over the full sample.

	\end{proof}

	\caption{Performance of branching on a subset of firms and movers against methods on the full sample}

	\label{fig:branching}
    \end{figure}

    \item (Construction of the set $S$ and training sets for the out-of-sample prediction exercise)
    We apply the pruning criteria of \citet[Algorithm 1]{kline2025branching} with \(k=3\) to the induced firm graph and let \(S\) be the firms in the largest remaining component. We verify that the induced firm graph on \(S\) contains two edge-disjoint spanning trees and that enough edges lie outside those trees to hold out 20\% of all edges on \(S\). The graph contains 21,565 firms and 80,554 edges representing 114,734 movers.

    We fix $S$ across the five train-test splits in \cref{tab:veneto_sample_splitting}. For each
    split we first pack two edge-disjoint spanning trees on $S$ and reserve both for training. We draw
    16,110 edges uniformly without replacement from the remaining 37,426 edges not belonging to either of the
    two trees. The movers associated with these 16,110 edges then form $M_3$, and those on the remaining 64,444 edges form $M_1 \cup M_2$.
    Within each split, all methods use the same training set $M_1\cup M_2$ and
    the same time effect $\hat\gamma_{\mathrm{stayer}}$ estimated from stayers. For the branching-based procedures, we construct 100 new tree packings using only the training edges.
    \item \citet{cheng2025optimal} propose a two-way EB-URE estimator and apply it to estimate teacher value-added from matched student–teacher data. We apply their estimator to matched worker–firm data, with movers playing the role of students and firms playing the role of teachers. We set the prior mean for firm effects \(\theta_i\) to zero, as they do for teacher effects. Since we use no worker covariates in the prior, we specify a common prior mean \(\mu\) for worker effects \(\alpha_w\) (they use student demographics to parametrize the prior mean for student effects). Following their approach, we choose \(\mu\) and the other hyperparameters by minimizing the estimated mean squared error of the firm effects.
\end{itemize}

\stopcontents
\end{document}